\documentclass[11pt]{article}
\usepackage{amssymb,amsmath,epsfig}
\usepackage[colorlinks=false,breaklinks=true,linkcolor=blue]{hyperref}
\usepackage{graphics,psfrag,graphicx,color}
\usepackage{float,hhline}
\usepackage{multicol}
\usepackage{comment}
\usepackage[dvipsnames]{xcolor}
\usepackage[sort]{cite}

\numberwithin{equation}{section}
\allowdisplaybreaks

\graphicspath{{images/}}

\newcommand{\ds}{\displaystyle}

\newcommand{\bgamma}{{\boldsymbol\gamma}}

\newcommand{\bLambda}{{\boldsymbol\Lambda}}
\newcommand{\bbeta}{{\boldsymbol\eta}}
\newcommand{\bdelta}{{\boldsymbol\delta}}
\newcommand{\bomega}{{\boldsymbol\omega}}
\newcommand{\bsi}{{\boldsymbol\sigma}}

\newcommand{\bphi}{{\boldsymbol\phi}}

\newcommand{\bvarphi}{{\boldsymbol\varphi}}
\newcommand{\btheta}{{\boldsymbol\theta}}
\newcommand{\bpsi}{{\boldsymbol\psi}}

\newcommand{\btau}{{\boldsymbol\tau}}
\newcommand{\bzeta}{{\boldsymbol\zeta}}

\newcommand{\bchi}{{\boldsymbol\chi}}
\newcommand{\bxi}{{\boldsymbol\xi}}
\newcommand{\brho}{{\boldsymbol\rho}}
\newcommand{\ubsi}{\underline{\bsi}}
\newcommand{\ubtau}{\underline{\btau}}
\newcommand{\ubu}{\underline{\bu}}
\newcommand{\ubv}{\underline{\bv}}

\newcommand{\ubvarphi}{\ul{\boldsymbol\varphi}}
\newcommand{\ubpsi}{\ul{\boldsymbol\psi}}

\newcommand{\bv}{{\mathbf{v}}}
\newcommand{\bw}{{\mathbf{w}}}
\newcommand{\f}{\mathbf{f}}

\newcommand{\bp}{\mathbf{p}}
\newcommand{\bq}{\mathbf{q}}
\newcommand{\br}{\mathbf{r}}
\newcommand{\bu}{\mathbf{u}}

\newcommand{\bt}{{\mathbf{t}}}
\newcommand{\bn}{{\mathbf{n}}}
\newcommand{\be}{{\mathbf{e}}}

\newcommand{\0}{{\mathbf{0}}}
\def\bA{\mathbf{A}}

\def\bF{\mathbf{F}}
\def\bG{\mathbf{G}}
\def\bK{\mathbf{K}}
\def\bI{\mathbf{I}}

\def\bX{\mathbf{X}}
\def\bY{\mathbf{Y}}
\def\bZ{\mathbf{Z}}
\def\bV{\mathbf{V}}

\def\bM{\mathbf{M}}

\def\bP{\mathbf{P}}

\def\bS{\mathbf{S}}

\def\bx{\mathbf{x}}
\newcommand{\bL}{\mathbf{L}}
\newcommand\bH{\mathbf{H}}

\newcommand\bbM{\mathbb{M}}

\newcommand\bbP{\mathbb{P}}
\newcommand\bbQ{\mathbb{Q}}

\newcommand\bbH{\mathbb{H}}
\newcommand\bbW{\mathbb{W}}
\newcommand\bbX{\mathbb{X}}

\newcommand\bbL{\mathbb{L}}

\newcommand{\cA}{\mathcal{A}}
\newcommand{\cB}{\mathcal{B}}
\newcommand{\cC}{\mathcal{C}}

\newcommand{\cE}{\mathcal{E}}

\newcommand{\cN}{\mathcal{N}}
\newcommand{\cM}{\mathcal{M}}
\newcommand{\cT}{\mathcal{T}}

\newcommand{\cD}{\mathcal{D}}
\newcommand{\cO}{\mathcal{O}}

\def\P{\mathrm{P}}
\def\R{\mathrm{R}}
\def\Q{\mathrm{Q}}

\def\H{\mathrm{H}}
\def\L{\mathrm{L}}
\def\M{\mathrm{M}}
\def\V{\mathrm{V}}
\def\W{\mathrm{W}}
\def\X{\mathrm{X}}
\def\rd{\mathrm{d}}

\def\rD{\mathrm{D}}
\def\rN{\mathrm{N}}
\def\rP{\mathrm{P}}

\def\rp{\mathrm{p}}

\def\rs{\mathrm{s}}
\def\rt{\mathrm{t}}

\def\esssup{\mathrm{ess\,sup}}
\def\bBDM{\mathbf{BDM}}
\def\bbBDM{\mathbb{BDM}}
\def\BJS{\mathtt{BJS}}
\def\dc{\mathrm{dc}}

\def\bdiv{\mathbf{div}}

\def\tr{\mathrm{tr}}

\def\div{\mathrm{div}}
\def\dist{\mathrm{dist}\,}
\def\pil{\left<}
\def\pir{\right>}

\def\sk{\mathrm{sk}}
\def\skf{{\mathrm{sk},f}}
\def\skp{{\mathrm{sk},p}}

\def\qin{{\quad\hbox{in}\quad}}
\def\qon{{\quad\hbox{on}\quad}}
\def\qan{{\quad\hbox{and}\quad}}

\def\ov{\overline}
\def\ul{\underline}
\def\wt{\widetilde}
\def\wh{\widehat}

\newtheorem{thm}{Theorem}[section]
\newtheorem{rem}{Remark}[section]
\newtheorem{lem}[thm]{Lemma}
\newtheorem{cor}[thm]{Corollary}

\newenvironment{proof}{\noindent{\it Proof.}}{\hfill$\square$}

\numberwithin{equation}{section}
\numberwithin{figure}{section}
\numberwithin{table}{section}

\title{A multipoint stress-flux mixed finite element method for the Stokes-Biot model}

\author{{\sc Sergio Caucao}\thanks{Departamento de Matem\'atica y F\'isica Aplicadas, Universidad Cat\'olica de la Sant\'isima Concepci\'on, Casilla 297, Concepci\'on, Chile, email: {\tt scaucao@ucsc.cl}. Supported in part by ANID-Chile through the project PAI77190084 of the PAI Program: Convocatoria Nacional Subvenci\'on a la Instalaci\'on en la Academia (convocatoria 2019) and Department of Mathematics, University of Pittsburgh}
\quad
{\sc Tongtong Li}\thanks{Department of Mathematics, University of Pittsburgh, Pittsburgh, PA 15260, USA, email: {\tt tol24@pitt.edu, yotov@math.pitt.edu}. Supported in part by NSF grant DMS 1818775}	
\quad
{\sc Ivan Yotov$^\dag$}}

\date{\today}

\begin{document}

\maketitle

\begin{abstract}
\noindent In this paper we present and analyze a fully-mixed
formulation for the coupled problem arising in the interaction between
a free fluid and a flow in a poroelastic medium.  The flows are
governed by the Stokes and Biot equations, respectively, and the
transmission conditions are given by mass conservation, balance of
stresses, and the Beavers-Joseph-Saffman law.  We apply dual-mixed
formulations in both domains, where the symmetry of the Stokes and
poroelastic stress tensors is imposed by setting the vorticity and
structure rotation tensors as auxiliary unknowns.  In turn, since the
transmission conditions become essential, they are imposed weakly,
which is done by introducing the traces of the fluid velocity,
structure velocity, and the poroelastic media pressure on the
interface as the associated Lagrange multipliers.  The existence and
uniqueness of a solution are established for the continuous weak
formulation, as well as a semidiscrete continuous-in-time formulation
with non-matching grids, together with the corresponding stability
bounds.  In addition, we develop a new multipoint stress-flux mixed
finite element method by involving the vertex quadrature rule, which
allows for local elimination of the stresses, rotations, and Darcy
fluxes.  Well-posedness and error analysis with corresponding rates of
convergence for the fully-discrete scheme are complemented by several
numerical experiments.
\end{abstract}

%

\maketitle


\section{Introduction}

The interaction of a free fluid with a deformable porous medium,
referred to as fluid-poroelastic structure interaction (FPSI), is a
challenging multiphysics problem. It has applications to predicting
and controlling processes arising in gas and oil extraction from
naturally or hydraulically fractured reservoirs, modeling arterial flows, and designing
industrial filters, to name a few.  For this physical
phenomenon, the free fluid region can be modeled by the Stokes (or
Navier--Stokes) equations, while the flow through the deformable
porous medium is modeled by the Biot system of
poroelasticity.  In the latter, the volumetric
deformation of the elastic porous matrix is complemented with the
Darcy equation that describes the average velocity of the fluid in the
pores.  The two regions are coupled via dynamic and kinematic
interface conditions, including balance of forces, continuity of
normal velocity, and a no slip or slip with friction tangential
velocity condition. The model exhibits features of both coupled
Stokes-Darcy flows and fluid-structure interaction (FSI).

To the authors' knowledge, one of the first works in analyzing the
Stokes-Biot coupled problem is \cite{s2005}, where well-posedness for
the fully dynamic problem is established by developing an appropriate
variational formulation and using semigroup methods. One of the first
numerical studies is presented in \cite{bqq2009}, where monolithic and
iterative partitioned methods are developed for the solution of the
coupled system. A non-iterative operator splitting scheme with a
non-mixed Darcy formulation is developed in \cite{byz2015}. Finite
element methods for mixed Darcy formulations, where the continuity of
normal flux condition becomes essential, are considered in
\cite{byzz2015} using Nitsche's coupling and in \cite{akyz2018} using
a pressure Lagrange multiplier. More recently, a nonlinear quasi-static
Stokes--Biot model for non-Newtonian fluids is studied in \cite{aeny2019}.
The authors establish well-posedness of the weak formulation in Banach
space setting, along with stability and convergence of the finite
element approximation. In \cite{cesm2017}, the fully dynamic coupled
Navier-Stokes/Biot system with a pressure-based Darcy formulation is
analyzed. Additional works include optimization-based decoupling
method \cite{Cesm-etal-optim}, a second order in time split scheme
\cite{Kunwar-etal}, various discretization methods
\cite{Wen-He,Bergkamp-etal,Cesm-Chid},
dimensionally reduced model for flow through fractures
\cite{Buk-Yot-Zun-fracture}, and coupling with transport
\cite{fpsi-transport}. All of the above mentioned works are based on
displacement formulations for the elasticity equation. In a recent
work \cite{fpsi-mixed-elast}, the first mathematical and numerical
study of a stress-displacement mixed elasticity
formulation for the Stokes-Biot model is presented.

The goal of the present paper is to develop a new fully mixed
formulation of the quasi-static Stokes-Biot model, which is based on
dual mixed formulations for all three components - Darcy, elasticity,
and Stokes. In particular, we use a velocity-pressure Darcy
formulation, a weakly symmetric stress-displacement-rotation elasticity
formulation, and a weakly symmetric stress-velocity-vorticity Stokes
formulation. This formulation exhibits multiple advantages, including
local conservation of mass for the Darcy fluid, local poroelastic and
Stokes momentum conservation, and accurate approximations with
continuous normal components across element edges or faces for the
Darcy velocity, the poroelastic stress, and the free fluid stress. In
addition, dual mixed formulations are known for their locking-free
properties and robustness with respect to the physical parameters,
including the regimes of almost incompressible materials, low
poroelastic storativity, and low permeability \cite{Lee-Biot-five-field,Yi-Biot-locking}.

Our five-field dual mixed Biot formulation is based on the model
developed in \cite{Lee-Biot-five-field} and studied further in
\cite{msfmfe-Biot}. It is also considered in \cite{fpsi-mixed-elast}
for the Stokes-Biot problem. Our analysis also extends to the strongly
symmetric mixed four-field Biot formulation developed in
\cite{Yi-Biot-mixed}. Our three-field dual mixed Stokes formulation is
based on the models developed in \cite{gos2011,gmor2014}. In
particular, we introduce the stress tensor and subsequently eliminate
the pressure unknown, by utilizing the deviatoric stress.  In order to
impose the symmetry of the Stokes stress and poroelastic stress
tensors, the vorticity and structure rotation, respectively, are
introduced as additional unknowns.  The transmission conditions
consisting of mass conservation, conservation of momentum, and the
Beavers--Joseph--Saffman slip with friction condition are imposed
weakly via the incorporation of additional Lagrange multipliers: the
traces of the fluid velocity, structure velocity and the poroelastic
media pressure on the interface.  The resulting variational system of
equations is then ordered so that it shows a twofold saddle point
structure. 
The well-posedness and uniqueness of both the continuous
and semidiscrete continuous-in-time formulations are proved by 
employing some classical results for parabolic problems
\cite{Showalter,s2010} and monotone operators, and an abstract theory
for twofold saddle point problems \cite{ghm2003,adgm2019}.
In the
discrete problem, for the three components of the model we consider
suitable stable mixed finite element spaces on non-matching grids
across the interface, coupled through either conforming or
non-conforming Lagrange multiplier discretizations. We develop
stability and error analysis, establishing rates of convergence to the
true solution. The estimates we establish are uniform in the limit of
the storativity coefficient going to zero.

Another main contribution of this paper is the development of a new
mixed finite element method for the Stokes-Biot model that can be
reduced to a positive definite cell-centered
pressure-velocities-traces system. We recall the multipoint flux mixed
finite element (MFMFE) method for Darcy flow developed in
\cite{iwy2010,wxy2012,wy2006,Brezzi.F;Fortin.M;Marini.L2006}, where
the lowest order Brezzi-Douglas-Marini $\bbBDM_1$ velocity spaces
\cite{bdm1985,Nedelec86,Brezzi-Fortin} and piecewise constant pressure
are utilized. An alternative formulation based on a broken
Raviart-Thomas velocity space is developed in
\cite{Klausen-Winther-2006a}. The use of the vertex quadrature rule for
the velocity bilinear form localizes the interaction between velocity
degrees of freedom around mesh vertices and leads to a block-diagonal
mass matrix.  Consequently, the velocity can be locally eliminated,
resulting in a cell-centered pressure system.  In turn, the multipoint
stress mixed finite element (MSMFE) method for elasticity is developed in
\cite{msmfe-simpl,msmfe-quads}. It utilizes stable
weakly symmetric elasticity finite element triples with $\bbBDM_1$
stress spaces
\cite{BBF-reduced,FarFor,lee2016towards,msmfe-quads,afw2007,awanou2013}.
Similarly to the MFMFE method, an application of the vertex quadrature
rule for the stress and rotation bilinear forms allows for local
stress and rotation elimination, resulting in a cell-centered
displacement system. We also refer the reader to the related finite
volume multipoint stress approximation (MPSA) method for elasticity
\cite{Jan-IJNME,nordbotten2015convergence,keilegavlen2017finite}.
Recently, combining the MSMFE and MFMFE methods, a multipoint
stress-flux mixed finite element (MSFMFE) method for the Biot
poroelasticity model is developed in \cite{msfmfe-Biot}. There, the
dual mixed finite element system is reduced to a cell-centered
displacement-pressure system. The reduced system is comparable in 
cost to the finite volume method developed in \cite{Jan-SINUM-Biot}.

In this paper we note for the first time that the MSMFE method for
elasticity can be applied to the weakly symmetric
stress-velocity-vorticity Stokes formulation from
\cite{gos2011,gmor2014} when $\bbBDM_1$-based stable finite element
triples are utilized. With the application of the vertex quadrature
rule, the fluid stress and vorticity can be locally eliminated,
resulting in a positive definite cell-centered velocity system. To the
best of our knowledge, this is the first such scheme for Stokes in the
literature.

Finally, we combine the MFMFE method for Darcy with the
MSMFE methods for elasticity and Stokes to develop a multipoint
stress-flux mixed finite element for the Stokes-Biot system.  We
analyze the stability and convergence of the semidiscrete formulation.
We further consider the fully discrete system with backward Euler time
discretization and show that the algebraic system on each time step
can be reduced to a positive definite cell-centered
pressure-velocities-traces system.

The rest of this work is organized as follows.  The remainder of this
section describes standard notation and functional spaces to be
employed throughout the paper.  In Section~\ref{sec:model-problem} we
introduce the model problem and in Section~\ref{sec:weak-formulation}
we derive a fully-mixed variational formulation, which is written as a
degenerate evolution problem with a twofold saddle point structure.
Next, existence, uniqueness and stability of the solution of the weak
formulation are obtained in Section~\ref{sec:well-posedness-model}.
The corresponding semidiscrete continuous-in-time approximation is
introduced and analyzed in Section~\ref{sec:semidiscrete-formulation},
where the discrete analogue of the theory used in the continuous case
is employed to prove its well-posedness.  Error estimates and rates of
convergence are also derived there.  In
Section~\ref{sec:multipoint-fem}, the multipoint stress-flux mixed
finite element method is presented and the corresponding rates of
convergence are provided, along with the analysis of the reduced
cell-centered system.  Finally, numerical experiments illustrating the
accuracy of our mixed finite element method and its applications to
coupling surface and subsurface flows and flow through poroelastic
medium with a cavity are reported in
Section~\ref{sec:numerical-results}.



\medskip

We end this section by introducing some definitions and fixing some
notations.  Let $\cO\subset \R^n$, $n\in \{2,3\}$, denote a domain
with Lipschitz boundary.  For $\rs\geq 0$ and $\rp\in
[1,+\infty]$, we denote by $\L^\rp(\cO)$ and $\W^{\rs,\rp}(\cO)$ the usual
Lebesgue and Sobolev spaces endowed with the norms
$\|\cdot\|_{\L^\rp(\cO)}$ and $\|\cdot\|_{\W^{\rs,\rp}(\cO)}$, respectively.
Note that $\W^{0,\rp}(\cO) = \L^\rp(\cO)$.  If $\rp=2$ we write $\H^\rs(\cO)$
in place of $\W^{\rs,2}(\cO)$, and denote the corresponding norm by
$\|\cdot\|_{\H^\rs(\cO)}$. Similar notation is used for a section $\Gamma$
of the boundary of $\cO$. 
By $\bM$ and $\bbM$ we will denote the corresponding
vectorial and tensorial counterparts of a generic scalar functional
space $\M$. The $\L^2(\cO)$ inner product for scalar, vector, or tensor valued functions
is denoted by $(\cdot,\cdot)_{\cO}$. The $\L^2(\Gamma)$ inner product or duality pairing
is denoted by $\pil\cdot,\cdot\pir_\Gamma$. For any vector field $\bv=(v_i)_{i=1,n}$, we set
the gradient and divergence operators, as
\begin{equation*}
\nabla\bv:=\left(\frac{\partial v_i}{\partial x_j}\right)_{i,j=1,n} \qan 
\div(\bv):=\sum_{j=1}^n \frac{\partial v_j}{\partial x_j}.
\end{equation*}
For any tensor fields $\btau:=(\tau_{ij})_{i,j=1,n}$ and
$\bzeta:=(\zeta_{ij})_{i,j=1,n}$, we let $\bdiv(\btau)$ be the
divergence operator $\div$ acting along the rows of $\btau$, and
define the transpose, the trace, the tensor inner product, and the
deviatoric tensor, respectively, as
\begin{equation*}
\btau^\rt := (\tau_{ji})_{i,j=1,n},\quad \tr(\btau):=\sum_{i=1}^n \tau_{ii},\quad \btau:\bzeta:=\sum_{i,j=1}^n \tau_{ij}\zeta_{ij},\qan \btau^\rd:=\btau-\frac{1}{n}\,\tr(\btau)\,\bI,
\end{equation*}
where $\bI$ is the identity matrix in $\R^{n\times n}$.
In addition, we recall the Hilbert space
\begin{equation*}
\bH(\div;\cO):=\Big\{ \bv\in\bL^2(\cO) :\quad \div(\bv)\in \L^2(\cO) \Big\},
\end{equation*}
equipped with the norm $\|\bv\|^2_{\bH(\div;\cO)} :=
\|\bv\|^2_{\bL^2(\cO)} + \|\div(\bv)\|^2_{\L^2(\cO)}$.
The space of matrix
valued functions whose rows belong to $\bH(\div;\cO)$ will be denoted
by $\bbH(\bdiv;\cO)$ and endowed with the norm
$\|\btau\|^2_{\bbH(\bdiv;\cO)} := \|\btau\|^2_{\bbL^2(\cO)} +
\|\bdiv(\btau)\|^2_{\bL^2(\cO)}$.  Finally, given a separable Banach
space $\V$ endowed with the norm $\| \cdot \|_{\V}$, we let
$\L^{\rp}(0,T;\V)$ be the space of classes of functions $f : (0,T)\to
\V$ that are Bochner measurable and such that
$\|f\|_{\L^{\rp}(0,T;\V)} < \infty$, with
\begin{equation*}
\|f\|^{\rp}_{\L^{\rp}(0,T;\V)} \,:=\, \int^T_0 \|f(t)\|^{\rp}_{\V} \,dt,\quad
\|f\|_{\L^\infty(0,T;\V)} \,:=\, \mathop{\esssup}\limits_{t\in [0,T]} \|f(t)\|_{\V}.
\end{equation*}
%

\section{The model problem}\label{sec:model-problem}

Let $\Omega\subset\R^n$, $n\in\{2,3\}$, be a Lipschitz domain, which is
subdivided into two non-overlapping and possibly non-connected
regions: fluid region $\Omega_f$ and poroelastic region $\Omega_p$.
Let $\Gamma_{fp} = \partial\Omega_f\cap\partial\Omega_p$ denote the
(nonempty) interface between these regions and let $\Gamma_f =
\partial\Omega_f\setminus\Gamma_{fp}$ and $\Gamma_p =
\partial\Omega_p\setminus \Gamma_{fp}$ denote the external parts on
the boundary $\partial\Omega$.  We denote by $\bn_f$ and $\bn_p$ the
unit normal vectors that point outward from $\partial\Omega_f$ and
$\partial\Omega_p$, respectively, noting that $\bn_f = - \bn_p$ on
$\Gamma_{fp}$.  Let $(\bu_\star,p_\star)$ be the velocity-pressure
pair in $\Omega_\star$ with $\star\in\{f,p\}$, and let $\bbeta_p$ be
the displacement in $\Omega_p$.  Let $\mu>0$ be the fluid viscosity,
let $\f_\star$ be the body force terms, and let $q_\star$ be external
source or sink terms.

We assume that the flow in $\Omega_f$ is governed by the Stokes
equations, which are written in the following stress-velocity-pressure
formulation:
\begin{equation}\label{eq:Stokes-1} 
\begin{array}{c}
\ds \bsi_f \,=\, -p_f\,\bI + 2\,\mu\,\be(\bu_f),\quad 
-\,\bdiv(\bsi_f) \,=\, \f_f,\quad 
\div(\bu_f) \,=\, q_f \qin \Omega_f\times (0,T], \\ [2ex]
\ds \bsi_f\bn_f \,=\, \0 \qon \Gamma^\rN_f\times (0,T],\quad 
\bu_f \,=\, \0 \qon \Gamma^\rD_f\times (0,T],
\end{array}
\end{equation}
where $\bsi_f$ is the stress tensor, $\be(\bu_f) :=
\dfrac{1}{2}\,\left( \nabla\bu_f + (\nabla\bu_f)^\rt \right)$ stands
for the deformation rate tensor, $\Gamma_f = \Gamma^\rN_f\cup
\Gamma^\rD_f$, and $T>0$ is the final time.  Next, we adopt the
approach from \cite{gmor2014,adgm2019}, and include as a new variable
the vorticity tensor $\bgamma_f$,
\begin{equation*}
\bgamma_f \,:=\, \frac{1}{2}\,\left( \nabla\bu_f - (\nabla\bu_f)^\rt \right).
\end{equation*}
In this way, owing to the fact that $\tr(\be(\bu_f)) = \div(\bu_f) = q_f$, we find that \eqref{eq:Stokes-1} can be rewritten, equivalently, as the set of equations with unknowns $\bsi_f, \bgamma_f$ and $\bu_f$, given by
\begin{equation}\label{eq:Stokes-2}
\begin{array}{c}
\ds \frac{1}{2\,\mu}\,\bsi^\rd_f \,=\, \nabla\bu_f - \bgamma_f - \frac{1}{n}\,q_f\,\bI,\quad
-\,\bdiv(\bsi_f) \,=\, \f_f \qin \Omega_f\times (0,T], \\ [2ex]
\ds \bsi_f \,=\, \bsi^\rt_f,\quad 
p_f \,=\, -\frac{1}{n}\,\left( \tr(\bsi_f) - 2\,\mu\,q_f \right) \qin \Omega_f\times (0,T], \\ [3ex]
\ds \bsi_f\bn_f \,=\, \0 \qon \Gamma^\rN_f\times (0,T],\quad 
\bu_f \,=\, \0 \qon \Gamma^\rD_f\times (0,T].
\end{array}
\end{equation}
Notice that the fourth equation in \eqref{eq:Stokes-2} has allowed us
to eliminate the pressure $p_f$ from the system and provides a formula
for its approximation through a post-processing procedure. For
simplicity we assume that $|\Gamma^\rN_f| > 0$, which will allow us to
control $\bsi_f$ by $\bsi_f^\rd$. The case $|\Gamma^\rN_f| = 0$
can be handled as in \cite{gmor2014,gos2011,gos2012} by introducing an additional variable corresponding to the
mean value of $\tr(\bsi_f)$.

In turn, let $\bsi_e$ and $\bsi_p$ be the elastic and poroelastic stress tensors,
respectively, satisfying
\begin{equation}\label{eq:bsie-bsip-definitions}
A \, \bsi_e \,= \be(\bbeta_p) \qan
\bsi_p \,:=\, \bsi_e - \alpha_p\,p_p\,\bI \qin \Omega_p\times (0,T],
\end{equation}
where $0 < \alpha_p \leq 1$ is the
Biot--Willis constant, and $A$ is the symmetric and positive definite compliance
tensor, which in the isotropic case has the form, for all tensors $\btau$,
\begin{equation}\label{eq:operator-A-definition}
A(\btau) := \frac{1}{2\,\mu_p}\,\left(\btau - \frac{\lambda_p}{2\,\mu_p + n\,\lambda_p}\,\tr(\btau)\,\bI\right),\quad\mbox{with}\quad
A^{-1}(\btau) = 2\,\mu_p\,\btau + \lambda_p\,\tr(\btau)\,\bI,
\end{equation}
satisfying
\begin{equation}\label{A-bounds}
  \forall\, \btau\in\R^{n\times n}, \quad \frac{1}{2 \mu_{\max} + n \, \lambda_{\max}}
  \, \btau : \btau \, \leq \, A(\btau):\btau \, \leq \, \frac{1}{2 \mu_{\min}} \,
  \btau : \btau \quad \forall\, \bx\in\Omega_p.
  \end{equation}
In this case, $\bsi_e \,:=\, \lambda_p\,\div(\bbeta_p)\,\bI + 2\,\mu_p\,\be(\bbeta_p)$, and $0<\lambda_{\min}\leq \lambda_p(\bx)\leq \lambda_{\max}$ and $0<\mu_{\min} \leq\mu_p(\bx) \leq\mu_{\max}$ are the Lam\'e parameters.
The poroelasticity region $\Omega_p$ is governed by the quasi-static Biot system
\cite{b1941}:
\begin{equation}\label{eq:Biot-model}
\begin{array}{c}
\ds-\,\bdiv(\bsi_p) = \f_p,\quad 
\mu\,\bK^{-1}\bu_p + \nabla\,p_p = \0,\quad
\frac{\partial}{\partial t}\left( s_0\,p_p + \alpha_p\,\div(\bbeta_p) \right) + \div(\bu_p) = q_p \qin \Omega_p\times(0,T], \\ [3ex]
\ds \bu_p\cdot\bn_p = 0 \qon \Gamma^\rN_p\times (0,T],\quad 
p_p = 0 \qon \Gamma^\rD_p\times (0,T], \\ [3ex]
\ds \bsi_p \bn_p = \0 \qon \tilde\Gamma_p^\rN\times (0,T], \quad
\bbeta_p = \0 \qon \tilde\Gamma_p^\rD\times (0,T],
\end{array}
\end{equation}
where $\Gamma_p = \Gamma^\rN_p\cup \Gamma^\rD_p = \tilde\Gamma^\rN_p\cup \tilde\Gamma^\rD_p$,
$s_0 > 0$ is a storativity coefficient
and $\bK(\bx)$ is the symmetric and uniformly positive definite rock permeability tensor,
satisfying, for some constants $0< k_{\min}\leq k_{\max}$,
\begin{equation}\label{eq:K-uniform-bound}
\forall\, \bw\in\R^n, \quad k_{\min}\,\bw\cdot\bw 
\,\leq\, (\bK \bw) \cdot \bw
\,\leq\, k_{\max}\,\bw\cdot\bw \quad \forall\, \bx\in\Omega_p.
\end{equation}
To avoid the issue with restricting the mean value of the pressure, we assume that $|\Gamma^\rD_p| > 0$.
We also assume that $\Gamma^\rD_f$, $\Gamma^\rD_p$, and $\tilde\Gamma^\rD_p$ are
not adjacent to the interface $\Gamma_{fp}$, i.e., $\exists \ s > 0$ such that
$\dist(\Gamma^\rD_f,\Gamma_{fp}) \ge s$, $\dist(\Gamma^\rD_p,\Gamma_{fp}) \ge s$, and
$\dist(\tilde\Gamma^\rD_p,\Gamma_{fp}) \ge s$. This assumption is used to simplify
the characterization of the normal trace spaces on $\Gamma_{fp}$.

Next, we introduce the following transmission conditions on the interface
$\Gamma_{fp}$ \cite{s2005,bqq2009,byzz2015,akyz2018}:
\begin{equation}\label{eq:interface-conditions}
\begin{array}{c}
\ds \bu_f\cdot\bn_f + \left(\frac{\partial\,\bbeta_p}{\partial t} + \bu_p\right)\cdot\bn_p \,=\, 0, \quad
\bsi_f\bn_f + \bsi_p\bn_p \,=\, \0 \qon \Gamma_{fp}\times (0,T], \\ [3ex]
\ds \bsi_f\bn_f + \mu\,\alpha_{\BJS}\sum^{n-1}_{j=1}\,\sqrt{\bK^{-1}_j}\left\{\left(\bu_f - \frac{\partial\,\bbeta_p}{\partial t}\right)\cdot\bt_{f,j}\right\}\,\bt_{f,j} \,=\, -\,p_p\bn_f \qon \Gamma_{fp}\times (0,T],
\end{array}
\end{equation}
where $\bt_{f,j}$, $1\leq j\leq n-1$, is an orthogonal system of unit tangent vectors on $\Gamma_{fp}$, $\bK_j = (\bK\,\bt_{f,j})\cdot\bt_{f,j}$, and $\alpha_{\BJS} \geq 0$ is an experimentally determined friction coefficient.
The first and second equations in \eqref{eq:interface-conditions} correspond to mass conservation and conservation of momentum on $\Gamma_{fp}$, respectively, whereas the third one can be decomposed into its normal and tangential components, as follows:
\begin{equation*}
(\bsi_f\bn_f)\cdot\bn_f \,=\, -\,p_p,\quad
(\bsi_f\bn_f)\cdot\bt_{f,j} \,=\, -\,\mu\,\alpha_{\BJS}\,\sqrt{\bK^{-1}_j}\left(\bu_f - \frac{\partial\,\bbeta_p}{\partial t}\right)\cdot\bt_{f,j} \qon \Gamma_{fp}\times (0,T],
\end{equation*}
representing balance of normal stress and the Beaver--Joseph--Saffman (BJS) slip with friction condition, respectively.

Finally, the above system of equations is complemented by the initial condition
$p_p(\bx,0) = p_{p,0}(\bx)$ in $\Omega_p$.
We stress that, similarly to \cite{fpsi-mixed-elast}, compatible initial data for the rest of the variables can be constructed from $p_{p,0}$ in a way that all equations in the system \eqref{eq:Stokes-2}--\eqref{eq:interface-conditions}, except for the unsteady conservation of mass equation in the first row of \eqref{eq:Biot-model}, hold at $t=0$.
This will be established in Lemma~\ref{lem:sol0-in-M-operator} below.
We will consider a weak formulation with a time-differentiated elasticity equation and compatible initial data $(\bsi_{p,0},p_{p,0})$.

\section{The weak formulation}\label{sec:weak-formulation}

In this section we proceed analogously to \cite[Section~3]{aeny2019} (see also \cite{gmor2014}) and derive a weak formulation of the coupled problem given by \eqref{eq:Stokes-2}, \eqref{eq:bsie-bsip-definitions}--\eqref{eq:Biot-model}, and \eqref{eq:interface-conditions}.

\subsection{Preliminaries}
For the stress tensor, velocity, and vorticity in the Stokes region,
we use the Hilbert spaces, respectively,
\begin{equation*}
  \bbX_f := \Big\{ \btau_f\in \bbH(\bdiv;\Omega_f) :  \btau_f\bn_f = \0
  \ \text{ on } \
  \Gamma^\rN_f \Big\},\,\,\, 
  \bV_f := \bL^2(\Omega_f),
\,\,\,
  \ds \bbQ_f := \Big\{ \bchi_f\in \bbL^2(\Omega_f) :  \bchi^\rt_f = -\,\bchi_f \Big\},
\end{equation*}
endowed with the corresponding norms
\begin{equation*}
\|\btau_f\|_{\bbX_f} := \|\btau_f\|_{\bbH(\bdiv;\Omega_f)},\quad
\|\bv_f\|_{\bV_f} := \|\bv_f\|_{\bL^2(\Omega_f)},\quad
\|\bchi_f\|_{\bbQ_f} := \|\bchi_f\|_{\bbL^2(\Omega_f)}.
\end{equation*}
For the unknowns in the Biot region we introduce the Hilbert spaces: 
\begin{equation*}
\begin{array}{c}
  \ds \bbX_p := \Big\{ \btau_p \in \bbH(\bdiv;\Omega_p): \btau_p \bn_p = \0
  \ \text{ on } \ \tilde\Gamma^\rN_p \Big\}, \,\,\,
\bV_s:=\bL^2(\Omega_p),\,\,\,
\bbQ_p := \Big\{ \bchi_p\in \bbL^2(\Omega_p) :  \bchi^\rt_p = -\,\bchi_p \Big\},\\ [2ex]
\ds \bV_p := \Big\{ \bv_p\in \bH(\div;\Omega_p) : \ \bv_p\cdot\bn_p = 0 \ \text{ on } \ \Gamma^\rN_p \Big\},\quad
\W_p := \L^2(\Omega_p),
\end{array}
\end{equation*}
endowed with the standard norms
\begin{equation*}
\begin{array}{c}
\ds \|\btau_p\|_{\bbX_p} := \|\btau_p\|_{\bbH(\bdiv;\Omega_p)},\quad
\|\bv_s\|_{\bV_s} := \|\bv_s\|_{\bL^2(\Omega_p)},\quad
\|\bchi_p\|_{\bbQ_p} := \|\bchi_p\|_{\bbL^2(\Omega_p)}, \\ [2ex]
\ds \|\bv_p\|_{\bV_p} := \|\bv_p\|_{\bH(\div;\Omega_p)},\quad
\|w_p\|_{\W_p} := \|w_p\|_{\L^2(\Omega_p)}.
\end{array}
\end{equation*}

Finally, analogously to \cite{gs2007,gmor2014,akyz2018,aeny2019,fpsi-mixed-elast} we
need to introduce the Lagrange multiplier spaces $\Lambda_p :=
(\bV_p\cdot\bn_p|_{\Gamma_{fp}})'$, $\bLambda_f :=
(\bbX_f\,\bn_f|_{\Gamma_{fp}})'$, and $\bLambda_s :=
(\bbX_p\,\bn_p|_{\Gamma_{fp}})'$.  According to the normal trace
theorem, since $\bv_p\in \bV_p\subset \bH(\div;\Omega_p)$, then
$\bv_p\cdot\bn_p\in \H^{-1/2}(\partial \Omega_p)$.  It is shown in
\cite{gs2007} that, if $\bv_p\cdot\bn_p = 0$ on
$\partial\,\Omega_p\setminus \Gamma_{fp}$, then $\bv_p\cdot\bn_p \in
\H^{-1/2}(\Gamma_{fp})$. This argument has been modified in \cite{akyz2018}
for the case $\bv_p\cdot\bn_p = 0$ on $\Gamma^\rN_p$ and 
$\dist(\Gamma^\rD_p,\Gamma_{fp}) \geq s > 0$. In particular, it holds that
\begin{equation}\label{trace-u}
  \langle \bv_p\cdot\bn_p,\xi \rangle_{\Gamma_{fp}}
  \le C \|\bv_p\|_{\bH(\div;\Omega_p)}\|\xi\|_{\H^{1/2}(\Gamma_{fp})}, \quad
  \forall \, \bv_p \in \bV_p, \, \xi \in \H^{1/2}(\Gamma_{fp}).
\end{equation}
Similarly,
\begin{equation}\label{trace-sigma}
  \langle \btau_\star \, \bn_\star,\bpsi \rangle_{\Gamma_{fp}}
  \le C \|\btau_\star\|_{\bbH(\div;\Omega_\star)}\|\bpsi\|_{\bH^{1/2}(\Gamma_{fp})},
  \quad \forall \, \btau_\star \in \bbX_\star, \,
  \bpsi \in \bH^{1/2}(\Gamma_{fp}), \ \star \in \{f,p\}.
\end{equation}
Therefore we can take $\Lambda_p := \H^{1/2}(\Gamma_{fp})$,
$\bLambda_f := \bH^{1/2}(\Gamma_{fp})$,
and $\bLambda_s := \bH^{1/2}(\Gamma_{fp})$, endowed with the norms
\begin{equation}\label{eq:norms-H-half}
\|\xi\|_{\Lambda_p} \,:=\, \|\xi\|_{\H^{1/2}(\Gamma_{fp})},\quad 
\|\bpsi\|_{\bLambda_f} \,:=\, \|\bpsi\|_{\bH^{1/2}(\Gamma_{fp})},\qan 
\|\bphi\|_{\bLambda_s} \,:=\, \|\bphi\|_{\bH^{1/2}(\Gamma_{fp})}.
\end{equation}

\subsection{Lagrange multiplier formulation}

We now proceed with the derivation of our Lagrange multiplier
variational formulation for the coupling of the Stokes and Biot
problems.  To this end, and inspired by \cite{aeny2019,gos2011}, we
begin by introducing the structure velocity $\bu_s :=
\partial_t\,\bbeta_p\in \bV_s$ satisfying $\bu_s = \0$ on
$\tilde\Gamma_p^\rD\times (0,T]$ (cf. the last equation in \eqref{eq:Biot-model}),
and three Lagrange multipliers modeling the Stokes
velocity, structure velocity and Darcy pressure on the interface,
respectively,
\begin{equation*}
\bvarphi \,:=\, \bu_f|_{\Gamma_{fp}} \in \bLambda_f,\quad 
\btheta \,:=\, \bu_s|_{\Gamma_{fp}}\in \bLambda_s,\qan
\lambda \,:=\, p_p|_{\Gamma_{fp}}\in \Lambda_p.
\end{equation*}
The reason for introducing these Lagrange multipliers is
twofold. First, $\bu_f$, $\bu_s$, and $p_p$ are all modeled in the
$\L^2$ space, thus they do not have sufficient regularity for their
traces on $\Gamma_{fp}$ to be well defined. Second, the Lagrange multipliers
are utilized to impose weakly the transmission conditions
\eqref{eq:interface-conditions}.

To impose the symmetry condition of $\bsi_p$ in a weak sense we introduce the rotation operator $\brho_p := \dfrac{1}{2}(\nabla\bbeta_p - \nabla\bbeta^{\rt}_p)$.
Notice that in the weak formulation we will use its time derivative, that is, the structure rotation velocity
\begin{equation*}
\bgamma_p \,:=\, \partial_t\brho_p = \frac{1}{2}\,\left( \nabla\bu_s - (\nabla\bu_s)^\rt \right)\in \bbQ_p.
\end{equation*}
From the definition of the elastic and poroelastic stress tensors
$\bsi_e, \bsi_p$ (cf. \eqref{eq:bsie-bsip-definitions}) and recalling
that $\bsi_e$ is connected to the displacement $\bbeta_p$ through the
relation $A(\bsi_e) = \be(\bbeta_p)$, we deduce the identities
\begin{equation}\label{eq:div-bbeta_p}
\div(\bbeta_p) \,=\, \tr (\be(\bbeta_p)) \, = \, \tr(A\bsi_e) \, = \, \tr A(\bsi_p+\alpha_p \,p_p\, \bI)
\end{equation}
and
\begin{equation}\label{eq:partial-t-A-identity}
\partial_t\,A(\bsi_p + \alpha_p\,p_p\,\bI) 
\,=\, \nabla\,\bu_s - \bgamma_p\,.
\end{equation}
Then, similarly to \cite{gos2011,gmor2014,akyz2018,aeny2019}, we test
the first equation of \eqref{eq:Stokes-2}, the second equation of
\eqref{eq:Biot-model}, and \eqref{eq:partial-t-A-identity} with
arbitrary $\btau_f\in \bbX_f, \bv_p\in \bV_p$, and $\btau_p\in
\bbX_p$, respectively, integrate by parts, utilize the fact that
$\bsi^\rd_f:\btau_f = \bsi^\rd_f:\btau^\rd_f$, test the third equation
of \eqref{eq:Biot-model} with $w_p\in \W_p$ employing
\eqref{eq:div-bbeta_p}, impose the remaining equations weakly, and
utilize the transmission conditions in \eqref{eq:interface-conditions}
to obtain the variational problem,
\begin{align}\label{eq:Stokes-Biot-formulation-1}
& \ds \frac{1}{2\mu}\,(\bsi^\rd_f,\btau^\rd_f)_{\Omega_f}
  + (\bu_f,\bdiv(\btau_f))_{\Omega_f}
  + (\bgamma_f,\btau_f)_{\Omega_f}
  - \pil\btau_f\bn_f,\bvarphi \pir_{\Gamma_{fp}}
  =  \ds -\frac{1}{n}\,(q_f \, \bI ,\btau_f)_{\Omega_f}, \nonumber \\
& \ds -\,(\bv_f,\bdiv(\bsi_f))_{\Omega_f} = \ds (\f_f,\bv_f)_{\Omega_f}, \nonumber \\ 
& \ds -\,(\bsi_f,\bchi_f)_{\Omega_f} = 0,\nonumber  \\
& \ds  (\partial_t\,A(\bsi_p + \alpha_p\,p_p\,\bI), \btau_p)_{\Omega_p}
  +\,(\bu_s, \bdiv(\btau_p))_{\Omega_p}
  + (\bgamma_p,\btau_p)_{\Omega_p}
  - \pil\btau_p\bn_p,\btheta\pir_{\Gamma_{fp}}
  = 0,\nonumber \\
& \ds -\,(\bv_s,\bdiv(\bsi_p))_{\Omega_p} = (\f_p, \bv_s)_{\Omega_p}, \nonumber \\
& \ds -\,(\bsi_p,\bchi_p)_{\Omega_p} = 0, \nonumber \\
& \ds \mu\,(\bK^{-1}\bu_p,\bv_p)_{\Omega_p}
  - (p_p,\div(\bv_p))_{\Omega_p}
  + \pil\bv_p\cdot\bn_p,\lambda\pir_{\Gamma_{fp}} = 0, \\[0.5ex]
& \ds (s_0\,\partial_t\,p_p, w_p)_{\Omega_p}
  + \alpha_p\,(\partial_t\,A(\bsi_p + \alpha_p\,p_p\,\bI), w_p\,\bI)_{\Omega_p}
\ds +\, (w_p,\div(\bu_p))_{\Omega_p} = (q_p,w_p)_{\Omega_p}, \nonumber  \\[0.5ex]
& \ds -\,\pil\bvarphi\cdot\bn_f + \left(\btheta + \bu_p\right)\cdot\bn_p,\xi\pir_{\Gamma_{fp}}
= 0, \nonumber \\[0.5ex]
& \ds \pil\bsi_f\bn_f,\bpsi\pir_{\Gamma_{fp}} + \mu\,\alpha_{\BJS}\,\sum_{j=1}^{n-1}
\pil\sqrt{\bK_j^{-1}}\left( \bvarphi - \btheta \right)\cdot\bt_{f,j},\bpsi\cdot\bt_{f,j} \pir_{\Gamma_{fp}}
\ds +\, \pil\bpsi\cdot\bn_f,\lambda\pir_{\Gamma_{fp}} = 0, \nonumber \\[0.5ex]
& \ds \pil\bsi_p\bn_p,\bphi\pir_{\Gamma_{fp}} - \mu\,\alpha_{\BJS}\,\sum_{j=1}^{n-1}
\pil\sqrt{\bK_j^{-1}}\left( \bvarphi - \btheta \right)\cdot\bt_{f,j},\bphi\cdot\bt_{f,j} \pir_{\Gamma_{fp}}
\ds +\, \pil\bphi\cdot\bn_p,\lambda\pir_{\Gamma_{fp}} = 0 \nonumber.
\end{align}
The last three equations impose weakly the transmission conditions
\eqref{eq:interface-conditions}. In particular, the equation with test
function $\xi$ imposes the mass conservation, the equation with $\bpsi$ imposes
the last equation in \eqref{eq:interface-conditions}, which is a combination
of balance of normal stress and the BJS condition, while the equation
with $\bphi$ imposes the conservation of momentum. We emphasize that this is
a new formulation. To our knowledge, this is the first fully dual-mixed formulation
for the Stokes-Biot problem.

\begin{rem}\label{rem:A-sigmap-pp-without-time-derivative}
The time differentiated equation in the fourth row of \eqref{eq:Stokes-Biot-formulation-1} allows us to eliminate the displacement variable $\bbeta_p$ and obtain a formulation that uses only $\bu_s$.
As part of the analysis we will construct suitable initial data such that, by integrating in time the fourth equation of \eqref{eq:Stokes-Biot-formulation-1}, we can recover the original equation
\begin{equation}\label{eq:A-sigmap-pp-without-time-derivative}
(A(\bsi_p + \alpha_p\,p_p\,\bI), \btau_p)_{\Omega_p}
+ (\bbeta_p, \bdiv(\btau_p))_{\Omega_p}
+ (\brho_p,\btau_p)_{\Omega_p}
- \pil\btau_p\bn_p,\bomega\pir_{\Gamma_{fp}}  
= 0,
\end{equation}
where $\bomega := \bbeta_p|_{\Gamma_{fp}}$.
\end{rem}

To simplify the notation, we set the following bilinear forms:
\begin{equation}\label{eq:bilinear-forms-1}
\begin{array}{c}
\ds a_f(\bsi_f,\btau_f) \,:=\, \frac{1}{2\,\mu}\,(\bsi^\rd_f,\btau^\rd_f)_{\Omega_f},\quad 
a_p(\bu_p,\bv_p) \,:=\, \mu\,(\bK^{-1}\bu_p,\bv_p)_{\Omega_p},  \\ [2ex]
a_e(\bsi_p, p_p; \btau_p, w_p) \,:=\, (A(\bsi_p + \alpha_p \, p_p \, \bI), \btau_p + \alpha_p \, w_p \, \bI )_{\Omega_p},  \\ [2ex]
\ds b_f(\btau_f,\bv_f) \,:=\, (\bdiv(\btau_f),\bv_f)_{\Omega_f},\quad
b_s(\btau_p,\bv_s) \,:=\, (\bdiv(\btau_p),\bv_s)_{\Omega_p}, \\ [2ex]
\ds b_p(\bv_p,w_p) \,:=\, -\,(\div(\bv_p),w_p)_{\Omega_p},\quad
b_\Gamma(\bv_p,\xi) \,:=\, \pil\bv_p\cdot\bn_p,\xi\pir_{\Gamma_{fp}}, \\ [2ex] 
\ds b_{\sk,\star}(\btau_\star,\bchi_\star) \,:=\, (\btau_\star,\bchi_\star)_{\Omega_\star},\quad
b_{\bn_\star}(\btau_\star,\bpsi) \,:=\, -\,\pil\btau_\star\bn_\star,\bpsi\pir_{\Gamma_{fp}},\mbox{ with } \star\in\big\{f,p\big\},
\end{array}
\end{equation}
and
\begin{equation}\label{eq:bilinear-forms-2}
\begin{array}{c}
\ds c_{\BJS}(\bvarphi,\btheta;\bpsi,\bphi) 
\,:=\, \mu\,\alpha_{\BJS}\,\sum^{n-1}_{j=1} \pil\sqrt{\bK_j^{-1}}(\bvarphi-\btheta)\cdot\bt_{f,j},(\bpsi-\bphi)\cdot\bt_{f,j}\pir_{\Gamma_{fp}}, \\ [3ex]
\ds c_{\Gamma}(\bpsi,\bphi;\xi) \,:=\, \pil\bpsi\cdot\bn_f,\xi\pir_{\Gamma_{fp}} + \pil\bphi\cdot\bn_p,\xi\pir_{\Gamma_{fp}}.
\end{array}
\end{equation}

There are many different ways of ordering the variables in
\eqref{eq:Stokes-Biot-formulation-1}. For the sake of the subsequent
analysis, we proceed as in \cite{gmor2014} and \cite{aeny2019}, and
adopt one leading to an evolution problem in a doubly-mixed
form. Hence, the variational formulation for the system
\eqref{eq:Stokes-Biot-formulation-1} reads: Given
\begin{equation*}
\f_f:[0,T] \to \bV_f',\quad \f_p:[0,T] \to \bV_s',\quad q_f:[0,T]\to \bbX'_f,\quad
q_p:[0,T]\to \W_p', \quad p_{p,0} \in \W_p, \quad \bsi_{p,0} \in \bbX_p,
\end{equation*}
find $(\bsi_f, \bu_p, \bsi_p, p_p, \bvarphi, \btheta, \lambda, \bu_f, \bu_s,
\bgamma_f, \bgamma_p):[0,T] \to
\bbX_{f}\times \bV_p\times \bbX_p\times \W_p\times \bLambda_f\times \bLambda_s
\times \Lambda_p\times \bV_f\times \bV_s\times \bbQ_f\times \bbQ_p$, such that
$p_p(0) = p_{p,0}$, $\bsi_p(0)=\bsi_{p,0}$ and for a.e. $t\in (0,T)$\,:\,
\begin{align}\label{eq:Stokes-Biot-formulation}
& \ds a_f(\bsi_f,\btau_f) + a_p(\bu_p,\bv_p) + 
a_e(\partial_t \, \bsi_p, \partial_t \, p_p; \btau_p, w_p) +
(s_0\,\partial_t\,p_p,w_p)_{\Omega_p} \nonumber \\[0.5ex] 
& \ds\quad +\,\,b_p(\bv_p,p_p) - b_p(\bu_p,w_p) + b_{\bn_f}(\btau_f,\bvarphi) + b_{\bn_p}(\btau_p,\btheta) + b_\Gamma(\bv_p,\lambda) \nonumber \\[0.5ex]
& \ds\quad +\,\,
b_f(\btau_f,\bu_f) + b_s(\btau_p,\bu_s) + b_\skf(\btau_f,\bgamma_f)
+ b_\skp(\btau_p,\bgamma_p) \,=\, -\,\frac{1}{n}\,(q_f \, \bI ,\btau_f)_{\Omega_f} + (q_p,w_p)_{\Omega_p}, \\[1ex]
& \ds -\,b_{\bn_f}(\bsi_f,\bpsi) - b_{\bn_p}(\bsi_p,\bphi) - b_\Gamma(\bu_p,\xi) + c_{\BJS}(\bvarphi,\btheta;\bpsi,\bphi) + c_{\Gamma}(\bpsi,\bphi;\lambda) - c_{\Gamma}(\bvarphi,\btheta;\xi) = 0, \nonumber \\[0.5ex] 
& \ds -\,b_f(\bsi_f,\bv_f) - b_s(\bsi_p,\bv_s) - b_\skf(\bsi_f,\bchi_f)
- b_\skp(\bsi_p,\bchi_p) \,=\, (\f_f,\bv_f)_{\Omega_f} + (\f_p,\bv_s)_{\Omega_p}, \nonumber
\end{align}
%
$ \forall \ \btau_f\in \bbX_{f}, \bv_p\in \bV_p, \btau_p\in \bbX_p, w_p\in \W_p, \bpsi\in \bLambda_f, \bphi\in \bLambda_s, \xi\in \Lambda_p, \bv_f\in \bV_f, \bv_s\in \bV_s, \bchi_f\in \bbQ_f, \bchi_p\in \bbQ_p$. 

Now, we group the spaces and test functions as follows:
\begin{equation*}
\begin{array}{c}
\ds \bX \,:=\, \bbX_{f}\times \bV_p\times \bbX_p\times \W_p,\quad
\bY \,:=\, \bLambda_f\times \bLambda_s\times \Lambda_p,\quad
\bZ \,:=\, \bV_f\times \bV_s\times \bbQ_f\times \bbQ_p, \\ [2ex]
\ds \ubsi \,:=\, (\bsi_f, \bu_p, \bsi_p, p_p)\in \bX,\quad
\ubvarphi \,:=\, (\bvarphi, \btheta, \lambda)\in \bY,\quad 
\ubu \,:=\, (\bu_f, \bu_s, \bgamma_f, \bgamma_p)\in \bZ, \\ [1ex]
\ds \ubtau \,:=\, (\btau_f, \bv_p, \btau_p, w_p)\in \bX,\quad
\ubpsi \,:=\, (\bpsi, \bphi, \xi)\in \bY,\quad 
\ubv \,:=\, (\bv_f, \bv_s, \bchi_f, \bchi_p)\in \bZ,
\end{array}
\end{equation*}
where the spaces $\bX, \bY$ and $\bZ$ are endowed with the norms, respectively,
\begin{equation*}
\begin{array}{l}
\|\ubtau\|_{\bX} \,:=\, \|\btau_f\|_{\bbX_f} + \|\bv_p\|_{\bV_p} + \|\btau_p\|_{\bbX_p} + \|w_p\|_{\W_p}, \quad
\|\ubpsi\|_{\bY} \,:=\, \|\bpsi\|_{\bLambda_f} + \|\bphi\|_{\bLambda_s} + \|\xi\|_{\Lambda_p}, \\ [2ex]
\|\ubv\|_{\bZ} \,:=\, \|\bv_f\|_{\bV_f} + \|\bv_s\|_{\bV_s} + \|\bchi_f\|_{\bbQ_f} + \|\bchi_p\|_{\bbQ_p}. 
\end{array}
\end{equation*}
Hence, we can write \eqref{eq:Stokes-Biot-formulation} in an operator notation
as a degenerate evolution problem in a doubly-mixed form:
\begin{equation}\label{eq:evolution-problem-in-operator-form}
\begin{array}{lllll}
\ds \frac{\partial}{\partial t}\,\cE(\ubsi(t)) + \cA(\ubsi(t)) + \cB'_1(\ubvarphi(t)) + \cB'(\ubu(t)) & = & \bF(t) & \mbox{ in } & \bX', \\ [1.5ex]
\ds -\,\cB_1(\ubsi(t)) + \cC(\ubvarphi(t)) & = & \0 & \mbox{ in } & \bY', \\ [1.5ex]
\ds -\,\cB\,(\ubsi(t)) & = & \bG(t) & \mbox{ in } & \bZ',
\end{array}
\end{equation}
where, according to \eqref{eq:bilinear-forms-1}--\eqref{eq:bilinear-forms-2},
the operators $\cA : \bX\to \bX', \cB_1 : \bX\to \bY', \cC : \bY\to \bY'$,
and $\cB : \bX\to \bZ'$, are defined by 
\begin{equation}\label{eq:operators-A-B1-C}
\begin{array}{l}
\cA(\ubsi)(\ubtau) \,:=\, a_f(\bsi_f,\btau_f) + a_p(\bu_p,\bv_p)
+ b_p(\bv_p,p_p) - b_p(\bu_p,w_p), \\ [2ex] 
\cB_1(\ubtau)(\ubpsi) \,:=\, b_{\bn_f}(\btau_f,\bpsi)
+ b_{\bn_p}(\btau_p,\bphi) + b_{\Gamma}(\bv_p,\xi), \\ [2ex] 
\cC(\ubvarphi)(\ubpsi) \,:=\, c_{\BJS}(\bvarphi,\btheta;\bpsi,\bphi) 
+ c_{\Gamma}(\bpsi,\bphi;\lambda) - c_{\Gamma}(\bvarphi,\btheta;\xi),
\end{array}
\end{equation}
and 
\begin{equation}\label{eq:operator-B}
\cB(\ubtau)(\ubv) \,:=\, b_f(\btau_f,\bv_f)
+ b_s(\btau_p,\bv_s) + b_{\skf}(\btau_f,\bchi_f) + b_{\skp}(\btau_p,\bchi_p),
\end{equation}
whereas the operator $\cE : \bX\to \bX'$ is given by
\begin{equation}\label{eq:time-operator-E}
\cE(\ubsi)(\ubtau) \,:=\, a_e(\bsi_p,p_p;\btau_p,w_p) + (s_0\,p_p,w_p)_{\Omega_p},
\end{equation}
and the functionals $\bF\in \bX'$, $\bG\in \bZ'$ are defined as
\begin{equation}\label{eq:data-F-G}
\bF(\ubtau) \,:=\, -\,\frac{1}{n}\,(q_f \, \bI ,\btau_f)_{\Omega_f} + (q_p,w_p)_{\Omega_p} \qan
\bG(\ubv) \,:=\, (\f_f,\bv_f)_{\Omega_f} + (\f_p,\bv_s)_{\Omega_p}.
\end{equation}

\section{Well-posedness of the model}\label{sec:well-posedness-model}

In this section we establish the solvability of
\eqref{eq:evolution-problem-in-operator-form} (equivalently
\eqref{eq:Stokes-Biot-formulation}).  To that end we first collect
some previous results that will be used in the forthcoming analysis.

\subsection{Preliminaries}

We begin by recalling the following key result given in
\cite[Theorem~IV.6.1(b)]{Showalter} that will be used to establish the
existence of a solution to
\eqref{eq:evolution-problem-in-operator-form}.

\begin{thm}\label{thm:solvability-parabolic-problem}
Let the linear, symmetric and monotone operator $\cN$ be given for the real vector space $E$ to its algebraic dual $E^*$, and let $E'_b$ be the Hilbert space which is the dual of $E$ with the seminorm
\begin{equation*}
|x|_b = \big(\cN\,x(x)\big)^{1/2} \quad x\in E.
\end{equation*}
Let $\cM\subset E\times E'_b$ be a relation with domain $\cD = \Big\{ x\in E \,:\, \cM(x) \neq \emptyset \Big\}$.
	
Assume $\cM$ is monotone and $Rg(\cN + \cM) = E'_b$.
Then, for each $u_0\in \cD$ and for each $f\in \W^{1,1}(0,T;E'_b)$, there is a solution $u$ of
\begin{equation}\label{eq:general-parabolic-problem}
\frac{d}{dt}\big(\cN\,u(t)\big) + \cM\big(u(t)\big) \ni f(t) \quad a.e. \ 0 < t < T,
\end{equation}
with
\begin{equation*}
\cN\,u\in \W^{1,\infty}(0,T;E'_b),\quad u(t)\in \cD,\quad \mbox{ for all }\, 0\leq t\leq T,\qan \cN\,u(0) = \cN\,u_0.
\end{equation*}
\end{thm}

In addition, in order to show the range condition of
Theorem~\ref{thm:solvability-parabolic-problem} in our context, we
will require the following theorem whose proof can be derived
similarly to \cite[Theorem~2.2]{ghm2003} (see also
\cite[Theorem~3.13]{adgm2019} for a generalized nonlinear Banach
version).
\begin{thm}\label{thm:system_solvability}
Let $X, Y$, and $Z$ be Hilbert spaces, and let $X', Y', Z'$ be their respective duals. 
Let $A:X\to X'$, $S:Y\to Y'$, $B_1:X\to Y'$, and $B:X\to Z'$ be linear bounded operators. We also let $B'_1:Y\to X'$ and $B':Z\to X'$ be the corresponding adjoints.
Finally, we let $V$ be the kernel of $B$, that is
\begin{equation*}
V \,:=\, \Big\{ \btau\in X :\quad B(\btau)(\bv) = 0 \quad \forall\, \bv\in Z \Big\}.
\end{equation*}
	
Assume that
\begin{enumerate}
\item[(i)] $A|_{V}:V\to V'$ is elliptic, that is, there exists a constant $\alpha>0$ such that
\begin{equation*}
A(\btau)(\btau) \,\geq\, \alpha\,\|\btau\|^2_{X} \quad \forall\, \btau\in V.
\end{equation*}
		
\item[(ii)] $S$ is positive semi-definite on $Y$, that is,
\begin{equation*}
S(\bpsi)(\bpsi) \,\geq\, 0 \quad \forall\,\bpsi\in Y.
\end{equation*}
		
\item[(iii)] $B_1$ satisfies an inf-sup condition on $V\times Y$, that is, there exists $\beta_1>0$ such that
\begin{equation*}
\sup_{\0\neq \btau\in V} \frac{B_1(\btau)(\bpsi)}{\|\btau\|_{X}} 
\,\geq\, \beta_1\,\|\bpsi\|_{Y} \quad \forall\, \bpsi\in Y.
\end{equation*}
		
\item[(iv)] $B$ satisfies an inf-sup condition on $X\times Z$, that is, there exists $\beta>0$ such that
\begin{equation*}
\sup_{\0\neq \btau\in X} \frac{B(\btau)(\bv)}{\|\btau\|_X} 
\,\geq\, \beta\,\|\bv\|_{Z} \quad \forall\, \bv\in Z.
\end{equation*}
\end{enumerate}
Then, for each $(F_1, F_2,G)\in X'\times Y'\times Z'$ there exists a unique $(\bsi,\bvarphi,\bu)\in X\times Y\times Z$, such that
\begin{equation*}
\begin{array}{llll}
A(\bsi)(\btau) + B'_1(\bvarphi)(\btau) + B'(\bu)(\btau) & = & F_1(\btau) & \forall\, \btau\in X, \\ [1ex]
B_1(\bsi)(\bpsi) - S(\bvarphi)(\bpsi) & = & F_2(\bpsi) & \forall\, \bpsi\in Y, \\ [1ex]
B(\bsi)(\bv) & = & G(\bv) & \forall\, \bv\in Z.
\end{array}
\end{equation*}
Moreover, there exists $C>0$, depending only on $\alpha, \beta_1, \beta, \|A\|, \|S\|$, and $\|B_1\|$ such that
\begin{equation*}
\|(\bsi,\bvarphi,\bu)\|_{X\times Y\times Z} \,\leq\, C\,\Big\{\|F_1\|_{X'} + \|F_2\|_{Y'} + \|G\|_{Z'}\Big\}.
\end{equation*}
\end{thm}

At this point we recall, for later use, that there exist positive constants $c_1(\Omega_f)$ and $c_2(\Omega_f)$, such that (see, \cite[Proposition~IV.3.1]{Brezzi-Fortin} and \cite[Lemma~2.5]{Gatica}, respectively)
\begin{equation}\label{eq:tau-d-H0div-inequality}
c_1(\Omega_f)\,\|\btau_{f,0}\|^2_{\bbL^2(\Omega_f)} 
\,\leq\, \|\btau^\rd_f\|^2_{\bbL^2(\Omega_f)} + \|\bdiv(\btau_f)\|^2_{\bL^2(\Omega_f)} \quad \forall\,\btau_f = \btau_{f,0} + \ell\,\bI\in \bbH(\bdiv;\Omega_f)
\end{equation}
and
\begin{equation}\label{eq:tau-H0div-Xf-inequality}
c_2(\Omega_f)\,\|\btau_f\|^2_{\bbX_f} 
\,\leq\, \|\btau_{f,0}\|^2_{\bbX_f} \quad \forall\,\btau_f = \btau_{f,0} + \ell\,\bI\in \bbX_f,
\end{equation}
where $\btau_{f,0}\in \bbH_0(\bdiv;\Omega_f) := \Big\{ \btau_{f}\in
\bbH(\bdiv;\Omega_f) :\quad (\tr(\btau_f),1)_{\Omega_f} = 0 \Big\}$
and $\ell\in \R$. We emphasize that \eqref{eq:tau-H0div-Xf-inequality}
holds since each $\btau_f \in \bbX_f$ satisfies the boundary condition
$\ds \btau_f\bn_f = \0$ on $\Gamma^\rN_f$ with $|\Gamma^\rN_f| > 0$.

\subsection{The resolvent  system}

Now, we proceed to analyze the solvability of
\eqref{eq:evolution-problem-in-operator-form} (equivalently
\eqref{eq:Stokes-Biot-formulation}).
First, recalling the definition of the operators $\cA, \cB_1, \cB, \cC$, and $\cE$ (cf. \eqref{eq:operators-A-B1-C}, \eqref{eq:operator-B} and \eqref{eq:time-operator-E}), we note that problem \eqref{eq:evolution-problem-in-operator-form} can be written in the form of \eqref{eq:general-parabolic-problem} with
\begin{equation}\label{eq:E-u-N-M-f}
E = \bX\times \bY\times \bZ,\quad
u = \left( \begin{array}{lll}
\ubsi \\ \ubvarphi \\ \ubu 
\end{array}\right),\quad
\cN = \left( \begin{array}{lll}
\cE & \0 & \0 \\
\0  & \0 & \0 \\
\0  & \0 & \0
\end{array}\right),\quad 
\cM = \left( \begin{array}{lll}
\cA & \cB'_1 & \cB' \\
-\cB_1 & \cC & \0 \\
-\cB & \0 & \0
\end{array}\right),\quad
f = \left( \begin{array}{lll}
\bF \\ \0 \\ \bG
\end{array}\right).
\end{equation}
In addition, the norm induced by the operator $\cE$ is
$|\ubtau|^2_{\cE} := s_0\,\|w_p\|^2_{\L^2(\Omega_p)} + \|A^{1/2}(\btau_p + \alpha_p\,w_p\,\bI)\|^2_{\bbL^2(\Omega_p)}$,
which is equivalent to $\|\btau_p\|^2_{\bbL^2(\Omega_p)} + \|w_p\|^2_{\L^2(\Omega_p)}$ since $s_0>0$.
We denote by $\bbX_{p,2}$ and $\W_{p,2}$ the closures of the spaces $\bbX_p$ and $\W_p$, respectively, with respect to the norms $\|\btau_p\|_{\bbX_{p,2}} := \|\btau_p\|_{\bbL^2(\Omega_p)}$ and $\|w_p\|_{\W_{p,2}} := \|w_p\|_{\L^2(\Omega_p)}$.
Note that $\bbX'_{p,2} = \bbL^2(\Omega_p)$ and $\W'_{p,2} = \W'_p$.
Next, denoting $\bX'_{2,0} := \0\times \0\times \bbX'_{p,2}\times \W'_{p,2}$, $\bY'_{2,0} := \0\times \0\times \0$, and $\bZ'_{2,0} := \0\times \0\times \0\times \0$,
the Hilbert space $E'_b$ and domain $\cD$ in Theorem~\ref{thm:solvability-parabolic-problem} for our context are 
\begin{equation}\label{eq:Eb-prima-D}
E'_b := \bX'_{2,0}\times \bY'_{2,0}\times \bZ'_{2,0},\quad
\cD := \Big\{ (\ubsi,\ubvarphi,\ubu)\in \bX\times \bY\times \bZ :\quad \cM(\ubsi,\ubvarphi,\ubu)\in E'_b \Big\}\,.
\end{equation}

\begin{rem}
The above definition of the space $E'_b$ and the corresponding domain $\cD$ implies that, in order to apply Theorem~\ref{thm:solvability-parabolic-problem} for our problem \eqref{eq:evolution-problem-in-operator-form}, we need to restrict $\f_f = \0, q_f = 0$, and $\f_p = \0$.
To avoid this restriction we will employ a translation argument \cite{s2010} to reduce the existence for \eqref{eq:evolution-problem-in-operator-form} to existence for the following initial-value problem:
Given initial data $(\wh{\ubsi}_0,\wh{\ubvarphi}_0,\wh{\ubu}_0)\in \cD$ and
source terms $(\wh{\f}_p,\wh{q}_p):[0,T]\to \bbX'_{p,2}\times \W'_{p,2}$, find $(\wh\ubsi,\wh\ubvarphi,\wh\ubu)\in [0,T]\to \bX\times \bY\times \bZ$ such that $(\wh\bsi_p(0),\wh p_p(0)) = (\wh{\bsi}_{p,0},\wh{p}_{p,0})$ and, for a.e. $t\in (0,T)$,
\begin{equation}\label{eq:auxiliary-evolution-problem-in-operator-form}
\begin{array}{lllll}
\ds \frac{\partial}{\partial t}\,\cE(\wh\ubsi(t)) + \cA(\wh\ubsi(t)) + \cB'_1(\wh\ubvarphi(t)) + \cB'(\wh\ubu(t)) & = & \wh{\bF}(t) & \mbox{ in } & \bX'_{2,0}, \\ [2ex]
\ds -\,\cB_1(\wh\ubsi(t)) + \cC(\wh\ubvarphi(t)) & = & \0 & \mbox{ in } & \bY'_{2,0}, \\ [2ex]
\ds -\,\cB\,(\wh\ubsi(t)) & = & \0 & \mbox{ in } & \bZ'_{2,0},
\end{array}
\end{equation}
where $\wh{\bF} = (\0,\0,\wh{\f}_p,\wh{q}_p)^\rt$.
\end{rem}

In order to apply Theorem~\ref{thm:solvability-parabolic-problem} for problem \eqref{eq:auxiliary-evolution-problem-in-operator-form}, we need to:
(1) establish the required properties of the operators $\cN$ and $\cM$,
(2) prove the range condition $Rg(\cN + \cM) = E'_b$, and
(3) construct compatible initial data $(\wh{\ubsi}_0,\wh{\ubvarphi}_0,\wh{\ubu}_0)\in \cD$.
We proceed with a sequence of lemmas establishing these results.

\begin{lem}\label{lem:N-M-properties}
  The linear operators $\cN$ and $\cM$ defined in \eqref{eq:E-u-N-M-f} are
  continuous and monotone. In addition, $\cN$ is symmetric.
\end{lem}
\begin{proof}
First, from the definition of the operators $\cE, \cA, \cB_1, \cC$ and
$\cB$ (cf. \eqref{eq:operators-A-B1-C}, \eqref{eq:operator-B},
\eqref{eq:time-operator-E}) it is clear that both $\cN$ and $\cM$
(cf. \eqref{eq:E-u-N-M-f}) are linear and continuous, using
the trace inequalities \eqref{trace-u}--\eqref{trace-sigma} for the
continuity of $\cB_1$. In turn, $\cN$
is symmetric since $\cE$ is.  Finally, using
\eqref{eq:K-uniform-bound}, we have
\begin{equation}\label{eq:monotonicity-E-A}
\begin{array}{l}
\ds \cE(\ubtau)(\ubtau) = s_0\|w_p\|^2_{\L^2(\Omega_p)} + \|A^{1/2}(\btau_p + \alpha_p w_p \bI)\|^2_{\bbL^2(\Omega_p)},\\[2ex]
\ds \cA(\ubtau)(\ubtau) \geq \frac{1}{2\,\mu}\,\|\btau^\rd_f\|^2_{\bbL^2(\Omega_f)} 
+ \mu\,k^{-1}_{\max}\|\bv_p\|^2_{\bL^2(\Omega_p)} \quad \forall\,\ubtau\in \bX,
\end{array}
\end{equation}
and recalling the definition of the operator $\cC$
(cf. \eqref{eq:bilinear-forms-2}, \eqref{eq:operators-A-B1-C}), we obtain 
\begin{equation}\label{eq:pos-sem-def-c}
\cC(\ubpsi)(\ubpsi) = \mu\,\alpha_{\BJS}\sum^{n-1}_{j=1} \pil\sqrt{\bK_j^{-1}} (\bpsi - \bphi)\cdot\bt_{f,j}, (\bpsi - \bphi)\cdot\bt_{f,j}\pir_{\Gamma_{fp}} \,\geq\,
\frac{\mu\,\alpha_{\BJS}}{\sqrt{k_{\max}}}\,|\bpsi - \bphi|^2_{\BJS}\,,
\end{equation}
for all $\ubpsi =(\bpsi,\bphi,\xi)\in \bY$, where $|\bpsi - \bphi|_{\BJS}^2 := \sum^{n-1}_{j=1} \|(\bpsi - \bphi)\cdot\bt_{f,j}\|^2_{\L^2(\Gamma_{fp})}$.
Thus, combining \eqref{eq:monotonicity-E-A} and \eqref{eq:pos-sem-def-c}, and the fact that the operators $\cE, \cA, \cC$ are linear, we deduce the monotonicity of the operators $\cN$ and $\cM$ completing the proof.
\end{proof}

Next, we establish the range condition $Rg(\cN + \cM) = E'_b$, which is done by solving the related resolvent system.
In fact, we will show a stronger result by considering a resolvent system where all source terms in $\bF$ and $\bG$ may be non-zero.
This stronger result will be used in the translation argument for proving existence of the original problem \eqref{eq:evolution-problem-in-operator-form}.
More precisely, let
\begin{equation*}
\bX_2 := \bbX_f\times \bV_p\times \bbX_{p,2}\times \W_{p,2} \supset \bX
\end{equation*}
and note that $\bX'_2 = \bbX'_f\times \bV'_p\times
\bbX'_{p,2}\times \W'_{p,2} \subset \bX'$.
We consider the following resolvent system:
\begin{equation}\label{eq:resolvent-problem-hat}
\begin{array}{lllll}
\ds (\cE + \cA)(\ubsi) + \cB'_1(\ubvarphi) + \cB'(\ubu) & = & \wh{\bF} & \mbox{ in } & \bX'_2, \\ [1ex]
\ds -\,\cB_1(\ubsi) + \cC(\ubvarphi) & = & \0 & \mbox{ in } & \bY', \\ [1ex]
\ds -\,\cB\,(\ubsi) & = & \wh{\bG} & \mbox{ in } & \bZ',
\end{array}  
\end{equation}
where $\wh{\bF}\in \bX'_2$ and $\wh{\bG}\in \bZ'$ are such that
\begin{equation*}
\begin{array}{l}
\ds \wh{\bF}(\ubtau) := (\wh{\f}_{\bsi_f},\btau_f)_{\Omega_f} + (\wh{\f}_{\bu_p},\bv_p)_{\Omega_p}  + (\wh{\f}_p,\btau_p)_{\Omega_p} + (\wh{q}_p,w_p)_{\Omega_p}\,, \\[2ex]
\ds \wh{\bG}(\ubv) := (\wh{\f}_{\bu_f},\bv_f)_{\Omega_f} + (\wh{\f}_{\bu_s},\bv_s)_{\Omega_p} + (\wh{\f}_{\bgamma_f},\bchi_f)_{\Omega_f} + (\wh{\f}_{\bgamma_p},\bchi_p)_{\Omega_p}\,.
\end{array}
\end{equation*}
We next focus on proving that the resolvent system \eqref{eq:resolvent-problem-hat} is well-posed. We start with the following preliminary lemma.

\begin{lem}\label{lem:equivalent-resolvent-system}
Let $(\ubsi,\ubvarphi,\ubu)\in \bX\times \bY\times \bZ$ be a solution to \eqref{eq:resolvent-problem-hat}.	
Then, for any positive constant $\kappa$, it satisfies
\begin{equation}\label{eq:aux-resolvent-system}
\begin{array}{llcll}
\ds (\cE + \wt{\cA})(\ubsi) + \cB'_1(\ubvarphi) + \cB'(\ubu) & = & \wt{\bF} & \mbox{ in } & \bX'_2, \\ [1ex]
\ds \cB_1(\ubsi) - \cC(\ubvarphi) & = & \0 & \mbox{ in } & \bY', \\ [1ex]
\ds \cB\,(\ubsi) & = & -\,\wh{\bG} & \mbox{ in } & \bZ',
\end{array}
\end{equation}
where
\begin{equation}\label{eq:a-tilde}
\wt{\cA}(\ubsi)(\ubtau) 
\,:=\,  \cA(\ubsi)(\ubtau) + \kappa\,\Big\{ (\div(\bu_p),\div(\bv_p))_{\Omega_p} + \big(s_0\,p_p + \alpha_p\,\tr\big(A(\bsi_p + \alpha_p\,p_p\,\bI)\big),\div(\bv_p)\big)_{\Omega_p} \Big\},
\end{equation}
and
\begin{equation*}
\wt{\bF}(\ubtau) := \wh{\bF}(\ubtau) + \kappa\,\big(\wh{q}_p,\div(\bv_p)\big)_{\Omega_p} .
\end{equation*}
Conversely, if $(\ubsi,\ubvarphi,\ubu)\in \bX\times \bY\times \bZ$
is a solution to \eqref{eq:aux-resolvent-system},
then it is also a solution to \eqref{eq:resolvent-problem-hat}.
\end{lem}
\begin{proof}
Let $(\ubsi,\ubvarphi,\ubu)\in \bX\times \bY\times \bZ$ be a solution to \eqref{eq:resolvent-problem-hat}. Using that $\div \, \bV_p = \W_p$, we take $\ubtau = (\0,w_p) = (\0,\div(\bv_p))\in \bX$ in the first row of \eqref{eq:resolvent-problem-hat}, multiply by a positive constant $\kappa$ and add that term to \eqref{eq:resolvent-problem-hat}, to obtain \eqref{eq:aux-resolvent-system}.
Conversely, if $(\ubsi,\ubvarphi,\ubu)\in \bX\times \bY\times \bZ$
satisfies \eqref{eq:aux-resolvent-system} we employ similar arguments,
but now subtracting, to recover \eqref{eq:resolvent-problem-hat}.
\end{proof}

Problem \eqref{eq:aux-resolvent-system} has the same structure as the one in
Theorem~\ref{thm:system_solvability}.
Therefore, in what follows we apply this result to establish the
well-posedness of \eqref{eq:aux-resolvent-system}.
To that end, we first observe that the kernel of the operator $\cB$,
cf. \eqref{eq:operator-B}, can be written as
\begin{equation}\label{eq:kernel-B}
\bV \,:=\, \Big\{ \ubtau\in \bX :\quad \cB(\ubtau)(\ubv) = 0 \quad \forall\,\bv\in \bZ \Big\} 
\,=\, \wt{\bbX}_{f}\times \bV_p\times \wt{\bbX}_p\times \W_p
\end{equation}
where
\begin{equation*}
\wt{\bbX}_{\star} := \Big\{ \btau_\star\in \bbX_{\star} :\quad \btau_\star = \btau^\rt_\star \qan \bdiv(\btau_\star) = \0 \qin \Omega_\star \Big\},\quad \star\in \{f,p\}.
\end{equation*}
We next verify the hypotheses of Theorem~\ref{thm:system_solvability}.
We begin by noting that the operators $\wt{\cA}, \cB_1, \cC, \cB$, and $\cE$
are linear and continuous. Next, we proceed with the ellipticity of the operator
$\cE + \wt{\cA}$ on $\bV$.

\begin{lem}\label{lem:ellipticity-tilde-a}
Assume that 
\begin{equation*}
\kappa\in \left(0, 2\,\min\left\{\delta_1, \frac{\delta_2}{\alpha_p} \right\}\right) \quad \mbox{with}\quad 
\delta_1\in \left(0, \frac{2}{s_0}\right) \qan 
\delta_2\in \left(0, \frac{4\mu_{min}}{n\,\alpha_p}\,
\left(1 - \frac{s_0}{2}\,\delta_1\right) \right).
\end{equation*}
Then, the operator $\cE + \wt{\cA}$ is elliptic on $\bV$.
\end{lem}
\begin{proof}
  From the definition of $\wt{\cA}$, cf. \eqref{eq:a-tilde}, and considering
  $\ubtau\in \bV$ we get
\begin{equation*}
\begin{array}{l}
  \ds (\cE + \wt{\cA})(\ubtau)(\ubtau) \,=\,
  \frac{1}{2\,\mu}\,\|\btau^\rd_f\|^2_{\bbL^2(\Omega_f)}
  + \mu\|\bK^{-1/2}\bv_p\|^2_{\bL^2(\Omega_p)}
  + s_0\,\|w_p\|^2_{\W_p} + \|A^{1/2}(\btau_p
  + \alpha_p\,w_p\,\bI)\|^2_{\bbL^2(\Omega_p)} \\ [2ex]
  \ds\quad +\,\,\kappa\,\|\div(\bv_p)\|^2_{\L^2(\Omega_p)}
  + s_0\,\kappa\,(w_p,\div(\bv_p))_{\Omega_p}
  + \alpha_p\,\kappa\,(A^{1/2}(\btau_p
  + \alpha_p\,w_p\,\bI), A^{1/2}(\div(\bv_p)\,\bI))_{\Omega_p}.
\end{array}
\end{equation*}
Hence, using the Cauchy--Schwarz and Young's inequalities, \eqref{eq:K-uniform-bound},
\eqref{A-bounds},
and \eqref{eq:tau-d-H0div-inequality}--\eqref{eq:tau-H0div-Xf-inequality}, we obtain
\begin{equation*}
\begin{array}{l}
\ds (\cE + \wt{\cA})(\ubtau)(\ubtau) 
\geq \frac{C_{\rd}}{2\,\mu} \|\btau_f\|^2_{\bbX_f}
+ \mu\, k_{\max}^{-1}\|\bv_p\|^2_{\bL^2(\Omega_p)} 
+ \kappa \left(\left(1{-}\frac{s_0}{2} \delta_1\right)
- \frac{n\,\alpha_p}{4\mu_{\min}} \delta_2\right)
\|\div(\bv_p)\|^2_{\L^2(\Omega_p)} \\ [2ex]
\ds\qquad +\,\left(1 - \frac{\alpha_p}{2\,\delta_2}\,\kappa\right)
\|A^{1/2}\,(\btau_p + \alpha_p\,w_p\,\bI)\|^2_{\bbL^2(\Omega_p)}
+ s_0\,\left(1 - \frac{\kappa}{2\,\delta_1}\right) \|w_p\|^2_{\W_p},
\end{array}
\end{equation*}
where $C_{\rd} := C_1(\Omega_f)\,C_2(\Omega_f)$.
Then, using the stipulated hypotheses on $\delta_1, \delta_2$ and $\kappa$,
we can define the positive constants
\begin{equation*}
\begin{array}{c}
\ds \alpha_1(\Omega_f) := \frac{C_\rd}{2\,\mu},\quad
\alpha_2(\Omega_p) := \min\left\{ \mu\,k_{\max}^{-1},
\kappa\,\left(\left(1 - \frac{s_0}{2}\,\delta_1\right)
- \frac{n\,\alpha_p}{4\mu_{\min}}\,\delta_2 \right)\right\},\\ [3ex]
\ds \alpha_3(\Omega_p) := \frac{s_0}{2}\,\left(1 - \frac{\kappa}{2\,\delta_1}\right),\quad
\alpha_4(\Omega_p) := \min\left\{ \left(1 - \frac{\alpha_p}{2\,\delta_2}\,\kappa\right), \alpha_3(\Omega_p) \right\}
\end{array}
\end{equation*}
which allow us to obtain
\begin{equation}\label{eq:ellipticity-1}
\begin{array}{l}
  \ds (\cE + \wt{\cA})(\ubtau)(\ubtau) \,\geq\, \alpha_1(\Omega_f)\,\|\btau_f\|^2_{\bbX_f}
  + \alpha_2(\Omega_p)\,\|\bv_p\|^2_{\bV_p} + \alpha_3(\Omega_p)\,\|w_p\|^2_{\W_p} \\ [2ex] 
\ds\qquad +\,\,\alpha_4(\Omega_p)\,\left(\|A^{1/2}(\btau_p + \alpha_p\,w_p\,\bI)\|^2_{\bbL^2(\Omega_p)} + \|w_p\|^2_{\W_p}\right).
\end{array}
\end{equation}
In turn, from \eqref{A-bounds}
and using the triangle inequality, we deduce
\begin{equation}\label{eq:aux-taup-qp-inequality}
\begin{array}{l}
\ds \|\btau_p\|^2_{\bbL^2(\Omega_p)} \,\leq\, (2\,\mu_{\max} + n\,\lambda_{\max})\,
\left(\|A^{1/2}(\btau_p + \alpha_p\,w_p\,\bI)\|^2_{\bbL^2(\Omega_p)} + \|A^{1/2}(\alpha_p\,w_p\,\bI)\|^2_{\bbL^2(\Omega_p)}\right) \\ [2ex]
\ds\qquad \,\leq\, C_p\,\left(\|A^{1/2}(\btau_p + \alpha_p\,w_p\,\bI)\|^2_{\bbL^2(\Omega_p)} + \|w_p\|^2_{\W_p}\right),  
\end{array}
\end{equation}
where $C_p := (2\,\mu_{\max} + n\,\lambda_{\max})
\max\Big\{1, \frac{n\,\alpha^2_p}{2\mu_{\min}} \Big\}$.
A combination of \eqref{eq:ellipticity-1} and \eqref{eq:aux-taup-qp-inequality},
and the fact that $\bdiv(\btau_p) = \0$ in $\Omega_p$, implies
\begin{equation*}
(\cE + \wt{\cA})(\ubtau)(\ubtau) 
\,\geq\, \alpha(\Omega_f,\Omega_p)\,\|\ubtau\|^2_{\bX} \quad \forall\,\ubtau\in \bV,
\end{equation*}
with $\alpha(\Omega_f,\Omega_p) := \min\big\{ \alpha_1(\Omega_f), \alpha_2(\Omega_p),
\alpha_3(\Omega_p), \alpha_4(\Omega_p)/C_p \big\}$,
hence $\cE + \wt{\cA}$ is elliptic on $\bV$.
\end{proof}

\begin{rem}
  To maximize the ellipticity constant $\alpha(\Omega_f,\Omega_p)$,
  we can choose explicitly the parameter $\kappa$ by taking the
  parameters $\delta_1$ and $\delta_2$ as the middle points of their feasible ranges.
More precisely, we can simply take
\begin{equation*}
\delta_1 = \frac{1}{s_0},\quad
\delta_2 = \frac{\mu_{min}}{n\,\alpha_p},\quad
\kappa = \min\left\{ \frac{1}{s_0}, \frac{\mu_{\min}}{n\,\alpha^2_p} \right\}.
\end{equation*}
\end{rem}

We continue with the verification of the hypotheses of
Theorem~\ref{thm:system_solvability}.

\begin{lem}\label{lem:semipositive-c-inf-sup-b-B}
There exist positive constants $\beta_1$ and $\beta$, such that
\begin{equation}\label{eq:inf-sup-b1}
\sup_{\0\neq \ubtau\in \bV} \frac{\cB_1(\ubtau)(\ubpsi)}{\|\ubtau\|_\bX} 
\,\geq\, \beta_1\,\|\ubpsi\|_{\bY} \quad \forall\,\ubpsi\in \bY,
\end{equation}
and
\begin{equation}\label{eq:inf-sup-B}
\sup_{\0\neq \ubtau\in \bX } \frac{\cB(\ubtau)(\ubv)}{\|\ubtau\|_\bX} 
\,\geq\, \beta\,\|\ubv\|_\bZ \quad \forall\, \ubv \in \bZ.
\end{equation}
\end{lem}
\begin{proof}
We begin with the proof of \eqref{eq:inf-sup-b1}. Due the diagonal character of
operator $\cB_1$, cf. \eqref{eq:operators-A-B1-C}, we need to show individual inf-sup
conditions for $b_{\bn_f}$, $b_{\bn_p}$, and $b_{\Gamma}$. The inf-sup condition for
$b_{\Gamma}$ follows from a slight adaptation of the argument in \cite[Lemma~3.2]{ejs2009}
to account for the presence of Dirichlet boundary $\Gamma_p^\rD$, using that
$\dist(\Gamma^\rD_p,\Gamma_{fp}) \geq s > 0$. The inf-sup conditions for
$b_{\bn_f}$ and $b_{\bn_p}$ follow in a similar way. Since the kernel space $\bV$
consists of symmetric and divergence-free tensors, the argument
in \cite[Lemma~3.2]{ejs2009} must be modified to account for that. For example, in
$\Omega_f$ we solve a problem
\begin{equation}\label{elast-aux}
\bdiv(\be(\bv_f)) = \0 \ \text{ in } \ \Omega_f, \quad \be(\bv_f)\, \bn_f = \bxi
\ \text{ on } \ \Gamma_{fp} \cup \Gamma_f^\rN, \quad \bv_f = \0 \ \text{ on } \ \Gamma_f^\rD,
\end{equation}
for given data $\bxi \in \bH^{-1/2}(\Gamma_{fp} \cup \Gamma_f^\rN)$
such that $\bxi = \0$ on $\Gamma_f^\rN$. We recall that $\Gamma_f^\rN$
is adjacent to $\Gamma_{fp}$. Furthermore,
$|\Gamma_f^\rD| > 0$, which guarantees the solvability of the problem. We refer to
\cite[Lemma~3.2]{ejs2009} for further details.

Finally, proceeding as above, using the diagonal character of operator $\cB$, cf. \eqref{eq:operator-B}, and employing the theory developed in
\cite[Section~2.4.3]{Gatica} to our context, we can deduce
\eqref{eq:inf-sup-B}.
\end{proof}

Now, we are in a position to establish that the resolvent system associated to \eqref{eq:auxiliary-evolution-problem-in-operator-form} is
well-posed. 
\begin{lem}\label{lem:resolvent-system}
For $\cN, \cM$ and $E'_b$ defined in \eqref{eq:E-u-N-M-f}--\eqref{eq:Eb-prima-D}, it holds that $Rg(\cN + \cM) = E'_b$, that is, given $f\in E'_b$, there exists $v\in \cD$ such that $(\cN + \cM)(v) = f$.
\end{lem}
\begin{proof}
Let us consider $\wh{\bF} = (\0,\0,\wh{\f}_p,\wh{q}_p)^\rt$ and $\wh{\bG}=\0$ in \eqref{eq:resolvent-problem-hat}--\eqref{eq:aux-resolvent-system} and $\kappa$ as in Lemma~\ref{lem:ellipticity-tilde-a}.
The well-posedness of \eqref{eq:aux-resolvent-system} follows from \eqref{eq:pos-sem-def-c}, Lemmas~\ref{lem:ellipticity-tilde-a} and \ref{lem:semipositive-c-inf-sup-b-B}, and a straightforward application of Theorem~\ref{thm:system_solvability} with $A = \cE + \wt{\cA}, B_1 = \cB_1, S = \cC$, and $B = \cB$. Then, employing Lemma~\ref{lem:equivalent-resolvent-system} we conclude that there exists a unique solution of the resolvent system of \eqref{eq:auxiliary-evolution-problem-in-operator-form}, implying the range condition.
\end{proof}

We are now ready to establish existence for the auxiliary initial value problem \eqref{eq:auxiliary-evolution-problem-in-operator-form}, assuming compatible initial data.
\begin{lem}\label{lem:auxiliar-well-posedness-result}
  For each compatible initial data $(\wh{\ubsi}_0,\wh{\ubvarphi}_0,\wh{\ubu}_0)\in \cD$ and each $(\wh{\f}_p,\wh{q}_p)\in \W^{1,1}(0,T;\bbX'_{p,2})\times \W^{1,1}(0,T;\W'_{p,2})$, the problem \eqref{eq:auxiliary-evolution-problem-in-operator-form} has a solution $(\wh\ubsi,\wh\ubvarphi,\wh\ubu):[0,T]\to \bX \times \bY \times \bZ$ such that 
  $(\wh\bsi_p,\wh p_p) \in \W^{1,\infty}(0,T;\bbL^2(\Omega_p))\times \W^{1,\infty}(0,T;\W_p)$ and $(\wh\bsi_p(0),\wh p_p(0)) = (\wh{\bsi}_{p,0},\wh{p}_{p,0})$.
\end{lem}
\begin{proof}
The assertion of the lemma follows by applying Theorem~\ref{thm:solvability-parabolic-problem} with $E, \cN, \cM$ defined in \eqref{eq:E-u-N-M-f}, using Lemmas~\ref{lem:N-M-properties} and \ref{lem:resolvent-system}.
\end{proof}

We will employ Lemma~\ref{lem:auxiliar-well-posedness-result} to obtain existence of a solution to our problem \eqref{eq:evolution-problem-in-operator-form}. 
To that end, we first construct compatible initial data $(\ubsi_0,\ubvarphi_0,\ubu_0)$.
\begin{lem}\label{lem:sol0-in-M-operator}
Assume that the initial data $p_{p,0}\in \W_p\cap \H$, where
\begin{equation}\label{eq:H-space-initial-condition}
\H \,:=\, \Big\{ w_p\in \H^1(\Omega_p) :\quad \bK\,\nabla\,w_p\in \bH^1(\Omega_p),\,\,  
\bK\,\nabla\,w_p\cdot\bn_p = 0 \mbox{ on } \Gamma^{\rN}_p,\,\,   
\ w_p = 0 \mbox{ on } \Gamma^{\rD}_p \Big\}. 
\end{equation}
Then, there exist $\ubsi_0 := (\bsi_{f,0}, \bu_{p,0},\bsi_{p,0},p_{p,0})\in \bX$, $\ubvarphi_0 := (\bvarphi_0, \btheta_0, \lambda_0)\in \bY$, and $\ubu_0 := (\bu_{f,0}, \bu_{s,0}, \bgamma_{f,0},\newline \bgamma_{p,0})\in \bZ$ such that 
\begin{equation}\label{eq:initial-data-problem}
\begin{array}{lllll}
\ds \cA(\ubsi_0) + \cB'_1(\ubvarphi_0) + \cB'(\ubu_0) & = & \wh{\bF}_0 & \mbox{ in } & \bX'_2, \\ [1.5ex]
\ds -\,\cB_1(\ubsi_0) + \cC(\ubvarphi_0) & = & \0 & \mbox{ in } & \bY', \\ [1.5ex]
\ds -\,\cB\,(\ubsi_0) & = & \bG(0) & \mbox{ in } & \bZ',
\end{array}
\end{equation}
where $\wh{\bF}_0 = (q_f(0),\0,\wh{\f}_{p,0},\wh{q}_{p,0})^\rt\in \bX'_2$, with suitable $(\wh{\f}_{p,0},\wh{q}_{p,0})\in \bbX'_{p,2}\times \W'_{p,2}$.
\end{lem}
\begin{proof}
Following the approach from \cite[Lemma~4.15]{aeny2019}, the initial data is constructed by 
solving a sequence of well-defined subproblems. We take the following steps.	
	
1. Define $\bu_{p,0} := -\dfrac{1}{\mu}\,\bK\nabla p_{p,0}$,
with $p_{p,0}\in \H$, cf. \eqref{eq:H-space-initial-condition}. It follows that
$\bu_{p,0} \in \bH(\div;\Omega_p)$ and 
\begin{equation}\label{eq:sol0-up0-pp0}
\mu\,\bK^{-1}\bu_{p,0} = -\nabla p_{p,0},\quad 
\div(\bu_{p,0}) = -\frac{1}{\mu}\,\div(\bK\nabla p_{p,0}) \qin \Omega_p,\quad
\bu_{p,0}\cdot\bn_p = 0 \qon \Gamma^{\rN}_p.
\end{equation}
Next, defining $\lambda_0 := p_{p,0}|_{\Gamma_{fp}}\in \Lambda_p$, \eqref{eq:sol0-up0-pp0}
implies
\begin{equation}\label{eq:system-sol0-1}
a_p(\bu_{p,0},\bv_p) + b_p(\bv_p,p_{p,0}) + b_{\Gamma}(\bv_p,\lambda_0)
  = 0 \quad \forall\,\bv_p\in \bV_p.
\end{equation}
	
2. Define $(\bsi_{f,0},\bvarphi_0,\bu_{f,0},\bgamma_{f,0})\in \bbX_{f}\times \bLambda_f\times \bV_f\times \bbQ_f$ as the unique solution of the problem
\begin{equation}\label{eq:system-sol0-2}
\begin{array}{l}
a_f(\bsi_{f,0},\btau_f) + b_{\bn_f}(\btau_f,\bvarphi_0)
+ b_f(\btau_f,\bu_{f,0}) + b_{\skf}(\btau_f,\bgamma_{f,0})
= \ds -\frac{1}{n}\,(q_f(0)\,\bI,\btau_f)_{\Omega_f}, \\ [1ex]
\ds -b_{\bn_f}(\bsi_{f,0},\bpsi) 
	= - \mu\,\alpha_{\BJS}\sum^{n-1}_{j=1}
	\pil\sqrt{\bK_j^{-1}}\bu_{p,0}\cdot\bt_{f,j},\bpsi\cdot\bt_{f,j}\pir_{\Gamma_{fp}}
	- \pil\bpsi\cdot\bn_f,\lambda_0\pir_{\Gamma_{fp}}, \\ [2ex]
-b_f(\bsi_{f,0},\bv_f) - b_{\skf}(\bsi_{f,0},\bchi_f)
=  (\f_f(0),\bv_f)_{\Omega_f}
\end{array}
\end{equation}
for all $(\btau_f,\bpsi,\bv_f,\bchi_f)\in \bbX_f\times
\bLambda_f\times \bV_f\times \bbQ_f$.  
Note that \eqref{eq:system-sol0-2} is well-posed, since it corresponds 
to the weak solution of the Stokes problem in a mixed formulation and 
its solvability can be shown using classical Babu{\v s}ka-Brezzi theory.  
Note also that $\bu_{p,0}$ and $\lambda_0$ are data for this problem.

3. Define $(\bsi_{p,0}, \bomega_0, \bbeta_{p,0}, \brho_{p,0})\in \bbX_p\times \bLambda_s\times \bV_s\times \bbQ_p$, as the unique solution of the problem
\begin{equation}\label{eq:system-sol0-3}
\begin{array}{l}
\ds (A(\bsi_{p,0}),\btau_p)_{\Omega_p} + b_{\bn_p}(\btau_p,\bomega_0) + b_s(\btau_p,\bbeta_{p,0}) + b_{\skp}(\btau_p,\brho_{p,0}) \,=\, - (A(\alpha_p\,p_{p,0}\,\bI),\btau_p)_{\Omega_p} \\ [1ex]
\ds -b_{\bn_p}(\bsi_{p,0},\bphi) \,=\,  
\mu\,\alpha_{\BJS}\sum^{n-1}_{j=1} \pil\sqrt{\bK_j^{-1}}\bu_{p,0}\cdot\bt_{f,j},\bphi\cdot\bt_{f,j}\pir_{\Gamma_{fp}} 
- \pil\bphi\cdot\bn_p,\lambda_0\pir_{\Gamma_{fp}} \\ [2ex]
-b_s(\bsi_{p,0},\bv_s) - b_{\skp}(\bsi_{p,0},\bchi_p) \,=\, (\f_{p}(0),\bv_s)_{\Omega_p},
\end{array}
\end{equation}
for all $(\btau_p,\bphi,\bv_s,\bchi_p)\in \bbX_p\times \bLambda_s\times \bV_s\times \bbQ_p$.
Problem \eqref{eq:system-sol0-3} corresponds to the weak solution of the elasticity
problem in a mixed formulation and its solvability can be shown using
classical Babu{\v s}ka-Brezzi theory.
Note that $p_{p,0}, \bu_{p,0}$, and $\lambda_0$ are data for this problem.
Here $\bbeta_{p,0}, \brho_{p,0}$, and $\bomega_0$ are auxiliary variables that are not part of the constructed initial data.
However, they can be used to recover the variables $\bbeta_p, \brho_p$, and $\bomega$ that satisfy the non-differentiated equation \eqref{eq:A-sigmap-pp-without-time-derivative}.

4. Define $\btheta_0\in \bLambda_s$ as
\begin{equation}\label{eq:definition-theta0}
\btheta_0 \,:=\, \bvarphi_0 - \bu_{p,0} \qon \Gamma_{fp},
\end{equation}
where $\bvarphi_0$ and $\bu_{p,0}$ are data obtained in the previous steps.
Note that \eqref{eq:definition-theta0} implies that the BJS terms in \eqref{eq:system-sol0-2} and \eqref{eq:system-sol0-3} can be rewritten with $\bu_{p,0}\cdot \bt_{f,j} = (\bvarphi_0 - \btheta_0)\cdot \bt_{f,j}$ and that the ninth equation in \eqref{eq:Stokes-Biot-formulation-1} holds for the initial data, that is,
\begin{equation}\label{eq:system-sol0-4}
-\pil\bvarphi_0\cdot\bn_f + (\btheta_0 + \bu_{p,0})\cdot\bn_p,\xi\pir_{\Gamma_{fp}} \,=\, 0 \quad \forall\,\xi\in \Lambda_p.
\end{equation}

5. Finally, define $(\wh\bsi_{p,0},\bu_{s,0}, \bgamma_{p,0}) \in
\bbX_p\times \bV_s\times \bbQ_p$, as the unique solution of the problem
\begin{equation}\label{eq:system-sol0-5}
\begin{array}{l}
\ds (A(\wh\bsi_{p,0}),\btau_p)_{\Omega_p} + b_s(\btau_p,\bu_{s,0})
+ b_{\skp}(\btau_p,\bgamma_{p,0}) \,=\, - b_{\bn_p}(\btau_p,\btheta_0) \\ [2ex]
-b_s(\wh\bsi_{p,0},\bv_s) - b_{\skp}(\wh\bsi_{p,0},\bchi_p) \,=\, 0,
\end{array}
\end{equation}
for all $(\btau_p,\bv_s,\bchi_p)\in \bbX_p \times \bV_s\times \bbQ_p$.
Problem \eqref{eq:system-sol0-5} corresponds to the weak solution of the elasticity
problem in $\Omega_p$ with Dirichlet datum $\btheta_0$ on $\Gamma_{fp}$.

Combining \eqref{eq:system-sol0-1}, \eqref{eq:system-sol0-2}, the second and
third equations in \eqref{eq:system-sol0-3}, \eqref{eq:system-sol0-4},
and the first equation in \eqref{eq:system-sol0-5}, we obtain $(\ubsi_0, \ubvarphi_0, \ubu_0)\in \bX\times \bY\times \bZ$ satisfying \eqref{eq:initial-data-problem} with
\begin{equation}\label{eq:data-fp0-qp0-hat}
(\wh{\f}_{p,0},\btau_p)_{\Omega_p} \,=\, - (A(\wh\bsi_{p,0}),\btau_p)_{\Omega_p} \qan
(\wh{q}_{p,0},w_p)_{\Omega_p} \,=\, - b_p(\bu_{p,0},w_p).
\end{equation}
The above equations imply
$$
\|\wh{\f}_{p,0}\|_{\bbL^2(\Omega_p)} + \|\wh{q}_{p,0}\|_{\L^2(\Omega_p)}
\le C\,\left( \|\wh{\bsi}_{p,0}\|_{\bbL^2(\Omega_p)}
+ \|\div(\bu_{p,0})\|_{\L^2(\Omega_p)} \right),
$$
hence $(\wh{\f}_{p,0}, \wh{q}_{p,0})\in \bbX'_{p,2}\times \W'_{p,2}$, completing the proof.
\end{proof}

\subsection{The main result}

We are now ready to prove the main result of this section.

\begin{thm}\label{thm:well-posdness-continuous}
For each compatible initial data $(\ubsi_0,\ubvarphi_0,\ubu_0)$ constructed in Lemma~\ref{lem:sol0-in-M-operator} and each
\begin{equation*}
\f_f\in \W^{1,1}(0,T;\bV_f'),\quad \f_p\in \W^{1,1}(0,T;\bV_s'),\quad q_f\in \W^{1,1}(0,T;\bbX'_f),\quad q_p\in \W^{1,1}(0,T;\W'_p),
\end{equation*}
there exists a unique solution of \eqref{eq:evolution-problem-in-operator-form},
$(\ubsi,\ubvarphi,\ubu):[0,T]\to \bX \times \bY \times \bZ$, such that
$(\bsi_p,p_p) \in \W^{1,\infty}(0,T;\bbL^2(\Omega_p))\times \W^{1,\infty}(0,T;\W_p)$
and $(\bsi_p(0), p_p(0)) = (\bsi_{p,0}, p_{p,0})$.
\end{thm}
\begin{proof}
For each fixed time $t\in [0,T]$, Lemma~\ref{lem:resolvent-system} implies that there exists a solution to the resolvent system \eqref{eq:resolvent-problem-hat} with $\wh{\bF} = \bF(t)$ and $\wh{\bG} = \bG(t)$ defined in \eqref{eq:data-F-G}.
More precisely, there exist $(\wt{\ubsi}(t), \wt{\ubvarphi}(t), \wt{\ubu}(t))$ such that
\begin{equation}\label{eq:aux-tilda-problem}
\begin{array}{lllll}
\ds \big( \cE + \cA \big)(\wt{\ubsi}(t)) + \cB'_1(\wt{\ubvarphi}(t)) + \cB'(\wt{\ubu}(t)) & = & \bF(t) & \mbox{ in } & \bX'_{2}, \\ [1.5ex]
\ds -\,\cB_1(\wt{\ubsi}(t)) + \cC(\wt{\ubvarphi}(t)) & = & \0 & \mbox{ in } & \bY', \\ [1.5ex]
\ds -\,\cB\,(\wt{\ubsi}(t)) & = & \bG(t) & \mbox{ in } & \bZ'.
\end{array}
\end{equation}
We look for a solution to \eqref{eq:evolution-problem-in-operator-form} in the form $\ubsi(t)= \wt{\ubsi}(t) + \wh{\ubsi}(t)$, $\ubvarphi(t)= \wt{\ubvarphi}(t) + \wh{\ubvarphi}(t)$, and $\ubu(t) = \wt{\ubu}(t) + \wh{\ubu}(t)$.
Subtracting \eqref{eq:aux-tilda-problem} from \eqref{eq:evolution-problem-in-operator-form} leads to the reduced evolution problem
\begin{equation}\label{eq:aux-hat-problem}
\begin{array}{lllll}
\ds \partial_t \cE(\wh{\ubsi}(t)) + \cA(\wh{\ubsi}(t)) + \cB'_1(\wh{\ubvarphi}(t)) + \cB'(\wh{\ubu}(t)) & = & \cE(\wt{\ubsi}(t)) - \partial_t \cE(\wt{\ubsi}(t)) & \mbox{ in } & \bX'_{2,0}, \\ [1.5ex]
\ds -\,\cB_1(\wh{\ubsi}(t)) + \cC(\wh{\ubvarphi}(t)) & = & \0 & \mbox{ in } & \bY'_{2,0}, \\ [1.5ex]
\ds -\,\cB\,(\wh{\ubsi}(t)) & = & \0 & \mbox{ in } & \bZ'_{2,0},
\end{array}
\end{equation}
with initial condition $\wh{\ubsi}(0) = \ubsi_0 - \wt{\ubsi}(0)$, $\wh{\ubvarphi}(0) = \ubvarphi_0 - \wt{\ubvarphi}(0)$, and $\wh{\ubu}(0) = \ubu_0 - \wt{\ubu}(0)$.
Subtracting \eqref{eq:aux-tilda-problem} at $t=0$ from \eqref{eq:initial-data-problem} gives
\begin{equation}\label{eq:aux-hat-problem-0}
\begin{array}{lllll}
\ds \cA(\wh{\ubsi}(0)) + \cB'_1(\wh{\ubvarphi}(0)) + \cB'(\wh{\ubu}(0)) & = & \cE(\wt{\ubsi}(0)) + \wh{\bF}_0 - \bF(0) & \mbox{ in } & \bX'_{2,0}, \\ [1.5ex]
\ds -\,\cB_1(\wh{\ubsi}(0)) + \cC(\wh{\ubvarphi}(0)) & = & \0 & \mbox{ in } & \bY'_{2,0}, \\ [1.5ex]
\ds -\,\cB\,(\wh{\ubsi}(0)) & = & \0 & \mbox{ in } & \bZ'_{2,0}.
\end{array}
\end{equation}
We emphasize that in \eqref{eq:aux-hat-problem-0}, $\wh{\bF}_0 - \bF(0) = (\0,\0,\wh{\f}_{p,0},\wh{q}_{p,0} - q_p(0))^\rt\in \bX'_{2,0}$.
Thus, $\cM(\wh{\ubsi}(0), \wh{\ubvarphi}(0), \wh{\ubu}(0))\in E'_b$, i.e., $(\wh{\ubsi}(0), \wh{\ubvarphi}(0), \wh{\ubu}(0))\in \cD$ (cf. \eqref{eq:Eb-prima-D}).
Thus, the reduced evolution problem \eqref{eq:aux-hat-problem} is in the form of \eqref{eq:auxiliary-evolution-problem-in-operator-form}.
According to Lemma~\ref{lem:auxiliar-well-posedness-result}, it has a solution, which establishes the existence of a solution to \eqref{eq:evolution-problem-in-operator-form} with the stated regularity satisfying $(\bsi_p(0),p_p(0)) = (\bsi_{p,0},p_{p,0})$.

We next show that the solution of \eqref{eq:evolution-problem-in-operator-form} is unique.
Since the problem is linear, it is sufficient to prove that the problem with zero data has only the zero solution.
Taking $\bF = \bG = \0$ in \eqref{eq:evolution-problem-in-operator-form} and testing it with the solution $(\ubsi,\ubvarphi,\ubu)$ yields
\begin{equation*}
\begin{array}{l}
\ds \frac{1}{2}\,\partial_t\,\left( \|A^{1/2}\,(\bsi_p + \alpha_p\,p_p\,\bI)\|^2_{\bbL^2(\Omega_p)} + s_0\,\|p_p\|^2_{\W_p} \right) 
+ \frac{1}{2\,\mu}\,\|\bsi^{\rd}_f\|^2_{\bbL^2(\Omega_f)} + a_p(\bu_p,\bu_p) + \cC(\ubvarphi)(\ubvarphi) \,=\, 0,
\end{array}
\end{equation*}
which together with \eqref{eq:aux-taup-qp-inequality}, \eqref{eq:K-uniform-bound} to bound $a_p$ (cf. \eqref{eq:bilinear-forms-1}), the semi-definite positive property of $\cC$ (cf. \eqref{eq:pos-sem-def-c}),
integrating in time from $0$ to $t\in (0,T]$, and using that the initial data is zero, implies
\begin{equation}\label{eq:uniqueness-bound-1}
\|\bsi_p\|^2_{\bbL^2(\Omega_p)} + \|p_p\|^2_{\W_p} 
+ \int^t_0 \left(\|\bsi^\rd_f\|^2_{\bbL^2(\Omega_f)} + \|\bu_p\|^2_{\bL^2(\Omega_p)}\right)\, ds \leq 0.
\end{equation} 
It follows from \eqref{eq:uniqueness-bound-1} that $\bsi^{\rd}_f(t) = \0, \bu_p(t) = \0, \bsi_p(t) = \0$, and $p_p(t) = 0$ for all $t\in (0,T]$.

Now, taking $\ubtau\in \bV$ (cf. \eqref{eq:kernel-B}) in the first equation of \eqref{eq:evolution-problem-in-operator-form} and employing the inf-sup condition of $\cB_1$ (cf. \eqref{eq:inf-sup-b1}), with $\ubpsi = \ubvarphi = (\bvarphi, \btheta, \lambda)\in \bY$, yields
\begin{equation*}
\wt{\beta}\,\|\ubvarphi\|_{\bY} 
\,\leq\, \sup_{\0\neq \ubtau\in \bV} \frac{\cB_1(\ubtau)(\ubvarphi)}{\|\ubtau\|_{\bX}} 
= -\,\sup_{\0\neq \ubtau\in \bV} \frac{(\partial_t\,\cE + \cA)(\ubsi)(\ubtau)}{\|\ubtau\|_{\bX}} = 0.
\end{equation*}
Thus, $\bvarphi(t) = \0, \btheta(t) = \0$, and $\lambda(t) = 0$ for all $t\in (0,T]$. 
In turn, from the inf-sup condition of $\cB$ (cf. \eqref{eq:inf-sup-B}), with $\ubv = \ubu = (\bu_f, \bu_s,\bgamma_f, \bgamma_p)\in \bZ$, we get
\begin{equation*}
\beta\,\|\ubu\|_{\bZ} 
\leq \sup_{\0\neq \ubtau\in \bX} \frac{\cB(\ubtau)(\ubu)}{\|\ubtau\|_{\bX}} 
= - \sup_{\0\neq \ubtau\in \bX} \frac{ (\partial_t\,\cE + \cA)(\ubsi)(\ubtau) + \cB_1(\ubtau)(\ubvarphi)}{\|\ubtau\|_{\bX}} \,=\, 0.
\end{equation*}
Therefore, $\bu_f(t) = \0, \bu_s(t) = \0, \bgamma_f(t) = \0$, and $\bgamma_p(t) = \0$ for all $t\in (0,T]$.
Finally, from the third row in \eqref{eq:Stokes-Biot-formulation}, we have the identity
\begin{equation*}
b_f(\bsi_f,\bv_f) = 0 \quad \forall\,\bv_f\in \bV_f.
\end{equation*}
Taking $\bv_f = \bdiv(\bsi_f)\in \bV_f$, we deduce that $\bdiv(\bsi_f(t)) = \0$ for all $t\in (0,T]$, which combined with the fact that $\bsi^{\rd}_f(t) = \0$ for all $t\in (0,T]$, and estimates \eqref{eq:tau-d-H0div-inequality}--\eqref{eq:tau-H0div-Xf-inequality} yields $\bsi_f(t) = \0$ for all $t\in (0,T]$. 
Then, \eqref{eq:evolution-problem-in-operator-form} has a unique solution.
\end{proof}

\begin{cor}\label{cor:initial-data}
The solution of \eqref{eq:evolution-problem-in-operator-form} satisfies $\bsi_f(0)= \bsi_{f,0}, \bu_f(0)=\bu_{f,0}, \bgamma_{f}(0) = \bgamma_{f,0}, \bu_p(0) = \bu_{p,0}, \bvarphi(0) = \bvarphi_0$, $\lambda(0) = \lambda_0$, and $\btheta(0) = \btheta_0$.
\end{cor}
\begin{proof}
  Let $\overline{\bsi}_f := \bsi_f(0) - \bsi_{f,0}$, with a similar definition and notation for the rest of the variables. Since Theorem~\ref{thm:solvability-parabolic-problem} implies that $\cM(u) \in \L^{\infty}(0,T;E'_b)$, we can take $t\to 0$ in all equations without time derivatives in \eqref{eq:aux-hat-problem}, and therefore also in 
\eqref{eq:evolution-problem-in-operator-form}. Using that the initial data $(\ubsi_0,\ubvarphi_0,\ubu_0)$ satisfies the same equations at $t=0$ (cf. \eqref{eq:initial-data-problem}), and that $\overline{\bsi}_p = \0$ and $\overline{p}_p = 0$, we obtain
\begin{align}\label{eq:Stokes-Biot-formulation-0}
& \ds \frac{1}{2\mu}\,(\ov{\bsi}^\rd_f,\btau^\rd_f)_{\Omega_f}
+ (\ov{\bu}_f,\bdiv(\btau_f))_{\Omega_f}
+ (\ov{\bgamma}_f,\btau_f)_{\Omega_f}
- \pil\btau_f\bn_f,\ov{\bvarphi} \pir_{\Gamma_{fp}}
=  0, \nonumber \\[0.5ex]
& \ds \mu\,(\bK^{-1}\ov{\bu}_p,\bv_p)_{\Omega_p}
+ \pil\bv_p\cdot\bn_p,\ov{\lambda}\pir_{\Gamma_{fp}} = 0, \nonumber \\[0.5ex] 
& \ds -\,(\bv_f,\bdiv(\ov{\bsi}_f))_{\Omega_f} = 0,\nonumber \\[0.5ex] 
& \ds -\,(\ov{\bsi}_f,\bchi_f)_{\Omega_f} = 0, \nonumber \\[0.5ex]
& \ds -\,\pil\ov{\bvarphi}\cdot\bn_f + \left(\ov{\btheta} + \ov{\bu}_p\right)\cdot\bn_p,\xi\pir_{\Gamma_{fp}}
= 0, \\[0.5ex]
& \ds \pil\ov{\bsi}_f\bn_f,\bpsi\pir_{\Gamma_{fp}} + \mu\,\alpha_{\BJS}\,\sum_{j=1}^{n-1}
\pil\sqrt{\bK_j^{-1}}\left( \ov{\bvarphi} - \ov{\btheta} \right)\cdot\bt_{f,j},\bpsi\cdot\bt_{f,j} \pir_{\Gamma_{fp}}
 +\, \pil\bpsi\cdot\bn_f,\ov{\lambda}\pir_{\Gamma_{fp}} = 0,\nonumber \\[0.5ex]
&\ds - \mu\,\alpha_{\BJS}\,\sum_{j=1}^{n-1}
\pil\sqrt{\bK_j^{-1}}\left( \ov{\bvarphi} - \ov{\btheta} \right)\cdot\bt_{f,j},\bphi\cdot\bt_{f,j} \pir_{\Gamma_{fp}}
 +\, \pil\bphi\cdot\bn_p,\ov{\lambda}\pir_{\Gamma_{fp}} = 0. \nonumber
\end{align} 
Taking $(\btau_f,\bv_p,\bv_f,\bchi_f,\xi,\bpsi,\bphi) = (\ov{\bsi}_f,\ov{\bu}_p,\ov{\bu}_f,\ov{\bgamma}_f,\ov{\lambda},\ov{\bvarphi},\ov{\btheta})$ and combining the equations results in
\begin{equation}\label{eq:sigmaf-up-varphi-theta-overline}
\|\ov{\bsi}^{\rd}_f\|^2_{\bbL^2(\Omega_f)} + \|\ov{\bu}_p\|^2_{\bL^2(\Omega_p)} + |\ov{\bvarphi} - \ov{\btheta}|^2_{\BJS} \leq 0 \,,
\end{equation}
implying $\ov{\bsi}^{\rd}_f = \0, \ov{\bu}_p = \0$, and $(\ov{\bvarphi} - \ov{\btheta})\cdot\bt_{f,j} = 0$.
The inf-sup conditions \eqref{eq:inf-sup-b1}--\eqref{eq:inf-sup-B},
together with \eqref{eq:Stokes-Biot-formulation-0}, imply that $\ov{\bu}_f = \0, \ov{\bgamma}_f = 0, \ov{\bvarphi} = \0$, and $\ov{\lambda} = 0$. 
Then \eqref{eq:sigmaf-up-varphi-theta-overline} yields
$\ov{\btheta}\cdot\bt_{f,j} = 0$. 
In turn, the fifth equation in \eqref{eq:Stokes-Biot-formulation-0}
implies that $\pil \ov{\btheta}\cdot\bn_p,\xi\pir_{\Gamma_{fp}} = 0$ for all
$\xi \in \H^{1/2}(\Gamma_{fp})$. Note that $\bn_p$ may be discontinuous on $\Gamma_{fp}$,
thus $\ov{\btheta}\cdot\bn_p \in \L^2(\Gamma_{fp})$. Since
$\H^{1/2}(\Gamma_{fp})$ is dense in $\L^2(\Gamma_{fp})$, then
$\ov{\btheta}\cdot\bn_p = 0$, and we conclude that $\ov{\btheta} = \0$.
In addition, taking $\bv_f=\bdiv(\ov{\bsi}_f)\in \bV_f$ in the third equation of \eqref{eq:Stokes-Biot-formulation-0} we deduce that $\bdiv(\ov{\bsi}_f) = \0$, which,  combined with \eqref{eq:tau-d-H0div-inequality}--\eqref{eq:tau-H0div-Xf-inequality}, yields $\ov{\bsi}_f = \0$, completing the proof.
\end{proof}

\begin{rem}
As we noted in Remark~\ref{rem:A-sigmap-pp-without-time-derivative}, the fourth equation in \eqref{eq:Stokes-Biot-formulation-1} can be used to recover the non-differentiated equation \eqref{eq:A-sigmap-pp-without-time-derivative}.
In particular, recalling the initial data construction \eqref{eq:system-sol0-3}, let
\begin{equation*}
\forall\,t\in [0,T],\quad
\bbeta_p(t) = \bbeta_{p,0} + \int^t_0 \bu_s(s)\,ds,\quad
\brho_p(t) = \brho_{p,0} + \int^t_0 \bgamma_p(s)\,ds,\quad
\bomega(t) = \bomega_0 + \int^t_0 \btheta(s)\,ds.
\end{equation*}
Then \eqref{eq:A-sigmap-pp-without-time-derivative} follows from integrating the fourth equation in \eqref{eq:Stokes-Biot-formulation-1} from $0$ to $t\in (0,T]$ and using the first equation in \eqref{eq:system-sol0-3}.
\end{rem}

We end this section with a stability bound for the solution of \eqref{eq:evolution-problem-in-operator-form}. We will use the inf-sup condition
\begin{equation}\label{eq:p-lambda-bound 0}
\|p_p\|_{\W_p} + \|\lambda\|_{\Lambda_p} 
\,\leq\, c\,\sup_{\0\neq\bv_p\in \bV_p} \frac{b_p(\bv_p,p_p) + b_{\Gamma}(\bv_p,\lambda)}{\|\bv_p\|_{\bV_p}},
\end{equation}
which follows from a slight adaptation of \cite[Lemma~3.3]{gos2012}.

\begin{thm}\label{thm:stability-bound}
For the solution of \eqref{eq:evolution-problem-in-operator-form}, assuming sufficient regularity of the data, there exists a positive constant $C$ independent of $s_0$ such that
\begin{align}\label{eq:stability-bound}
& \ds \|\bsi_f\|_{\L^\infty(0,T;\bbX_f)} + \|\bsi_f\|_{\L^2(0,T;\bbX_f)} + \|\bu_p\|_{\L^\infty(0,T;\bL^2(\Omega_p))} + \|\bu_p\|_{\L^2(0,T;\bV_p)} + |\bvarphi - \btheta|_{\L^\infty(0,T;\BJS)} \nonumber \\[0.5ex]
& \ds\quad
+ |\bvarphi - \btheta|_{\L^2(0,T;\BJS)} + \|\lambda\|_{\L^\infty(0,T;\Lambda_p)} 
+ \|\ubvarphi\|_{\L^2(0,T;\bY)} + \|\ubu\|_{\L^2(0,T;\bZ)}
+ \|A^{1/2}(\bsi_p)\|_{\L^\infty(0,T;\bbL^2(\Omega_p))} \nonumber  \\[0.5ex] 
& \ds \quad
+ \|\bdiv(\bsi_p)\|_{\L^\infty(0,T;\bL^2(\Omega_p))}
+ \|\bdiv(\bsi_p)\|_{\L^2(0,T;\bL^2(\Omega_p))}
+ \|p_p\|_{\L^\infty(0,T;\W_p)}
+ \|p_p\|_{\L^2(0,T;\W_p)} \nonumber  \\[0.5ex] 
& \ds\quad 
+ \|\partial_t\,A^{1/2}(\bsi_p+\alpha_p p_p \bI)\|_{\L^2(0,T;\bbL^2(\Omega_p))}
+ \sqrt{s_0}\|\partial_t\,p_p\|_{\L^2(0,T;\W_p)} \\[0.5ex]
& \ds \leq C\,\Big( \|\f_f\|_{\H^1(0,T;\bV'_f)} 
+ \|\f_p\|_{\H^1(0,T;\bV'_s)} 
+ \|q_f\|_{\H^1(0,T;\bbX'_f)} 
+ \|q_p\|_{\H^1(0,T;\W'_p)} \nonumber \\[0.5ex]
& \ds \quad \qquad
+ (1+\sqrt{s_0})\|p_{p,0}\|_{\W_p}
+ \|\bK\nabla p_{p,0}\|_{\H^1(\Omega_p)} \Big). \nonumber
\end{align}
\end{thm}
\begin{proof}
We begin by choosing $(\ubtau,\ubpsi,\ubv) = (\ubsi,\ubvarphi,\ubu)$ in \eqref{eq:Stokes-Biot-formulation} to get
\begin{align}\label{eq:stability-identity}
& \ds \frac{1}{2}\,\partial_t\,\Big( 
\|A^{1/2}(\bsi_p + \alpha_p\,p_p\,\bI)\|^2_{\bbL^2(\Omega_p)}
+ s_0\,\|p_p\|^2_{\W_p} \Big) + \frac{1}{2\,\mu}\,\|\bsi^{\rd}_f\|^2_{\bbL^2(\Omega_f)} + a_p(\bu_p,\bu_p) + c_{\BJS}(\bvarphi, \btheta; \bvarphi, \btheta) \nonumber \\ 
& \ds\quad \,=\, -\frac{1}{n}\,(q_f \, \bI,\bsi_f)_{\Omega_f} + (q_p,p_p)_{\Omega_p} + (\f_f,\bu_f)_{\Omega_f} + (\f_p,\bu_s)_{\Omega_p}.
\end{align}
Next, we integrate \eqref{eq:stability-identity} from $0$ to $t\in (0,T]$,
  use the coercivity bounds
  \eqref{eq:monotonicity-E-A}--\eqref{eq:pos-sem-def-c}, 
  and apply the Cauchy--Schwarz and Young's inequalities, to find
\begin{align}\label{eq:stability-bound-1}
& \ds \|A^{1/2}(\bsi_p + \alpha_p\,p_p\,\bI)(t)\|^2_{\bbL^2(\Omega_p)} + s_0\|p_p(t)\|^2_{\W_p} 
+ \int^t_0 \Big( \|\bsi_f^\rd\|^2_{\bbL^2(\Omega_f)} + \|\bu_p\|^2_{\bL^2(\Omega_p)} + |\bvarphi - \btheta|^2_{\BJS} \Big)\,ds \nonumber \\
& \ds\quad\leq\, C\,\Bigg( \int^t_0 \Big( \|\f_f\|^2_{\bV'_f} + \|\f_p\|^2_{\bV'_s} + \|q_f\|^2_{\bbX'_f} + \|q_p\|^2_{\W'_p} \Big)\,ds + \|A^{1/2}(\bsi_p(0) + \alpha_p\,p_{p}(0)\bI)\|^2_{\bbL^2(\Omega_p)} \\
& \ds\qquad\, +\, s_0\,\|p_p(0)\|^2_{\W_p} \Bigg) \, + \,\delta \int^t_0 \Big( \|\bsi_f\|^2_{\bbX_f} + \|p_p\|^2_{\W_p} + \|\bu_f\|^2_{\bV_f} + \|\bu_s\|^2_{\bV_s} \Big)\,ds, \nonumber
\end{align}
where $\delta > 0$ will be suitably chosen.
In addition, \eqref{eq:p-lambda-bound 0} and the first equation in \eqref{eq:Stokes-Biot-formulation}, yields 
\begin{equation}\label{eq:p-lambda-bound}
\|p_p\|_{\W_p} + \|\lambda\|_{\Lambda_p} 
\,\leq\, c\,\sup_{\0\neq\bv_p\in \bV_p} \frac{b_p(\bv_p,p_p) + b_{\Gamma}(\bv_p,\lambda)}{\|\bv_p\|_{\bV_p}}
\,=\, -c\,\sup_{\0\neq\bv_p\in \bV_p} \frac{a_p(\bu_p,\bv_p)}{\|\bv_p\|_{\bV_p}} 
\,\leq\, C\,\|\bu_p\|_{\bL^2(\Omega_p)}.
\end{equation}
Taking $\ubtau\in \bV$ (cf. \eqref{eq:kernel-B}) in the first equation of \eqref{eq:evolution-problem-in-operator-form}, using the continuity of the operators $\cE$ and $\cA$ in Lemma \ref{lem:N-M-properties}, and the inf-sup condition of $\cB_1$ for $\ubvarphi\in \bY$ (cf. \eqref{eq:inf-sup-b1}), we deduce
\begin{equation}\label{eq:stability-bound-2}
\begin{array}{l}
\ds \beta_1\,\|\ubvarphi\|_{\bY} 
\,\leq\, \sup_{\0\neq\ubtau\in \bV}\,\frac{\cB_1(\ubtau)(\ubvarphi)}{\|\ubtau\|_{\bX}} 
\,=\, -\,\sup_{\0\neq \ubtau\in \bV}\,\frac{(\partial_t\,\cE + \cA)(\ubsi)(\ubtau) - \bF(\ubtau)}{\|\ubtau\|_{\bX}}  \\[3ex]
\ds \leq C \left( \|\bsi_f\|_{\bbX_f} + \|\bu_p\|_{\bV_p} + \|\partial_t\,A^{1/2}(\bsi_p+ \alpha_p p_p \bI)\|_{\bbL^2(\Omega_p)} + \sqrt{s_0}\|\partial_t\,p_p\|_{\W_p} + \|q_f\|_{\bbX'_f} + \|q_p\|_{\W'_p} \right) .
\end{array}
\end{equation}
In turn, from the first equation in \eqref{eq:evolution-problem-in-operator-form}, applying the inf-sup condition of $\cB$ (cf. \eqref{eq:inf-sup-B}) for $\ubu = (\bu_f,\bu_s,\bgamma_f,\bgamma_p)\in \bZ$, and \eqref{eq:stability-bound-2}, we obtain
\begin{equation}\label{eq:stability-bound-3}
\begin{array}{l}
\ds \beta\,\|\ubu\|_{\bZ} 
\,\leq\, \sup_{\0\neq \ubtau\in \bX} \frac{\cB(\ubtau)(\ubu)}{\|\ubtau\|_{\bX}} 
\,=\, -\,\sup_{\0\neq \ubtau\in \bX} \frac{(\partial_t\,\cE + \cA)(\ubsi)(\ubtau) + \cB_1(\ubtau)(\ubvarphi) - \bF(\ubtau)}{\|\ubtau\|_{\bX}} \\ [3ex]
\ds \leq C \left( \|\bsi_f\|_{\bbX_f} + \|\bu_p\|_{\bV_p} + \|\partial_t\,A^{1/2}(\bsi_p+ \alpha_p p_p \bI)\|_{\bbL^2(\Omega_p)} + \sqrt{s_0}\|\partial_t\,p_p\|_{\W_p} + \|q_f\|_{\bbX'_f} + \|q_p\|_{\W'_p} \right).
\end{array}
\end{equation}

In addition, taking $w_p = \div(\bu_p)$, $\bv_f = \bdiv(\bsi_f)$, and $\bv_s = \bdiv(\bsi_p)$ in the first and third equations of \eqref{eq:Stokes-Biot-formulation}, we get
\begin{equation}\label{eq:div-up-sigmaf-bound}
\begin{array}{c}
\ds \|\bdiv(\bsi_f)\|_{\bL^2(\Omega_f)} \,\leq\, \|\f_f\|_{\bV'_f},\quad
\|\bdiv(\bsi_p)\|_{\bL^2(\Omega_p)} \,\leq\, \|\f_p\|_{\bV'_s}\,, \\ [2ex]
\ds \|\div(\bu_p)\|_{\L^2(\Omega_p)} \,\leq\, C\left( \|\partial_t\,A^{1/2}(\bsi_p+ \alpha_p p_p \bI)\|_{\bbL^2(\Omega_p)} + \sqrt{s_0}\|\partial_t\,p_p\|_{\W_p} + \|q_p\|_{\W'_p} \right).
\end{array}
\end{equation}
Then, combining \eqref{eq:stability-bound-1}--\eqref{eq:div-up-sigmaf-bound},
using \eqref{eq:tau-d-H0div-inequality}--\eqref{eq:tau-H0div-Xf-inequality},
and choosing $\delta$ small enough, we obtain
\begin{equation}\label{eq:stability-bound-4}
\begin{array}{l}
\ds \|A^{1/2}(\bsi_p+ \alpha_p p_p \bI)(t)\|^2_{\bbL^2(\Omega_p)} + s_0\|p_p(t)\|^2_{\W_p} \\ [3ex]
\ds\quad
+ \, \int^t_0 \Big( \|\bsi_f\|^2_{\bbX_f} + \|\bu_p\|^2_{\bV_p} + \|\bdiv(\bsi_p)\|^2_{\bL^2(\Omega_p)} + \| p_p \|^2_{\W_p}+ |\bvarphi - \btheta|^2_{\BJS} + \|\ubvarphi\|^2_{\bY} + \|\ubu\|^2_{\bZ} \Big)\,ds \\ [3ex]
\ds\leq\, C\,\Bigg( \int^t_0 \Big( \|\f_f\|^2_{\bV'_f} + \|\f_p\|^2_{\bV'_s} + \|q_f\|^2_{\bbX'_f} + \|q_p\|^2_{\W'_p} \Big)\,ds
+ \|A^{1/2}(\bsi_p(0) + \alpha_p\,p_{p}(0)\bI)\|^2_{\bbL^2(\Omega_p)} \\[3ex]
\ds\quad\, +\, s_0\,\|p_p(0)\|^2_{\W_p} \,+\,\int^t_0 \Big(  \|\partial_t\,A^{1/2}(\bsi_p+ \alpha_p p_p \bI)\|^2_{\bbL^2(\Omega_p)} +  s_0\|\partial_t\,p_p\|^2_{\W_p} \Big)\, ds \Bigg).
\end{array}
\end{equation}
Finally, in order to bound the last two terms in \eqref{eq:stability-bound-4}, we test \eqref{eq:Stokes-Biot-formulation} with $\ubtau = (\partial_t\,\bsi_f, \bu_p, \partial_t\,\bsi_p, \partial_t\,p_p)\linebreak\in \bX$, $\ubpsi = (\bvarphi, \btheta, \partial_t\,\lambda)\in \bY$,  $\ubv = (\bu_f, \bu_s, \bgamma_f, \bgamma_p)\in \bZ$ and differentiate in time the rows in \eqref{eq:Stokes-Biot-formulation} associated to $\bv_p, \bpsi, \bphi, \bv_f, \bv_s, \bchi_f$ and $\bchi_p$, to deduce 
\begin{equation*}
\begin{array}{l}
\ds \frac{1}{2}\,\partial_t\,\Big( \frac{1}{2\,\mu}\,\|\bsi^{\rd}_f\|^2_{\bbL^2(\Omega_f)} + a_p(\bu_p, \bu_p) + c_{\BJS}(\bvarphi, \btheta; \bvarphi, \btheta) \Big) 
+ \|\partial_t\,A^{1/2}(\bsi_p + \alpha_p\,p_p\,\bI)\|^2_{\bbL^2(\Omega_p)}
+ s_0\,\|\partial_t\,p_p\|^2_{\W_p} \\ [2ex]
\ds\quad  
\,=\, \frac{1}{n}\,(q_f\, \bI ,\partial_t\,\bsi_f)_{\Omega_f} + (q_p,\partial_t\,p_p)_{\Omega_p} + (\partial_t\,\f_f,\bu_f)_{\Omega_f} + (\partial_t\,\f_p,\bu_s)_{\Omega_p}, 
\end{array}
\end{equation*}
which together with the identities 
\begin{equation*}
\begin{array}{c}
\ds \int^t_0 (q_f \, \bI,\partial_t\,\bsi_f)_{\Omega_f} = (q_f \, \bI,\bsi_f)_{\Omega_f}\Big|^t_0 - \int^t_0 (\partial_t\,q_f\, \bI,\bsi_f)_{\Omega_f}, \\[2ex]
\ds \int^t_0 (q_p,\partial_t\,p_p)_{\Omega_p} = (q_p,p_p)_{\Omega_p}\Big|^t_0 - \int^t_0 (\partial_t\,q_p,p_p)_{\Omega_p}, 
\end{array}
\end{equation*}
and the positive semi-definite property of $\cC$ (cf. \eqref{eq:pos-sem-def-c}), yields
\begin{align}\label{eq:stability-bound-5}
& \ds \|\bsi^\rd_f(t)\|^2_{\bbL^2(\Omega_f)} + \|\bu_p(t)\|^2_{\bL^2(\Omega_p)} + |\bvarphi(t) - \btheta(t)|^2_{\BJS} 
+ \int^t_0 \Big(\|\partial_t A^{1/2}(\bsi_p+ \alpha_p p_p \bI)\|^2_{\bbL^2(\Omega_p)} +  s_0\|\partial_t p_p\|^2_{\W_p} \Big)\,ds \nonumber \\
& \ds \leq\, C\,\Bigg( \int^t_0 \Big( \|\partial_t\,\f_f\|^2_{\bV'_f} + \|\partial_t\,\f_p\|^2_{\bV'_s} + \|\partial_t\,q_f\|^2_{\L^2(\Omega_f)}
+ \|\partial_t \, q_p\|^2_{\W'_p} \Big)\,ds
+ \|q_f(t)\|^2_{\bbX'_f}
+ \|q_p(t)\|^2_{\W'_p} \nonumber \\[1ex]
& \ds\quad\,+\,\,
\|q_f(0)\|^2_{\bbX'_f}
+ \|q_p(0)\|^2_{\W'_p} 
+ \|\bsi_f(0)\|^2_{\bbX_f}
+ \|\bu_p(0)\|^2_{\bL^2(\Omega_p)} 
+ \|p_p(0)\|^2_{\W_p} 
+ |\bvarphi(0) - \btheta(0)|^2_{\BJS} \Bigg) \\ 
& \ds\quad\,+\,\, 
\delta_1\,\Big(\|\bsi_f(t)\|^2_{\bbX_f}+ \|p_p(t)\|^2_{\W_p}
\Big)
+ \delta_2\,\int^t_0 \Big( \|\bsi_f\|^2_{\bbL^2(\Omega_f)} +
\|p_p\|^2_{\W_p} +
\|\bu_f\|^2_{\bV_f} + \|\bu_s\|^2_{\bV_s} \Big)\,ds. \nonumber
\end{align}
Using \eqref{eq:p-lambda-bound} and the first two inequalities in
\eqref{eq:div-up-sigmaf-bound}, and choosing $\delta_1$ small enough,
we derive from \eqref{eq:stability-bound-5} and
\eqref{eq:tau-d-H0div-inequality}--\eqref{eq:tau-H0div-Xf-inequality} that
\begin{align}\label{eq:stability-bound-6}
& \ds \|\bsi_f(t)\|^2_{\bbX_f} + \|\bu_p(t)\|^2_{\bL^2(\Omega_p)} 
+ \|\bdiv(\bsi_p(t))\|^2_{\bL^2(\Omega_p)}
+ |\bvarphi(t) - \btheta(t)|^2_{\BJS} 
+ \|p_p(t)\|^2_{\W_p}
+ \|\lambda(t)\|^2_{\Lambda_p}
\nonumber \\[1ex]
& \,\ds \ + \,\int^t_0 \Big(\|\partial_t\,A^{1/2}(\bsi_p+ \alpha_p p_p \bI)\|^2_{\bbL^2(\Omega_p)} +  s_0\|\partial_t\,p_p\|^2_{\W_p} \Big)\,ds \nonumber \\[1ex]
& \ds \leq\, C\,\Bigg( \int^t_0 \Big( \|\partial_t\,\f_f\|^2_{\bV'_f} + \|\partial_t\,\f_p\|^2_{\bV'_s} + \|\partial_t\,q_f\|^2_{\L^2(\Omega_f)}
+ \|\partial_t q_p\|^2_{\W'_p} \Big)\,ds 
+ \|\f_f(t)\|^2_{\bV'_f}
+ \|\f_p(t)\|^2_{\bV'_s} \\[1ex]
& \ds\quad
+\,\|q_f(t)\|^2_{\bbX'_f} + \|q_p(t)\|^2_{\W'_p}  
+ \|q_f(0)\|^2_{\bbX'_f} + \|q_p(0)\|^2_{\W'_p} 
+ \|\bsi_f(0)\|^2_{\bbX_f}
+ \|\bu_p(0)\|^2_{\bL^2(\Omega_p)}
+ \|p_p(0)\|^2_{\W_p} \nonumber \\[1ex]
& \ds\quad +\,|\bvarphi(0) - \btheta(0)|^2_{\BJS} \Bigg)
+ \delta_2 \int^t_0 \Big( \|\bsi_f\|^2_{\bbX_f} +
\|p_p\|^2_{\W_p}+\|\bu_f\|^2_{\bV_f} + \|\bu_s\|^2_{\bV_s} \Big)\,ds. \nonumber
\end{align}
We next bound the initial data terms in \eqref{eq:stability-bound-4} and 
\eqref{eq:stability-bound-6}. Recalling from Corollary~\ref{cor:initial-data} that
$(\ubsi(0), \bvarphi(0),\btheta(0))
= (\ubsi_{0}, \bvarphi_0, \btheta_{0})$, using the stability of the continuous
initial data problems \eqref{eq:sol0-up0-pp0}--\eqref{eq:system-sol0-3} 
and the steady-state version of the arguments leading to 
\eqref{eq:stability-bound-4}, we obtain
\begin{equation}\label{init-data-bound}
\begin{array}{l}
\ds \|\bsi_{f}(0)\|^2_{\bbX_f} + \|\bu_p(0)\|^2_{\bL^2(\Omega_p)} 
+ \|A^{1/2}(\bsi_p(0))\|^2_{\bbL^2(\Omega_p)}
+ \|p_p(0)\|^2_{\W_p} 
+ |\bvarphi(0) - \btheta(0)|^2_{\BJS} \\[2ex]
\quad \le\, C\,
\left( \|p_{p,0}\|^2_{\W_p} + \|\bK\nabla p_{p,0}\|^2_{\H^1(\Omega_p)}
+ \| \f_f(0) \|^2_{\bV_f'}
+ \| \f_p(0) \|^2_{\bV_s'}
+ \| q_f(0) \|^2_{\bbX_f'}  \right),
\end{array}
\end{equation}
Therefore, combining \eqref{eq:stability-bound-4} with
\eqref{eq:stability-bound-6} and \eqref{init-data-bound},
choosing $\delta_2$ small enough, and using
the estimate (cf. \eqref{eq:aux-taup-qp-inequality}):
\begin{equation}\label{eq:A-sigmap-t-bound}
\|A^{1/2}(\bsi_p(t))\|_{\bbL^2(\Omega_p)} \leq C\,\left( \|A^{1/2}(\bsi_p + \alpha_p\,p_p\,\bI)(t)\|_{\bbL^2(\Omega_p)} + \|p_p(t)\|_{\W_p} \right),
\end{equation}
and the Sobolev embedding of $\H^1(0,T)$ into $\L^\infty(0,T)$, we
conclude \eqref{eq:stability-bound}.
\end{proof}

\section{Semidiscrete continuous-in-time approximation}\label{sec:semidiscrete-formulation}

In this section we introduce and analyze the semidiscrete
continuous-in-time approximation of
\eqref{eq:evolution-problem-in-operator-form}.  We analyze its
solvability by employing the strategy developed in
Section~\ref{sec:well-posedness-model}.  In addition, we derive
error estimates with rates of convergence.

Let $\cT_h^f$ and $\cT_h^p$ be shape-regular and quasi-uniform affine
finite element partitions of $\Omega_f$ and $\Omega_p$, respectively.
The two partitions may be non-matching along the interface $\Gamma_{fp}$.
For the discretization, we consider the following conforming finite
element spaces:
\begin{equation*}
\bbX_{fh}\times \bV_{fh}\times \bbQ_{fh} \subset \bbX_{f}\times \bV_{f}\times \bbQ_{f},
\quad
\bbX_{ph}\times \bV_{sh}\times \bbQ_{ph} \subset \bbX_{p}\times \bV_{s}\times \bbQ_{p},
\quad
\bV_{ph}\times \W_{ph} \subset \bV_{p}\times \W_{p}.
\end{equation*}
We take $(\bbX_{fh}, \bV_{fh}, \bbQ_{fh})$ and $(\bbX_{ph}, \bV_{sh},
\bbQ_{ph})$ to be any stable finite element spaces for mixed
elasticity with weakly imposed stress symmetry, such as the
Amara--Thomas \cite{at1979}, PEERS \cite{abd1984}, Stenberg
\cite{stenberg1988}, Arnold--Falk--Winther \cite{afw2007,awanou2013},
or Cockburn--Gopalakrishnan--Guzman \cite{cgg2010} families of spaces.
We choose $(\bV_{ph}, \W_{ph})$ to be any stable
mixed finite element Darcy spaces, such as the Raviart--Thomas or
Brezzi-Douglas-Marini
spaces \cite{Brezzi-Fortin}. For the Lagrange
multipliers $(\bLambda_{fh}, \bLambda_{sh}, \Lambda_{ph})$ 
we consider the following two options of discrete spaces.
\begin{itemize}
\item[{\bf (S1)}] Conforming spaces:
\begin{equation}\label{eq:conforming-spaces-Yh}
\bLambda_{fh} \subset \bLambda_f,\quad 
\bLambda_{sh} \subset \bLambda_s,\quad 
\Lambda_{ph} \subset \Lambda_p\,,
\end{equation}
equipped with $\H^{1/2}$-norms as in \eqref{eq:norms-H-half}.  If the
normal traces of the spaces $\bbX_{fh}$, $\bbX_{ph}$, or $\bV_{ph}$
contain piecewise polynomials in $\P_k$ on simplices or $\Q_k$ on cubes with $k \ge 1$,
where $\P_k$ denotes polynomials of total degree $k$ and $\Q_k$
stands for polynomials of degree $k$ in each variable, we take the
Lagrange multiplier spaces to be continuous piecewise polynomials in
$\P_k$ or $\Q_k$ on the traces of the corresponding subdomain grids.
In the case of $k = 0$, we take the Lagrange multiplier spaces to be continuous
piecewise polynomials in $\P_1$ or $\Q_1$ on grids obtained by
coarsening by two the traces of the subdomain grids.

\item[{\bf (S2)}] Non-conforming spaces:
\begin{equation}\label{eq:non-conforming-spaces-Yh}
\bLambda_{fh} := \bbX_{fh}\bn_f|_{\Gamma_{fp}},\quad 
\bLambda_{sh} := \bbX_{ph}\bn_p|_{\Gamma_{fp}},\quad 
\Lambda_{ph} := \bV_{ph}\cdot\bn_p|_{\Gamma_{fp}}\,,
\end{equation}
which consist of discontinuous piecewise polynomials and are equipped with $\L^2$-norms.
\end{itemize}
It is also possible to mix conforming and non-conforming choices, but we will focus on {\bf (S1)} and {\bf (S2)} for simplicity of the presentation.

\begin{rem}\label{rem:non-conf}
  We note that, since $\H^{1/2}(\Gamma_{fp})$ is dense in $\L^2(\Gamma_{fp})$, 
  the last three equations in the continuous weak formulation
  \eqref{eq:Stokes-Biot-formulation-1} hold for test functions in $\L^2(\Gamma_{fp})$,
  assuming that the solution is smooth enough. In particular, these equations hold
  for $\xi_h \in \Lambda_{ph}$, $\bpsi_h \in \bLambda_{fh}$, and
  $\bphi_h \in \bLambda_{sh}$ in both the conforming case {\bf (S1)} and the
  non-conforming case {\bf (S2)}.
  \end{rem}

Now, we group the spaces similarly to the continuous case:
\begin{equation*}
\begin{array}{c}
\ds \bX_h := \bbX_{fh}\times \bV_{ph} \times \bbX_{ph} \times \W_{ph},\quad
\bY_h := \bLambda_{fh} \times \bLambda_{sh} \times \Lambda_{ph},\quad
\bZ_h := \bV_{fh} \times \bV_{sh} \times \bbQ_{fh} \times \bbQ_{ph}, \\ [2ex]
\ds \ubsi_h := (\bsi_{fh}, \bu_{ph}, \bsi_{ph}, p_{ph})\in \bX_h,\quad
\ubvarphi_h := (\bvarphi_h, \btheta_h, \lambda_h)\in \bY_h,\quad 
\ubu_h := (\bu_{fh}, \bu_{sh}, \bgamma_{fh}, \bgamma_{ph})\in \bZ_h, \\ [1ex]
\ds \ubtau_h := (\btau_{fh}, \bv_{ph}, \btau_{ph}, w_{ph})\in \bX_h,\quad
\ubpsi_h := (\bpsi_h, \bphi_h, \xi_h)\in \bY_h,\quad 
\ubv_h := (\bv_{fh}, \bv_{sh}, \bchi_{fh}, \bchi_{ph})\in \bZ_h.
\end{array}
\end{equation*}
The spaces $\bX_h$ and $\bZ_h$ are endowed with the same norms as their
continuous counterparts. For $\bY_h$ we consider the norm $\|\ubpsi_h\|_{\bY_h} := \|\bpsi_h\|_{\bLambda_{fh}} + \|\bphi_h\|_{\bLambda_{sh}} + \|\xi_h\|_{\Lambda_{ph}}$, where
\begin{equation}\label{eq:norm-in-Yh}
\|\xi_h\|_{\Lambda_{ph}} :=
\left\{ \begin{array}{l}
\|\xi_h\|_{\Lambda_{p}}
\mbox{ for conforming subspaces } {\bf(\bS1)} \mbox{ (cf. \eqref{eq:norms-H-half})} \,,\\[1.5ex]
\|\xi_h\|_{\L^2(\Gamma_{fp})} \mbox{ for non-conforming subspaces } {\bf(\bS2)}\,.
\end{array}\right.
\end{equation}
Analogous notation is used for $\|\bpsi_h\|_{\bLambda_{fh}}$ and $\|\bphi_h\|_{\bLambda_{sh}}$.

The continuity of all operators in the discrete case follows from their
continuity in the continuous case (cf. Lemma~\ref{lem:N-M-properties}),
with the exception of $\cB_1$ (cf. \eqref{eq:operators-A-B1-C})
in the case of non-conforming
Lagrange multipliers ${\bf(\bS2)}$. In this case it follows for each
fixed $h$ from the
discrete trace-inverse inequality for piecewise polynomial
functions, $\|\varphi\|_{\L^2(\Gamma)} \le C h^{-1/2}\|\varphi\|_{\L^2(\cO)}$,
where $\Gamma \subset \partial \cO$. In particular,
\begin{equation}\label{cont-B1-discrete}
b_{\bn_f}(\btau_f,\bpsi) \le C \|\btau_f\|_{\bbL^2(\Gamma_{fp})}\|\bpsi\|_{\bL^2(\Gamma_{fp})}
\le C h^{-1/2}\|\btau_f\|_{\bbL^2(\Omega_f)}
\|\bpsi\|_{\bL^2(\Gamma_{fp})},
\end{equation}
with similar bounds for $b_{\bn_p}(\btau_p,\bphi)$ and $b_{\Gamma}(\bv_p,\xi)$.

We next discuss the discrete inf-sup conditions that are satisfied by the 
finite element spaces. Let
\begin{equation}\label{eq:subspace-bbXh-tilde}
\wt{\bX}_h := \Big\{ \ubtau_{h}\in \bX_h :\quad \btau_{fh}\bn_f = \0 \qan \btau_{ph}\bn_p = \0 \qon \Gamma_{fp} \Big\}.
\end{equation}
In addition, define the discrete kernel of the operator $\cB$ as
\begin{equation}\label{eq:discrete-kernel-B}
  \bV_h := \Big\{ \ubtau_h\in \bX_h :\quad \cB(\ubtau_h)(\ubv_h) = 0
  \quad \forall\,\bv_h\in \bZ_h \Big\} 
= \wt{\bbX}_{fh}\times \bV_{ph}\times \wt{\bbX}_{ph}\times \W_{ph},
\end{equation}
where
\begin{equation*}
\wt{\bbX}_{\star h} := \Big\{ \btau_{\star h}\in \bbX_{\star h} :\,
(\btau_{\star h}, \bxi_{\star h})_{\Omega_\star} = 0 \ \ \forall \, \bxi_{\star h} \in \bbQ_{\star h}
\qan \bdiv(\btau_{\star h}) = \0 \qin \Omega_\star \Big\},\quad \star\in \{f,p\}. 
\end{equation*}
In the above, $\bdiv(\btau_{\star h}) = \0$ follows from 
$\bdiv(\bbX_{fh}) = \bV_{fh}$ and $\bdiv(\bbX_{ph}) = \bV_{sh}$, which is true
for all stable elasticity spaces. 

\begin{lem}\label{lem:discr-inf-sup}
There exist positive constants $\wt{\beta}$ and $\wt{\beta}_{1}$ such that
\begin{equation}\label{eq:inf-sup-B-semidiscrete}
\sup_{\0\neq \ubtau_h\in \wt{\bX}_h} \frac{\cB(\ubtau_h)(\ubv_h)}{\|\ubtau_h\|_{\bX}}
\,\geq\, \wt{\beta}\,\|\ubv_h\|_{\bZ} \quad \forall\, \ubv_h \in \bZ_h,
\end{equation}
\begin{equation}\label{eq:inf-sup-b1-semidiscrete}
\sup_{\0\neq \ubtau_h \in \bV_h} \frac{\cB_1(\ubtau_h)(\ubpsi_h)}{\|\ubtau_h \|_{\bX}} 
\,\geq\, \wt{\beta}_{1}\,\|\ubpsi_h\|_{\bY_h} \quad \forall\,\ubpsi_h\in \bY_h.
\end{equation}
\end{lem}

\begin{proof}
We begin with the proof of \eqref{eq:inf-sup-B-semidiscrete}.
We recall that the space $\bX_h$ consists of stresses and velocities
with zero normal traces on the Neumann boundaries, while the space
$\wt{\bX}_h$ involves further restriction on $\Gamma_{fp}$. The 
inf-sup condition \eqref{eq:inf-sup-B-semidiscrete}
without restricting the normal stress or velocity
on the subdomain boundary follows from the stability of the
elasticity and Darcy finite element spaces. The restricted inf-sup
condition \eqref{eq:inf-sup-B-semidiscrete} can be shown using the
argument in \cite[Theorem~4.2]{msmfe-simpl}.

We continue with the proof of \eqref{eq:inf-sup-b1-semidiscrete}.
Similarly to the continuous case, due the diagonal character of
operator $\cB_1$ (cf. \eqref{eq:operators-A-B1-C}), we need to show
individual inf-sup conditions for $b_{\bn_f}$, $b_{\bn_p}$, and
$b_{\Gamma}$. We first focus on $b_{\Gamma}$. For the conforming case
{\bf (S1)} (cf. \eqref{eq:conforming-spaces-Yh}), the proof of
\eqref{eq:inf-sup-b1-semidiscrete} can be derived from a slight
adaptation of \cite[Lemma~4.4]{ejs2009} (see also
\cite[Section~5.3]{gmor2014} for the case $k=0$), whereas from
\cite[Section~5.1]{aeny2019} we obtain the proof for the
non-conforming version {\bf (S2)}
(cf. \eqref{eq:non-conforming-spaces-Yh}). We next consider the
inf-sup condition \eqref{eq:inf-sup-b1-semidiscrete} for $b_{\bn_f}$,
with argument for $b_{\bn_p}$ being similar.
The proof utilizes a suitable interpolant of $\btau_f := \be(\bv_f)$,
the solution to the auxiliary problem \eqref{elast-aux}.
Due to the stability of the spaces $(\bbX_{fh}, \bV_{fh}, \bbQ_{fh})$
(cf. \eqref{eq:inf-sup-B-semidiscrete}), there exists an interpolant
$\tilde\Pi_h^f: \bbH^1(\Omega_f) \to \bbX_{fh}$ satisfying
\begin{equation}\label{pi-tilde}
  \begin{array}{c}
  b_f(\tilde\Pi_h^f\btau_f - \btau_f,\bv_{fh}) = 0 \quad
  \forall \, \bv_{fh} \in \bV_{fh}, \quad
  b_{\sk,f}(\tilde\Pi_h^f\btau_f - \btau_f,\bchi_{fh}) = 0 \quad
  \forall \, \bchi_{fh} \in \bbQ_{fh}, \\ [1ex]
  \langle(\tilde\Pi_h^f\btau_f - \btau_f)\bn_f,
  \btau_{fh}\bn_f\rangle_{\Gamma_{fp}\cup\Gamma_f^N} = 0
  \quad \forall \, \btau_{fh} \in \bbX_{fh}.
  \end{array}
\end{equation}
The interpolant $\tilde\Pi_h^f\btau_f$ is defined as the elliptic
projection of $\btau_f$ satisfying Neumann boundary condition on
$\Gamma_{fp}\cup\Gamma_f^N$ \cite[(3.11)--(3.15)]{Khat-Yot}.
Due to \eqref{pi-tilde},
it holds that $\tilde\Pi_h^f\btau_f \in \wt{\bbX}_{fh}$. With this
interpolant, the proof of \eqref{eq:inf-sup-b1-semidiscrete} for
$b_\Gamma$ discussed above
can be easily modified for $b_{\bn_f}$, see \cite[Lemma~4.4]{ejs2009}
and \cite[Section~5.3]{gmor2014} for {\bf (S1)} and \cite[Section~5.1]{aeny2019}
for {\bf (S2)}.
\end{proof}

\begin{rem} The stability analysis requires only a discrete
inf-sup condition for $\cB$ in $\bX_h\times \bZ_h$. The more
restrictive inf-sup condition \eqref{eq:inf-sup-B-semidiscrete} is
used in the error analysis in order to simplify the proof.
\end{rem}

Finally, we will utilize the following inf-sup condition:
there exists a constant $c>0$ such that
\begin{equation}\label{eq:p-lambda-bound semidiscrete}
\|p_{ph}\|_{\W_p} + \|\lambda_h\|_{\Lambda_{ph}} 
\,\leq\, c\,\sup_{\0\neq \bv_{ph}\in \bV_{ph}} \frac{b_p(\bv_{ph},p_{ph})
  + b_{\Gamma}(\bv_{ph},\lambda_h)}{\|\bv_{ph}\|_{\bV_p}},
\end{equation}
whose proof for the conforming case \eqref{eq:conforming-spaces-Yh} 
follows from a slight adaptation of \cite[Lemma~5.1]{gos2012}, 
whereas the non-conforming case \eqref{eq:non-conforming-spaces-Yh} can be
found in \cite[Section~5.1]{aeny2019}.

The semidiscrete continuous-in-time approximation to
\eqref{eq:evolution-problem-in-operator-form} reads: 
find 
$(\ubsi_h,\ubvarphi_h,\ubu_h):[0,T]\to \bX_h\times \bY_h\times \bZ_h$ such that for all $(\ubtau_h,\ubpsi_h,\ubv_h)\in \bX_h\times \bY_h\times \bZ_h$, and for a.e. $t\in (0,T)$, 
\begin{equation}\label{eq:Stokes-Biot-formulation-semidiscrete}
\begin{array}{lll}
\ds \frac{\partial}{\partial t}\,\cE(\ubsi_h)(\ubtau_h) + \cA(\ubsi_h)(\ubtau_h) + \cB_1(\ubtau_h)(\ubvarphi_h) + \cB(\ubtau_h)(\ubu_h) & = & \bF(\ubtau_h) , \\ [2ex]
\ds -\,\cB_1(\ubsi_h)(\ubpsi_h) + \cC(\ubvarphi_h)(\ubpsi_h) & = & 0 , \\ [2ex]
\ds -\,\cB\,(\ubsi_h)(\ubv_h) & = & \bG(\ubv_h) .
\end{array}
\end{equation}

We next discuss the choice of compatible discrete initial data $(\ubsi_{h,0},\ubvarphi_{h,0},\ubu_{h,0})$, whose construction is based on a modification of the step-by-step procedure for the continuous initial data.

1. Define $\btheta_{h,0} := P^{\bLambda_s}_h(\btheta_0)$, where $P^{\bLambda_s}_h : \bLambda_s \to \bLambda_{sh}$ is the classical $\L^2$-projection operator, satisfying, for all $\bphi\in \bL^2(\Gamma_{fp})$,
\begin{equation*}
\pil \bphi - P^{\bLambda_s}_h(\bphi), \bphi_h\pir_{\Gamma_{fp}} = 0 \quad \forall\,\bphi_h\in \bLambda_{sh}\,.
\end{equation*} 

2. Define $(\bsi_{fh,0},\bvarphi_{h,0},\bu_{fh,0},\bgamma_{fh,0})\in \bbX_{fh}\times \bLambda_{fh}\times \bV_{fh}\times \bbQ_{fh}$ and $(\bu_{ph,0},p_{ph,0},\lambda_{h,0})\in \bV_{ph}\times \W_{ph}\times \Lambda_{ph}$ by solving a coupled Stokes-Darcy problem: 
\begin{align}\label{eq:auxiliary-initial-condition-semidiscrete-1}
& a_f(\bsi_{fh,0},\btau_{fh}) + b_{\bn_f}(\btau_{fh},\bvarphi_{h,0})
+ b_f(\btau_{fh},\bu_{fh,0}) + b_{\skf}(\btau_{fh},\bgamma_{fh,0}) \nonumber \\
& \ds\quad =\, a_f(\bsi_{f,0},\btau_{fh}) + b_{\bn_f}(\btau_{fh},\bvarphi_{0})
+ b_f(\btau_{fh},\bu_{f,0}) + b_{\skf}(\btau_{fh},\bgamma_{f,0})
= \ds -\frac{1}{n}\,(q_f(0)\,\bI,\btau_{fh})_{\Omega_f}, \nonumber \\
& \ds -b_{\bn_f}(\bsi_{fh,0},\bpsi_h) 
+ \mu\,\alpha_{\BJS}\sum^{n-1}_{j=1}
\pil\sqrt{\bK_j^{-1}} (\bvarphi_{h,0} - \btheta_{h,0})\cdot\bt_{f,j},\bpsi_h\cdot\bt_{f,j}\pir_{\Gamma_{fp}}
+ \pil\bpsi_h\cdot\bn_f,\lambda_{h,0}\pir_{\Gamma_{fp}} \nonumber \\
& \ds\quad =\, -b_{\bn_f}(\bsi_{f,0},\bpsi_h) 
+ \mu\,\alpha_{\BJS}\sum^{n-1}_{j=1}
\pil\sqrt{\bK_j^{-1}} (\bvarphi_{0} - \btheta_{0})\cdot\bt_{f,j},\bpsi_h\cdot\bt_{f,j}\pir_{\Gamma_{fp}}
+ \pil\bpsi_h\cdot\bn_f,\lambda_{0}\pir_{\Gamma_{fp}}  = 0, \nonumber \\[0.5ex]
& -b_f(\bsi_{fh,0},\bv_{fh}) - b_{\skf}(\bsi_{fh,0},\bchi_{fh})
= -b_f(\bsi_{f,0},\bv_{fh}) - b_{\skf}(\bsi_{f,0},\bchi_{fh})
=  (\f_f(0),\bv_{fh})_{\Omega_f}, \\[0.5ex]
& \ds a_p(\bu_{ph,0},\bv_{ph}) + b_p(\bv_{ph},p_{ph,0}) + b_{\Gamma}(\bv_{ph},\lambda_{h,0})
= a_p(\bu_{p,0},\bv_{ph}) + b_p(\bv_{ph},p_{p,0}) + b_{\Gamma}(\bv_{ph},\lambda_{0})
= 0 \,, \nonumber \\[0.5ex]
& \ds - b_p(\bu_{ph,0},w_{ph}) = - b_p(\bu_{p,0},w_{ph}) = -\mu^{-1} (\div(\bK\nabla p_{p,0}),w_{ph})_{\Omega_p}, \nonumber \\[0.5ex]
& \ds -\pil\bvarphi_{h,0}\cdot\bn_f + (\btheta_{h,0} + \bu_{ph,0})\cdot\bn_p,\xi_h\pir_{\Gamma_{fp}} 
= -\pil\bvarphi_{0}\cdot\bn_f + (\btheta_{0} + \bu_{p,0})\cdot\bn_p,\xi_h\pir_{\Gamma_{fp}} 
= 0, \nonumber
\end{align}
for all $(\btau_{fh}, \bpsi_h, \bv_{fh}, \bchi_{fh}) \in \bbX_{fh}\times \bLambda_{fh}\times \bV_{fh}\times \bbQ_{fh}$ and $(\bv_{ph},w_{ph},\xi_h)\in \bV_{ph}\times \W_{ph}\times \Lambda_{ph}$.
Note that \eqref{eq:auxiliary-initial-condition-semidiscrete-1} is well-posed as a direct application of Theorem~\ref{thm:system_solvability}.
Note also that $\btheta_{h,0}$ is data for this problem.

3. Define $(\bsi_{ph,0}, \bomega_{h,0}, \bbeta_{ph,0}, \brho_{ph,0})\in \bbX_{ph}\times \bLambda_{sh}\times \bV_{sh}\times \bbQ_{ph}$, as the unique solution of the problem
\begin{align}\label{eq:auxiliary-initial-condition-semidiscrete-2}
  & \ds (A(\bsi_{ph,0}),\btau_{ph})_{\Omega_p} + b_{\bn_p}(\btau_{ph},\bomega_{h,0})
  + b_s(\btau_{ph},\bbeta_{ph,0}) + b_{\skp}(\btau_{ph},\brho_{ph,0})
  + (A(\alpha_p\,p_{ph,0}\,\bI),\btau_{ph})_{\Omega_p} \nonumber \\[0.5ex]
  & \ds\quad =\, (A(\bsi_{p,0}),\btau_{ph})_{\Omega_p} + b_{\bn_p}(\btau_{ph},\bomega_{0})
  + b_s(\btau_{ph},\bbeta_{p,0}) + b_{\skp}(\btau_{ph},\brho_{p,0})
  + (A(\alpha_p\,p_{p,0}\,\bI),\btau_{ph})_{\Omega_p} = 0, \nonumber \\
& \ds -b_{\bn_p}(\bsi_{ph,0},\bphi_h) +  
\mu\,\alpha_{\BJS}\sum^{n-1}_{j=1} \pil\sqrt{\bK_j^{-1}}(\bvarphi_{h,0} - \btheta_{h,0})\cdot\bt_{f,j},\bphi_h\cdot\bt_{f,j}\pir_{\Gamma_{fp}} 
+ \pil\bphi_h\cdot\bn_p,\lambda_{h,0}\pir_{\Gamma_{fp}} \\
& \ds\quad = -b_{\bn_p}(\bsi_{p,0},\bphi_h) +  
\mu\,\alpha_{\BJS}\sum^{n-1}_{j=1} \pil\sqrt{\bK_j^{-1}}(\bvarphi_{0} - \btheta_{0})\cdot\bt_{f,j},\bphi_h\cdot\bt_{f,j}\pir_{\Gamma_{fp}} 
+ \pil\bphi_h\cdot\bn_p,\lambda_{0}\pir_{\Gamma_{fp}} = 0, \nonumber \\
& \ds -b_s(\bsi_{ph,0},\bv_{sh}) - b_{\skp}(\bsi_{ph,0},\bchi_{ph}) 
= -b_s(\bsi_{p,0},\bv_{sh}) - b_{\skp}(\bsi_{p,0},\bchi_{ph})
= (\f_{p}(0),\bv_{sh})_{\Omega_p}, \nonumber
\end{align}
for all $(\btau_{ph}, \bphi_{h}, \bv_{sh}, \bchi_{ph})\in \bbX_{ph}\times \bLambda_{sh}\times \bV_{sh}\times \bbQ_{ph}$.
Note that the well-posedness of \eqref{eq:auxiliary-initial-condition-semidiscrete-2} follows from the
classical Babu{\v s}ka-Brezzi theory.
Note also that $p_{ph,0}, \bvarphi_{h,0}, \btheta_{h,0}$, and $\lambda_{h,0}$ are data for this problem.

4. Finally, define $(\wh\bsi_{ph,0},\bu_{sh,0}, \bgamma_{ph,0}) \in
\bbX_{ph}\times \bV_{sh}\times \bbQ_{ph}$, as the unique solution of the problem
\begin{equation}\label{eq:auxiliary-initial-condition-semidiscrete-3}
\begin{array}{l}
\ds (A(\wh\bsi_{ph,0}),\btau_{ph})_{\Omega_p} + b_s(\btau_{ph},\bu_{sh,0})
+ b_{\skp}(\btau_{ph},\bgamma_{ph,0}) \,=\, - b_{\bn_p}(\btau_{ph},\btheta_{h,0}) \,, \\ [2ex]
-b_s(\wh\bsi_{ph,0},\bv_{sh}) - b_{\skp}(\wh\bsi_{ph,0},\bchi_{ph}) \,=\, 0 \,,
\end{array}
\end{equation}
for all $(\btau_{ph},\bv_{sh}, \bchi_{ph})\in \bbX_{ph}\times \bV_{sh}\times \bbQ_{ph}$.
Problem \eqref{eq:auxiliary-initial-condition-semidiscrete-3} is well-posed 
as a direct application of the classical Babu{\v s}ka-Brezzi theory.
Note that $\btheta_{h,0}$ is data for this problem.

\medskip
We then define $\ubsi_{h,0} = (\bsi_{fh,0},\bu_{ph,0},\bsi_{ph,0},p_{ph,0})\in \bX_h, \ubvarphi_{h,0}= (\bvarphi_{h,0},\btheta_{h,0},\lambda_{h,0})\in \bY_h$, and $\ubu_{h,0} = (\bu_{fh,0},\bu_{sh,0},\bgamma_{fh,0},\bgamma_{ph,0})\in \bZ_h$.
This construction guarantees that the discrete initial data is compatible in the sense of Lemma~\ref{lem:sol0-in-M-operator}:
\begin{equation}\label{eq:initial-condition-semidiscrete}
\begin{array}{llll}
\ds \cA(\ubsi_{h,0})(\ubtau_h) + \cB_1(\ubtau_h)(\ubvarphi_{h,0}) + \cB(\ubtau_h)(\ubu_{h,0}) & = & \wh{\bF}_{h,0}(\ubtau_h) & \forall\,\ubtau_h\in \bX_h, \\ [2ex]
\ds -\,\cB_1(\ubsi_{h,0})(\ubpsi_h) + \cC(\ubvarphi_{h,0})(\ubpsi_h) & = & 0 & \forall\,\ubpsi_h\in \bY_h, \\ [2ex]
\ds -\,\cB\,(\ubsi_{h,0})(\ubv_h) & = & \bG_0(\ubv_h) & \forall\,\ubv_h\in \bZ_h,
\end{array}
\end{equation}
where $\wh{\bF}_{h,0} = (q_f(0),\0,\wh{\f}_{ph,0},\wh{q}_{ph,0})^\rt\in \bX'_2$ and $\bG_0=\bG(0)\in \bZ'$, with $\wh{\f}_{ph,0}\in \bbX'_{p,2}$ and $\wh{q}_{ph,0}\in \W'_{p,2}$ suitable data.
Furthermore, it provides compatible initial data for the non-differentiated elasticity variables $(\bbeta_{ph,0},\brho_{ph,0},\bomega_{h,0})$ in the sense of the first equation in \eqref{eq:system-sol0-3} (cf. \eqref{eq:auxiliary-initial-condition-semidiscrete-2}).

\subsection{Existence and uniqueness of a solution}

Now, we establish the well-posedness of problem \eqref{eq:Stokes-Biot-formulation-semidiscrete} and the corresponding stability bound.
\begin{thm}\label{thm:well-posedness-semidiscrete}
For each compatible initial data $(\ubsi_{h,0}, \ubvarphi_{h,0}, \ubu_{h,0})$ satisfying \eqref{eq:initial-condition-semidiscrete} and
\begin{equation*}
\f_f\in \W^{1,1}(0,T;\bV_f'),\quad \f_p\in \W^{1,1}(0,T;\bV_s'),\quad q_f\in \W^{1,1}(0,T;\bbX'_f),\quad q_p\in \W^{1,1}(0,T;\W'_p)\,,
\end{equation*}	
there exists a unique solution of \eqref{eq:Stokes-Biot-formulation-semidiscrete}, $(\ubsi_h,\ubvarphi_h,\ubu_h):[0,T]\to \bX_h\times \bY_h\times \bZ_h$ such that
$(\bsi_{ph},p_{ph}) \in \W^{1,\infty}(0,T;\bbX_{ph}) \times  \W^{1,\infty}(0,T;\W_{ph})$, and $(\ubsi_h(0), \ubvarphi_h(0), \bu_{fh}(0), \bgamma_{fh}(0)) = (\ubsi_{h,0}, \ubvarphi_{h,0}, \bu_{fh,0}, \bgamma_{fh,0})$. Moreover, assuming sufficient regularity of the data, there exists a positive constant $C$ independent of $h$ and $s_0$, such that
\begin{align}\label{eq:discrete-stability-bound}
& \ds \|\bsi_{fh}\|_{\L^\infty(0,T;\bbX_f)} + \|\bsi_{fh}\|_{\L^2(0,T;\bbX_f)} + \|\bu_{ph}\|_{\L^\infty(0,T;\bL^2(\Omega_p))} + \|\bu_{ph}\|_{\L^2(0,T;\bV_p)}
+ |\bvarphi_h - \btheta_h|_{\L^\infty(0,T;\BJS)} \nonumber \\[0.5ex] 
& \ds\quad+ |\bvarphi_h - \btheta_h|_{\L^2(0,T;\BJS)} 
+\|\lambda_h\|_{\L^\infty(0,T;\Lambda_{ph})}
+ \|\ubvarphi_h\|_{\L^2(0,T;\bY_h)}
+ \|\ubu_h\|_{\L^2(0,T;\bZ)}
+ \|A^{1/2}(\bsi_{ph})\|_{\L^\infty(0,T;\bbL^2(\Omega_p))} \nonumber \\[0.5ex]  
& \ds\quad
+ \|\bdiv(\bsi_{ph})\|_{\L^\infty(0,T;\bL^2(\Omega_p))} 
+ \|\bdiv(\bsi_{ph})\|_{\L^2(0,T;\bL^2(\Omega_p))}
+ \|p_{ph}\|_{\L^\infty(0,T;\W_p)} 
+ \|p_{ph}\|_{\L^2(0,T;\W_p)} \nonumber \\[0.5ex]
& \ds \quad 
+ \|\partial_t\,A^{1/2}(\bsi_{ph}+\alpha_p p_{ph} \bI)\|_{\L^2(0,T;\bbL^2(\Omega_p))} 
+ \sqrt{s_0}\|\partial_t\,p_{ph}\|_{\L^2(0,T;\W_p)}  \\[0.5ex] 
& \ds\leq C\,\Big( \|\f_f\|_{\H^1(0,T;\bV'_f)} + \|\f_p\|_{\H^1(0,T;\bV'_s)} + \|q_f\|_{\H^1(0,T;\bbX'_f)} + \|q_p\|_{\H^1(0,T;\W'_p)} \nonumber \\[0.5ex]
& \ds \quad \qquad
+ (1+\sqrt{s_0})\|p_{p,0}\|_{\W_p}
+ \|\bK\nabla p_{p,0}\|_{\H^1(\Omega_p)} \Big). \nonumber
\end{align}
\end{thm}
\begin{proof}
  From the fact that $\bX_h\subset \bX$, $\bZ_h\subset \bZ$,
  and $\bdiv(\bbX_{fh}) = \bV_{fh}$, 
$\bdiv(\bbX_{ph}) = \bV_{sh}$, $\div(\bV_{ph}) = \W_{ph}$, considering 
  $(\ubsi_{h,0}, \ubvarphi_{h,0}, \ubu_{h,0})$ satisfying
  \eqref{eq:initial-condition-semidiscrete}, 
  and employing the continuity and monotonicity properties of the
  operators $\cN$ and $\cM$ (cf. Lemma~\ref{lem:N-M-properties}
and \eqref{cont-B1-discrete}), as well as the discrete inf-sup conditions
\eqref{eq:inf-sup-B-semidiscrete}, \eqref{eq:inf-sup-b1-semidiscrete}, and 
\eqref{eq:p-lambda-bound semidiscrete}, the proof is identical to the proofs of 
Theorems \ref{thm:well-posdness-continuous} and \ref{thm:stability-bound}, 
and Corollary~\ref{cor:initial-data}.
We note that the proof of Corollary \ref{cor:initial-data} works in the discrete case 
due to the choice of the discrete initial data as the elliptic projection of the 
continuous initial data
(cf. \eqref{eq:auxiliary-initial-condition-semidiscrete-1}--\eqref{eq:auxiliary-initial-condition-semidiscrete-3}).
\end{proof}

\begin{rem}\label{rem:post-processed-formulae}
As in the continuous case, we can recover the non-differentiated elasticity variables
\begin{equation*}
\bbeta_{ph}(t) = \bbeta_{ph,0} + \int^t_0 \bu_{sh}(s)\,ds,\quad
\brho_{ph}(t) = \brho_{ph,0} + \int^t_0 \bgamma_{ph}(s)\,ds,\quad
\bomega_h(t) = \bomega_{h,0} + \int^t_0 \btheta_h(s)\,ds\,,
\end{equation*}
for each $t\in [0,T]$.
Then \eqref{eq:A-sigmap-pp-without-time-derivative} holds discretely, which follows from integrating the equation associated to $\btau_{ph}$ in \eqref{eq:Stokes-Biot-formulation-semidiscrete} from $0$ to $t\in (0,T]$ and using the first equation in \eqref{eq:auxiliary-initial-condition-semidiscrete-2} (cf. \eqref{eq:system-sol0-3}).
\end{rem}

\subsection{Error analysis}

We proceed with establishing rates of convergence.
To that end, let us set $\V\in \big\{ \W_p, \bV_f, \bV_s, \bbQ_f, \bbQ_p \big\}$, 
$\Lambda\in \big\{ \bLambda_f, \bLambda_s, \Lambda_p \big\}$ and let
$\V_h, \Lambda_h$ be the discrete counterparts. Let $P_h^{\V}: \V\to \V_h$ and 
$P_h^{\Lambda}: \Lambda\to \Lambda_h$ be the $\L^2$-projection operators, satisfying
\begin{equation}\label{eq:projection1}
\begin{array}{rl}
( u - P_h^{\V}(u), v_{h} )_{\Omega_\star} \,=\, 0   & \forall \, v_h\in \V_{h}, \\ [2ex]
\langle \varphi - P_h^{\Lambda}(\varphi), \psi_h \rangle_{\Gamma_{fp}} \,=\, 0 & \forall\, \psi_h\in \Lambda_{h},
\end{array}
\end{equation}
where $\star\in \{f,p\}$, $u\in \big\{ p_p, \bu_f, \bu_s, \bgamma_f,
\bgamma_p \big\}$, $\varphi\in \big\{ \bvarphi, \btheta, \lambda
\big\}$, and $v_h, \psi_h$ are the corresponding discrete test
functions. We have the approximation properties \cite{ciar1978}:
\begin{equation}\label{eq:approx-property1}
\begin{array}{c}
\|u - P^{\V}_h(u)\|_{\L^2(\Omega_\star)} \,\leq\, Ch^{s_{u} + 1}\, \|u\|_{\H^{s_{u} + 1}(\Omega_\star)}, \\ [2ex]
\|\varphi - P^{\Lambda}_h(\varphi)\|_{\Lambda_h} \,\leq\, Ch^{s_{\varphi} + r}\,
\|\varphi\|_{\H^{s_{\varphi} + 1}(\Gamma_{fp})},
\end{array}
\end{equation}
where $s_{u}\in \big\{ s_{p_p}, s_{\bu_f}, s_{\bu_s}, s_{\bgamma_f},
s_{\bgamma_p} \big\}$ and $s_{\varphi}\in \big\{ s_{\bvarphi},
s_{\btheta}, s_{\lambda} \big\}$ are the degrees of polynomials in the
spaces $\V_h$ and $\Lambda_h$, respectively, and (cf. \eqref{eq:norm-in-Yh}), 
\begin{equation*}
\|\varphi\|_{\Lambda_h} := \left\{\begin{array}{l}
\|\varphi\|_{\H^{1/2}(\Gamma_{fp})}, \mbox{ with } r=1/2 \mbox{ in } \eqref{eq:approx-property1} \mbox{ for conforming spaces } \textbf{(S1)}, \\[1.5ex]
\|\varphi\|_{\L^2(\Gamma_{fp})}, \mbox{ with } r=1 \mbox{ in } \eqref{eq:approx-property1} \mbox{ for non-conforming spaces } \textbf{(S2)}.
\end{array}\right.
\end{equation*}

Next, denote $\X\in \big\{ \bbX_f, \bbX_p, \bV_p \big\}$, $\sigma\in
\big\{ \bsi_f, \bsi_p, \bu_p \big\}\in \X$ and let $\X_h$ and $\tau_h$ be
their discrete counterparts. 
For the case \textbf{(S2)} when the discrete Lagrange multiplier spaces 
are chosen as in \eqref{eq:non-conforming-spaces-Yh}, \eqref{eq:projection1} implies
\begin{equation}\label{eq:projection2}
\langle \varphi - P_h^{\Lambda}(\varphi), \tau_{h}\bn_\star \rangle_{\Gamma_{fp}} \,=\, 0 \quad \forall\, \tau_{h}\in \X_{h},
\end{equation}
where $\star\in \{f,p\}$. We note that \eqref{eq:projection2} does not hold
for the case \textbf{(S1)}.

Let $I^{\X}_h : \X \cap \H^{1}(\Omega_{\star})\to \X_{h}$ be the mixed finite element
projection operator \cite{Brezzi-Fortin} satisfying
\begin{equation}\label{eq:projection3}
\begin{array}{cl}
(\div(I^{\X}_h(\sigma)), w_h)_{\Omega_\star} = (\div(\sigma), w_h)_{\Omega_\star} & \forall\,w_h\in \W_h, \\ [2ex]
\pil I^{\X}_h(\sigma)\bn_{\star}, \tau_h\bn_{\star} \pir_{\Gamma_{fp}} = \pil \sigma \bn_{\star}, \tau_{h}\bn_{\star} \pir_{\Gamma_{fp}} & \forall\,\tau_h\in \X_h, 
\end{array}
\end{equation}
and
\begin{equation}\label{eq:approx-property2}
\begin{array}{c}
\|\sigma - I^{\X}_h(\sigma)\|_{\L^2(\Omega_{\star})} \,\leq\, C\,h^{s_{\sigma} + 1} \| \sigma \|_{\H^{s_{\sigma} + 1}(\Omega_{\star})}, \\ [2ex]
\|\div(\sigma - I^{\X}_h(\sigma))\|_{\L^2(\Omega_{\star})} \,\leq\, C\,h^{s_{\sigma} + 1} \|\div(\sigma)\|_{\H^{s_{\sigma} + 1}(\Omega_{\star})},
\end{array}
\end{equation}
where $w_h\in \big\{ \bv_{fh}, \bv_{sh}, w_{ph} \big\}$, $\W_h\in
\big\{ \bV_f, \bV_s, \W_p \big\}$, and $s_{\sigma}\in \big\{
s_{\bsi_f}, s_{\bsi_p}, s_{\bu_p} \big\}$ -- the degrees of polynomials
in the spaces $\X_h$.

Now, let $(\bsi_f, \bu_p, \bsi_p, p_p, \bvarphi, \btheta, \lambda,
\bu_f, \bu_s, \bgamma_f, \bgamma_p)$ and $(\bsi_{fh}, \bu_{ph},
\bsi_{ph}, p_{ph}, \bvarphi_h, \btheta_h, \lambda_h, \bu_{fh},
\bu_{sh}, \bgamma_{fh}, \bgamma_{ph})$ be the solutions of
\eqref{eq:evolution-problem-in-operator-form} and
\eqref{eq:Stokes-Biot-formulation-semidiscrete}, respectively.  We
introduce the error terms as the differences of these two solutions and
decompose them into approximation and discretization errors
using the interpolation operators:
\begin{equation}\label{eq:error-decomposition}
\begin{array}{l}
  \ds \be_{\sigma} \,:=\, \sigma - \sigma_{h} \,=\,
  (\sigma - I^{\X}_h(\sigma)) + (I^{\X}_h(\sigma) - \sigma_{h})
  \,:=\, \be^I_{\sigma} + \be^h_{\sigma},
  \quad \sigma\in \big\{ \bsi_f, \bsi_p, \bu_p \big\}, \\ [2ex]
  \ds \be_{\varphi} \,:=\, \varphi - \varphi_h
  \,=\, (\varphi - P^{\Lambda}_h(\varphi)) + (P^{\Lambda}_h(\varphi) - \varphi_h)
  \,:=\, \be^I_{\varphi} + \be^h_{\varphi},
\quad \varphi\in \big\{ \bvarphi, \btheta, \lambda \big\},
  \\ [2ex]
  \ds \be_{u} \,:=\, u - u_{h} \,=\, (u - P^{\V}_h(u)) + (P^{\V}_h(u) - u_{h})
  \,:=\, \be^I_{u} + \be^h_{u}, \quad
  u\in \big\{ p_p, \bu_f, \bu_s, \bgamma_f, \bgamma_p \big\}.
\end{array}
\end{equation}
Then, we set the errors
\begin{equation*}
\be_{\ubsi} := (\be_{\bsi_f}, \be_{\bu_p}, \be_{\bsi_p}, \be_{p_p}),\quad
\be_{\ubvarphi} := (\be_{\bvarphi}, \be_{\btheta}, \be_{\lambda}),\qan
\be_{\ubu} := (\be_{\bu_f}, \be_{\bu_s}, \be_{\bgamma_f}, \be_{\bgamma_p}).
\end{equation*}
We next form the error system by subtracting the discrete problem
\eqref{eq:Stokes-Biot-formulation-semidiscrete} from the continuous
one \eqref{eq:evolution-problem-in-operator-form}. Using that $\bX_h\subset \bX$ and
$\bZ_h\subset \bZ$, as well as Remark~\ref{rem:non-conf}, we obtain
\begin{equation}\label{eq:error-equation}
\begin{array}{lll}
\ds (\partial_t\,\cE + \cA)(\be_{\ubsi})(\ubtau_h) + \cB_1(\ubtau_h)(\be_{\ubvarphi}) + \cB(\ubtau_h)(\be_{\ubu}) & = & 0 \quad \forall\,\ubtau_h\in \bX_h, \\ [2ex]
\ds -\,\cB_1(\be_{\ubsi})(\ubpsi_h) + \cC(\be_{\ubvarphi})(\ubpsi_h) & = & 0 \quad\forall\,\ubpsi_h\in \bY_h, \\ [2ex]
\ds -\,\cB(\be_{\ubsi})(\ubv_h) & = & 0 \quad \forall\,\ubv_h\in \bZ_h.
\end{array}
\end{equation}

We now establish the main result of this section.
\begin{thm}\label{thm:error-analysis-result}
For the solutions of the continuous and discrete problems
\eqref{eq:evolution-problem-in-operator-form} and
\eqref{eq:Stokes-Biot-formulation-semidiscrete}, respectively,
assuming sufficient regularity of the true solution
according to \eqref{eq:approx-property1} and \eqref{eq:approx-property2},
there exists a
positive constant $C$ independent of $h$ and $s_0$, such that
\begin{align}\label{eq:error-analysis-result}
& \ds \|\be_{\bsi_f}\|_{\L^\infty(0,T;\bbX_{f})}
+ \|\be_{\bsi_f}\|_{\L^2(0,T;\bbX_{f})}
+ \|\be_{\bu_p}\|_{\L^\infty(0,T;\bL^2(\Omega_p))}
+ \|\be_{\bu_p}\|_{\L^2(0,T;\bV_{p})} 
+ |\be_{\bvarphi} - \be_{\btheta}|_{\L^\infty(0,T;\BJS)} \nonumber \\[0.5ex]
& \ds\quad
+ |\be_{\bvarphi} - \be_{\btheta}|_{\L^2(0,T;\BJS)}
+ \|\be_{\lambda}\|_{\L^\infty(0,T;\Lambda_{ph})}
+ \|\be_{\ubvarphi}\|_{\L^2(0,T;\bY_h)}
+ \|\be_{\ubu}\|_{\L^2(0,T;\bZ)}
+ \|A^{1/2}(\be_{\bsi_p})\|_{\L^\infty(0,T;\bbL^2(\Omega_{p}))} \nonumber \\[0.5ex]
& \ds\quad
+ \|\bdiv(\be_{\bsi_p})\|_{\L^\infty(0,T;\bL^2(\Omega_p))}
+ \|\bdiv(\be_{\bsi_p})\|_{\L^2(0,T;\bL^2(\Omega_p))}
+ \|\be_{p_p}\|_{\L^\infty(0,T;\W_{p})} 
+ \|\be_{p_p}\|_{\L^2(0,T;\W_{p})} \nonumber \\[0.5ex]
& \ds\quad 
+ \|\partial_t\,A^{1/2}(\be_{\bsi_p} + \alpha_p \be_{p_p}\bI)\|_{\L^2(0,T;\bbL^2(\Omega_p))}
+ \sqrt{s_0}\|\partial_t\,\be_{p_p}\|_{\L^2(0,T;\W_{p})} \nonumber \\[0.5ex]
& \ds \leq\,\, 
C\,\sqrt{\exp(T)}\,\Big( h^{s_{\ubsi}+1} 
+ h^{s_{\ubvarphi}+r} 
+ h^{s_{\ubu}+1} 
\Big),
\end{align}
where $s_{\ubsi} = \min\{ s_{\bsi_f}, s_{\bu_p}, s_{\bsi_p}, s_{p_p} \}$, $s_{\ubvarphi} = \min\{ s_{\bvarphi}, s_{\btheta}, s_{\lambda} \}$, $s_{\ubu} = \min\{ s_{\bu_f}, s_{\bu_s}, s_{\bgamma_f}, s_{\bgamma_p} \}$, and $r$ is defined in \eqref{eq:approx-property1}.
\end{thm}
\begin{proof}
We present in detail the proof for the conforming case \textbf{(S1)}.
The proof in the non-conforming case \textbf{(S2)} is simpler, since
several error terms are zero. We explain the differences at the end of
the proof.
  
We proceed as in Theorem~\ref{thm:stability-bound}.
Taking $(\ubtau_{h}, \ubpsi_h, \ubv_{h}) = (\be^h_{\ubsi}, \be^h_{\ubvarphi}, \be^h_{\ubu})$ in \eqref{eq:error-equation}, we obtain
\begin{align}\label{eq:error-equation3}
& \ds \frac{1}{2}\,\partial_t\left( a_e(\be^h_{\bsi_p},\be^h_{p_p};\be^h_{\bsi_p},\be^h_{p_p})
+ s_0\,(\be^h_{p_p}, \be^h_{p_p})_{\Omega_p} \right) 
+ a_f(\be^h_{\bsi_f}, \be^h_{\bsi_f}) 
+ a_p(\be^h_{\bu_p}, \be^h_{\bu_p}) 
+ c_{\BJS}(\be^h_{\bvarphi}, \be^h_{\btheta};\be^h_{\bvarphi}, \be^h_{\btheta}) \nonumber \\[1ex]
& \ds =\,-\,a_f(\be^I_{\bsi_f}, \be^h_{\bsi_f})
- a_p(\be^I_{\bu_p}, \be^h_{\bu_p})
- a_e(\partial_t\,\be^I_{\bsi_p}, \partial_t\,\be^I_{p_p}; \be^h_{\bsi_p}, \be^h_{p_p})
- \cC(\be^I_{\ubvarphi})(\be^h_{\ubvarphi}) \nonumber \\[1ex]
& \ds\quad\, 
-\,b_{\bn_f}(\be^h_{\bsi_f},\be^I_{\bvarphi})
-b_{\bn_p}(\be^h_{\bsi_p},\be^I_{\btheta})
-b_{\Gamma}(\be^h_{\bu_p},\be^I_{\lambda})
+b_{\bn_f}(\be^I_{\bsi_f},\be^h_{\bvarphi})
+b_{\bn_p}(\be^I_{\bsi_p},\be^h_{\btheta})
+b_{\Gamma}(\be^I_{\bu_p},\be^h_{\lambda}) \\[1ex]
& \ds\quad\, -\,b_\skf(\be^h_{\bsi_f},\be^I_{\bgamma_f}) 
- b_\skp(\be^h_{\bsi_p},\be^I_{\bgamma_p})
+ b_\skf(\be^I_{\bsi_f},\be^h_{\bgamma_f})
+ \, b_\skp(\be^I_{\bsi_p},\be^h_{\bgamma_p}), \nonumber
\end{align}
where, the right-hand side of \eqref{eq:error-equation3} has been
simplified, since the projection properties \eqref{eq:projection1} and
\eqref{eq:projection3}, and the fact that $\div(\be^h_{\bu_p})\in
\W_{ph}$, $\bdiv(\be^h_{\bsi_f})\in \bV_{fh}$, and
$\bdiv(\be^h_{\bsi_p})\in \bV_{sh}$, imply that the following terms
are zero:
\begin{equation}\label{zero-terms}
\ds s_0(\partial_t\,\be^I_{p_p},\be^h_{p_p}),\,
b_p(\be^h_{\bu_p},\be^I_{p_p}),\,
b_p(\be^I_{\bu_p},\be^h_{p_p}),\,
b_f(\be^h_{\bsi_f},\be^I_{\bu_f}),\,
b_f(\be^I_{\bsi_f},\be^h_{\bu_f}),\,
b_s(\be^h_{\bsi_p},\be^I_{\bu_s}),\,
b_s(\be^I_{\bsi_p},\be^h_{\bu_s}).
\end{equation}
In turn, from the equations in \eqref{eq:error-equation} corresponding to
test functions $\bv_{fh}$, $\bv_{sh}$, and $w_{ph}$,
using the projection properties \eqref{eq:projection3}, we find
that
\begin{equation*}
\begin{array}{c}
\ds b_f(\be^h_{\bsi_f}, \bv_{fh}) = 0 \quad \forall\,\bv_{fh}\in \bV_{fh},\quad
b_s(\be^h_{\bsi_p}, \bv_{sh}) = 0 \quad \forall\,\bv_{sh}\in \bV_{sh}, \\ [2ex]
\ds b_p(\be^h_{\bu_p}, w_{ph}) = a_e(\partial_t\,\be^h_{\bsi_p}, \partial_t\,\be^h_{p_p};\0,w_{ph}) +
a_e(\partial_t\,\be^I_{\bsi_p}, \partial_t\,\be^I_{p_p};\0,w_{ph}) +(s_0\,\partial_t\,\be^h_{p_p}, w_{ph})_{\Omega_p} \quad \forall\, w_{ph}\in \W_{ph}.
\end{array}
\end{equation*}
Therefore $\bdiv(\be^h_{\bsi_\star}) = \0$ in $\Omega_\star$, with
$\star\in \{f,p\}$, and using
\eqref{eq:tau-d-H0div-inequality}--\eqref{eq:tau-H0div-Xf-inequality}
we deduce
\begin{equation}\label{eq:error-analysis7}
\begin{array}{c}
\| (\be^h_{\bsi_f})^\rd\|^2_{\bbL^2(\Omega_f)} \,\geq\, C\,\| \be^h_{\bsi_f} \|^2_{\bbX_f},\quad
\|\bdiv(\be^h_{\bsi_p})\|_{\bL^2(\Omega_p)} \,=\, 0\,, \\ [1.5ex]
\ds \|\div(\be^h_{\bu_p})\|_{\L^2(\Omega_p)} 
\,\leq\, C\,\Big( \| \partial_t\,A^{1/2}(\be^I_{\bsi_p} + \alpha_p\,\be^I_{p_p}\bI) \|_{\bbL^2(\Omega_p)} \\[1.5ex]
\ds \qquad \qquad 
+\,\| \partial_t\,A^{1/2}(\be^h_{\bsi_p} + \alpha_p\,\be^h_{p_p}\bI) \|_{\bbL^2(\Omega_p)}
+ \sqrt{s_0}\,\|\partial_t\,\be^h_{p_p} \|_{\W_p} \Big)\,.
\end{array}
\end{equation}
Then, applying the ellipticity and continuity bounds of the bilinear
forms involved in \eqref{eq:error-equation3}
(cf. Lemma~\ref{lem:N-M-properties}) and the Cauchy--Schwarz and
Young's inequalities, in combination with \eqref{eq:error-analysis7},
we get
\begin{equation*}
\begin{array}{l}
\ds \partial_t\left(\|A^{1/2}(\be^h_{\bsi_p} + \alpha_p \be^h_{p_p}\bI)\|^2_{\bbL^2(\Omega_p)}
+ s_0 \|\be^h_{p_p}\|^2_{\W_p} \right)
+ \|\be^h_{\bsi_f}\|^2_{\bbX_f}
+ \|\be^h_{\bu_p}\|^2_{\bV_p}
+ \|\bdiv(\be^h_{\bsi_p})\|^2_{\bL^2(\Omega_p)}
+ |\be^h_{\bvarphi} - \be^h_{\btheta}|^2_{\BJS}  \\[2ex]
\ds \leq\, C\,\Big(
\|\be^I_{\bsi_f}\|^2_{\bbX_f}
+ \|\be^I_{\bu_p}\|^2_{\bV_p} 
+ \|\be^I_{\bsi_p}\|^2_{\bbX_p}
+ |\be^I_{\bvarphi} - \be^I_{\btheta}|^2_{\BJS}
+ \|\be^I_{\ubvarphi}\|^2_{\bY_h}
+ \|\be^I_{\bgamma_f}\|^2_{\bbQ_f}
+ \|\be^I_{\bgamma_p}\|^2_{\bbQ_p}
\\ [2ex]
\ds\quad
+\,\,\|\partial_t\,A^{1/2}\,(\be^I_{\bsi_p} + \alpha_p\,\be^I_{p_p}\bI)\|^2_{\bbL^2(\Omega_p)}
+ \|A^{1/2}\,(\be^h_{\bsi_p} + \alpha_p\,\be^h_{p_p}\bI)\|^2_{\bbL^2(\Omega_p)} 
 \\ [2ex]
\ds\quad 
+\,\,\|\partial_t\,A^{1/2}\,(\be^h_{\bsi_p} + \alpha_p\,\be^h_{p_p}\bI)\|^2_{\bbL^2(\Omega_p)}
+ s_0\|\partial_t\,\be^h_{p_p}\|^2_{\W_p} \Big) \\ [2ex]
\ds +\,\,\delta_1\Big(
\|\be^h_{\bsi_f}\|^2_{\bbX_f}
+ \|\be^h_{\bu_p}\|^2_{\bV_p}
+ |\be^h_{\bvarphi} - \be^h_{\btheta}|^2_{\BJS} \Big)
+ \delta_2\,\Big( 
\|\be^h_{\bsi_p}\|^2_{\bbL^2(\Omega_p)}
+ \|\be^h_{\ubvarphi}\|^2_{\bY_h}
+ \|\be^h_{\bgamma_f}\|^2_{\bbQ_f}
+ \|\be^h_{\bgamma_p}\|^2_{\bbQ_p} \Big),
\end{array}
\end{equation*}
where for the bound on $b_{\bn_p}(\be^h_{\bsi_p},\be^I_{\btheta})$ we
used the trace inequality \eqref{trace-sigma} and the fact that
$\bdiv(\be^h_{\bsi_p}) = \0$. Next, integrating from $0$ to $t\in (0,T]$, using
  \eqref{eq:aux-taup-qp-inequality} to control the term
  $\|\be^h_{\bsi_p}\|^2_{\bbL^2(\Omega_p)}$, and choosing $\delta_1$
  small enough, we find that
  \begin{align}\label{eq:error-analysis1}
& \ds \|A^{1/2}(\be^h_{\bsi_p} + \alpha_p \be^h_{p_p}\bI)(t)\|^2_{\bbL^2(\Omega_p)}
+ s_0\|\be^h_{p_p}(t)\|^2_{\W_p} \nonumber \\
& \ds\quad
+ \int^t_0 \Big( \|\be^h_{\bsi_f}\|^2_{\bbX_f}
+ \|\be^h_{\bu_p}\|^2_{\bV_p}
+ \|\bdiv(\be^h_{\bsi_p})\|^2_{\bL^2(\Omega_p)}
+ |\be^h_{\bvarphi} - \be^h_{\btheta}|^2_{\BJS} \Big)\, ds \nonumber \\ 
& \ds
\leq C\,\Bigg( \int^t_0 \Big( 
\|\be^I_{\bsi_f}\|^2_{\bbX_f}
+ \|\be^I_{\bu_p}\|^2_{\bV_p}
+ |\be^I_{\bvarphi} - \be^I_{\btheta}|^2_{\BJS}
+ \|\be^I_{\ubvarphi}\|^2_{\bY_h} 
+ \|\be^I_{\bgamma_f}\|^2_{\bbQ_f}
+ \|\be^I_{\bgamma_p}\|^2_{\bbQ_p}
+ \|\be^I_{\bsi_p}\|^2_{\bbX_p}
\Big)\,ds \nonumber \\[1ex]
& \ds \quad
+ \int^t_0 \Big( 
\|\partial_t\,A^{1/2}\,(\be^I_{\bsi_p} + \alpha_p\,\be^I_{p_p}\bI)\|^2_{\bbL^2(\Omega_p)}
+ \|A^{1/2}\,(\be^h_{\bsi_p} + \alpha_p\,\be^h_{p_p}\bI)\|^2_{\bbL^2(\Omega_p)} \Big)\,ds \\
& \ds\quad +\,\int^t_0 \Big( 
\|\partial_t\,A^{1/2}\,(\be^h_{\bsi_p} + \alpha_p\,\be^h_{p_p}\bI)\|^2_{\bbL^2(\Omega_p)} 
+ s_0\|\partial_t\,\be^h_{p_p}\|^2_{\W_p} \Big)\,ds 
+ \|A^{1/2}(\be^h_{\bsi_p} + \alpha_p\,\be^h_{p_p}\,\bI)(0)\|^2_{\bbL^2(\Omega_p)} \nonumber \\
& \ds\quad
+\,s_0\|\be^h_{p_p}(0)\|^2_{\W_p} \Bigg)
+ \delta_2 \int^t_0 \Big( 
\|\be^h_{p_p}\|^2_{\W_p}
+ \|\be^h_{\ubvarphi}\|^2_{\bY_h}
+ \|\be^h_{\bgamma_f}\|^2_{\bbQ_f}
+ \|\be^h_{\bgamma_p}\|^2_{\bbQ_p} \Big)\,ds\,. \nonumber
\end{align}

On the other hand, taking $\ubtau_h =
(\btau_{fh},\bv_{ph},\btau_{ph},0)\in \bV_h$
(cf. \eqref{eq:discrete-kernel-B}) in the first equation of
\eqref{eq:error-equation}, we obtain
\begin{equation*}
\cB_1(\ubtau_h)(\be^h_{\ubvarphi})
\,=\, -\,(\partial_t\,\cE + \cA)(\be_{\ubsi})(\ubtau_h) - \cB_1(\ubtau_h)(\be^I_{\ubvarphi})\,,
\end{equation*}
In the above, thanks to the projection properties \eqref{eq:projection1},
the following terms are zero:
$b_p(\bv_{ph},\be^I_{p_p})$, $b_f(\btau_{fh}, \be^I_{\bu_f})$, and $b_s(\btau_{ph},\be^I_{\bu_s})$.
Then the discrete inf-sup condition of $\cB_1$
(cf. \eqref{eq:inf-sup-b1-semidiscrete}) for
$\be^h_{\ubvarphi} = (\be^h_{\bvarphi}, \be^h_{\btheta}, \be^h_{\lambda})\in \bY_h$
gives
\begin{align}\label{eq:error-inf-sup1}
& \ds \|\be^h_{\ubvarphi}\|_{\bY_h} 
\,\leq\, C \Big( 
\|\be^I_{\bsi_f}\|_{\bbX_f} 
+ \|\be^I_{\bu_p}\|_{\bV_p} 
+\|\be^I_{\ubvarphi}\|_{\bY_h}
+ \|\be^I_{\bgamma_f}\|^2_{\bbQ_f}
+ \|\be^I_{\bgamma_p}\|^2_{\bbQ_p}
+ \|\partial_t\,A^{1/2}\,(\be^I_{\bsi_p} + \alpha_p\,\be^I_{p_p}\bI)\|_{\bbL^2(\Omega_p)} \nonumber \\
& \ds\quad+\,
\|\be^h_{\bsi_f}\|_{\bbX_f} 
+ \|\be^h_{\bu_p}\|_{\bV_p}
+ \|\be^h_{\bgamma_f}\|^2_{\bbQ_f}
+ \|\be^h_{\bgamma_p}\|^2_{\bbQ_p}
+ \|\partial_t\,A^{1/2}\,(\be^h_{\bsi_p} + \alpha_p\,\be^h_{p_p}\bI)\|_{\bbL^2(\Omega_p)} 
+ \|\be^h_{p_p}\|_{\W_p} \Big)\,.
\end{align}
In turn, to bound $\|\be^h_{\ubu}\|_{\bZ}$, we test \eqref{eq:error-equation} with $\ubtau_h = (\btau_{fh},\0,\btau_{ph},0)\in \wt{\bX}_h$ (cf. \eqref{eq:subspace-bbXh-tilde}), to find that
\begin{equation*}
\cB(\ubtau_h)(\be^h_{\ubu})
\,=\, -\,\Big( a_f(\be_{\bsi_f},\btau_{fh}) + a_e(\partial_t\,\be_{\bsi_p},\partial_t\,\be_{p_p};\btau_{ph},0) + \cB(\ubtau_h)(\be^I_{\ubu}) \Big).
\end{equation*}
In the above, the terms $b_f(\btau_{fh},\be^I_{\bu_f})$ and
$b_s(\btau_{ph},\be^I_{\bu_s})$ are zero, due to 
the projection property \eqref{eq:projection1}. Then,
the discrete inf-sup condition of
$\cB$ (cf. \eqref{eq:inf-sup-B-semidiscrete}) for $\be^h_{\ubu}\in
\bZ_h$, yields
\begin{equation}\label{eq:error-inf-sup2}
\begin{array}{l}
\ds \| \be^h_{\ubu} \|_{\bZ} 
\,\leq\, 
C\,\Big( 
\|\be^I_{\bsi_f}\|_{\bbX_f} 
+ \|\partial_t\,A^{1/2}\,(\be^I_{\bsi_p} + \alpha_p\,\be^I_{p_p}\bI)\|_{\bbL^2(\Omega_p)} 
+ \|\be^I_{\bgamma_f}\|_{\bbQ_f}
+ \|\be^I_{\bgamma_p}\|_{\bbQ_p} \\ [2ex] 
\ds\qquad\qquad\qquad+\,
\|\be^h_{\bsi_f}\|_{\bbX_f}
+ \|\partial_t\,A^{1/2}\,(\be^h_{\bsi_p} + \alpha_p\,\be^h_{p_p}\bI)\|_{\bbL^2(\Omega_p)}  \Big)\,.
\end{array}
\end{equation}
Finally, to bound $\|\be^h_{p_p}\|_{\W_p}$, we test \eqref{eq:error-equation} with $\ubtau_h = (\btau_{fh},\bv_{ph},\btau_{ph},0)\in \bX_h$ to get
\begin{equation*}
b_p(\bv_{ph}, \be^h_{p_p})+b_{\Gamma}(\bv_{ph},\be^h_{\lambda}) \,=\,-\,\Big(a_p(\be_{\bu_p},\bv_{ph})+b_p(\bv_{ph}, \be^I_{p_p})+b_{\Gamma}(\bv_{ph},\be^I_{\lambda})\Big).
\end{equation*}
Note that $b_p(\bv_{ph}, \be^I_{p_p})=0$ due to the projection property \eqref{eq:projection1}, thus the discrete inf-sup condition \eqref{eq:p-lambda-bound semidiscrete} gives
\begin{equation}\label{eq:error-inf-sup3}
\|\be^h_{p_p}\|_{\W_p}+\|\be^h_{\lambda}\|_{\Lambda_{ph}}\leq C \Big( \|\be^I_{\bu_p}\|_{\bL^2(\Omega_p)}+\|\be^I_{\lambda}\|_{\Lambda_{ph}}+\|\be^h_{\bu_p}\|_{\bL^2(\Omega_p)}\Big).
\end{equation}
Combining \eqref{eq:error-analysis1} with \eqref{eq:error-inf-sup1}, \eqref{eq:error-inf-sup2}, and \eqref{eq:error-inf-sup3}, choosing $\delta_2$ small enough, and employing the Gronwall's inequality to deal with the term $\ds \int^t_0  \|A^{1/2}\,(\be^h_{\bsi_p} + \alpha_p\,\be^h_{p_p}\bI)\|^2_{\bbL^2(\Omega_p)} \,ds $, we obtain
\begin{align}\label{eq:error-analysis2}
&\ds \| A^{1/2}(\be^h_{\bsi_p} + \alpha_p\,\be^h_{p_p}\bI)(t) \|^2_{\bbL^2(\Omega_p)}
+ s_0\,\|\be^h_{p_p}(t)\|^2_{\W_p} \nonumber \\[0.5ex]
&\ds\quad +\,\int^t_0 \Big( 
\|\be^h_{\bsi_f}\|^2_{\bbX_f}
+ \|\be^h_{\bu_p}\|^2_{\bV_p}
+ \|\bdiv(\be^h_{\bsi_p})\|^2_{\bL^2(\Omega_p)} 
+ \|\be^h_{p_p}\|^2_{\W_p}
+ |\be^h_{\bvarphi} - \be^h_{\btheta}|^2_{\BJS} 
+ \|\be^h_{\ubvarphi}\|^2_{\bY_h} 
+ \|\be^h_{\ubu}\|^2_{\bZ}
\Big)\, ds \nonumber \\[0.5ex] 
&\ds
\leq C\,\exp(T)\,\Bigg(\int^t_0 \Big( 
\|\be^I_{\ubsi}\|^2_{\bX}
+ \|\be^I_{\ubvarphi}\|^2_{\bY_h}
+ \|\be^I_{\ubu}\|^2_{\bZ}
+ |\be^I_{\bvarphi} - \be^I_{\btheta}|^2_{\BJS}
+ \|\partial_t\,A^{1/2}\,(\be^I_{\bsi_p} + \alpha_p\,\be^I_{p_p}\,\bI)\|^2_{\bbL^2(\Omega_p)} 
\Big)\,ds \nonumber \\[0.5ex]
&\ds\quad +\,\int^t_0 \Big( 
\|\partial_t\,A^{1/2}\,(\be^h_{\bsi_p} + \alpha_p\,\be^h_{p_p}\bI)\|^2_{\bbL^2(\Omega_p)} 
+ s_0\|\partial_t\,\be^h_{p_p}\|^2_{\W_p}
\Big)\, ds  \\[0.5ex]
&\ds\quad
+\,\,\|A^{1/2}(\be^h_{\bsi_p} + \alpha_p\,\be^h_{p_p}\bI)(0)\|^2_{\bbL^2(\Omega_p)}
+ s_0\|\be^h_{p_p}(0)\|^2_{\W_p} \Bigg). \nonumber
\end{align}

Now, in order to bound $\ds\int^t_0 \Big( \|\partial_t\,A^{1/2}\,(\be^h_{\bsi_p} + \alpha_p\,\be^h_{p_p}\bI) \|^2_{\bbL^2(\Omega_p)} + s_0\|\partial_t\,\be^h_{p_p}\|^2_{\W_p} \Big)\,ds$ on the right-hand side of \eqref{eq:error-analysis2}, we test \eqref{eq:error-equation} with $\ubtau_{h} = (\partial_t \be^h_{\bsi_f}, \be^h_{\bu_p}, \partial_t \be^h_{\bsi_p}, \partial_t \be^h_{p_p})$, $\ubpsi_h = (\be^h_{\bvarphi}, \be^h_{\btheta}, \partial_t \be^h_{\lambda})$, and $\ubv_{h} = (\be^h_{\bu_f}, \be^h_{\bu_s}, \be^h_{\bgamma_f}, \be^h_{\bgamma_p})$, differentiate in 
time the rows in \eqref{eq:error-equation} associated to $\bv_{ph}, \bpsi_h, \bphi_h, \bv_{fh}, \bv_{sh}, \bchi_{fh}, \bchi_{ph}$, and employ the projections properties \eqref{eq:projection1}--\eqref{eq:projection3} to eliminate
some of the terms (cf. \eqref{zero-terms}), obtaining
\begin{align}\label{eq:error-analysis3}
& \ds  
\frac{1}{2} \partial_t \Big( \frac{1}{2 \mu} \|(\be^h_{\bsi_f})^\rd\|^2_{\bbL^2(\Omega_f)} 
+ a_p(\be^h_{\bu_p}, \be^h_{\bu_p}) 
+ c_{\BJS}(\be^h_{\bvarphi}, \be^h_{\btheta};\be^h_{\bvarphi}, \be^h_{\btheta}) \Big) \nonumber \\[0.5ex]
& \ds \quad
+ \|\partial_t A^{1/2} (\be^h_{\bsi_p} {+} \alpha_p \be^h_{p_p}\bI)\|^2_{\bbL^2(\Omega_p)} 
+ s_0 \|\partial_t \be^h_{p_p}\|^2_{\W_p} \nonumber \\[0.5ex] 
& \ds
= -\,a_f(\be^I_{\bsi_f},\partial_t\,\be^h_{\bsi_f})
- a_p(\partial_t\,\be^I_{\bu_p}, \be^h_{\bu_p})
- a_e(\partial_t\,\be^I_{\bsi_p}, \partial_t\,\be^I_{p_p}; \partial_t\,\be^h_{\bsi_p},  \partial_t\,\be^h_{p_p})
- c_{\BJS}(\partial_t \, \be^I_{\bvarphi}, \partial_t \, \be^I_{\btheta}; \be^h_{\bvarphi},  \be^h_{\btheta})
 \nonumber \\[0.5ex] 
& \ds\quad 
+c_{\Gamma}( \be^h_{\bvarphi},  \be^h_{\btheta}; \partial_t \, \be^I_{\lambda})
-c_{\Gamma}( \be^I_{\bvarphi}, \be^I_{\btheta}; \partial_t \, \be^h_{\lambda})
-\,b_{\bn_f}(\partial_t \,\be^h_{\bsi_f},\be^I_{\bvarphi})
-b_{\bn_p}(\partial_t \,\be^h_{\bsi_p},\be^I_{\btheta})
-b_{\Gamma}(\be^h_{\bu_p},\partial_t \,\be^I_{\lambda}) \\[0.5ex] 
& \ds\quad 
+\,b_{\bn_f}(\partial_t \,\be^I_{\bsi_f},\be^h_{\bvarphi})
+b_{\bn_p}(\partial_t \,\be^I_{\bsi_p},\be^h_{\btheta})
+b_{\Gamma}(\be^I_{\bu_p},\partial_t \,\be^h_{\lambda})
-b_\skf(\partial_t\,\be^h_{\bsi_f},\be^I_{\bgamma_f}) 
- b_\skp(\partial_t\,\be^h_{\bsi_p},\be^I_{\bgamma_p}) \nonumber \\[0.5ex] 
& \ds\quad 
+\, b_\skf(\partial_t\,\be^I_{\bsi_f},\be^h_{\bgamma_f}) 
+ b_\skp(\partial_t\,\be^I_{\bsi_p},\be^h_{\bgamma_p})\,. \nonumber
\end{align}
Then, integrating \eqref{eq:error-analysis3} from $0$ to $t\in (0,T]$,
using the identities
\begin{equation}\label{eq:integration-by-parts-time-identity}
\begin{array}{l}
\ds \int^t_0 a_f(\be^I_{\bsi_f}, \partial_t\,\be^h_{\bsi_f})\,ds 
\,=\, a_f(\be^I_{\bsi_f}, \be^h_{\bsi_f})\Big|_0^t 
- \int^t_0 a_f(\partial_t\,\be^I_{\bsi_f}, \be^h_{\bsi_f})\,ds\,, \\ [2.5ex]
\ds \int^t_0 b_{\bn_\star}( \partial_t\,\be^h_{\bsi_\star}, \be^I_{\circ})\,ds 
\,=\, b_{\bn_\star}( \be^h_{\bsi_\star},\be^I_{\circ})\Big|_0^t 
- \int^t_0 b_{\bn_\star}( \be^h_{\bsi_\star}, \partial_t\,\be^I_{\circ})\,ds\,,\,\,\star\in \{f,p\},\,\circ\in \{ \bvarphi, \btheta \}\,, \\ [2.5ex]
\ds \int^t_0 b_{\sk,\star}(\partial_t\,\be^h_{\bsi_\star},\be^I_{\bgamma_\star})\,ds 
\,=\, b_{\sk,\star}(\be^h_{\bsi_\star},\be^I_{\bgamma_\star})\Big|_0^t 
- \int^t_0 b_{\sk,\star}(\be^h_{\bsi_\star},\partial_t\,\be^I_{\bgamma_\star})\,ds\,, \\ [2.5ex]
\ds \int^t_0 \pil \be^I_{\diamond}\cdot\bn_f,\partial_t\,\be^h_{\lambda} \pir_{\Gamma_{fp}} ds 
\,=\, \pil \be^I_{\diamond}\cdot\bn_f, \be^h_{\lambda}\pir_{\Gamma_{fp}} \Big|_0^t 
- \int^t_0 \pil \partial_t\,\be^I_{\diamond}\cdot\bn_f, \be^h_{\lambda}\pir_{\Gamma_{fp}} ds\,,\,\,\diamond\in \{ \bvarphi, \btheta, \bu_p \},
\end{array}
\end{equation}
and applying the ellipticity and continuity bounds of the bilinear
forms involved (cf. Lemma~\ref{lem:N-M-properties}), the Cauchy-Schwarz
and Young's inequalities, and the
fact that $\bdiv(\be^h_{\bsi_\star}) = \0$ in $\Omega_\star$ with
$\star\in \{f,p\}$ 
  (cf. \eqref{eq:error-analysis7}), we obtain
\begin{align}\label{eq:error-analysis4}
&\ds \|\be^h_{\bsi_f}(t)\|^2_{\bbX_f}
+ \|\be^h_{\bu_p}(t)\|^2_{\bL^2(\Omega_p)}
+ \|\bdiv(\be^h_{\bsi_p}(t))\|^2_{\bL^2(\Omega_p)}
+ |(\be^h_\bvarphi - \be^h_\btheta)(t)|^2_{\BJS} \nonumber \\[1ex]
&\ds\quad +\,\int^t_0 \Big( 
\|\partial_t\,A^{1/2}\,(\be^h_{\bsi_p} + \alpha_p\,\be^h_{p_p}\bI)\|^2_{\bbL^2(\Omega_p)} 
+ s_0\|\partial_t\,\be^h_{p_p}\|^2_{\W_p}\Big)\,ds \nonumber \\[1ex]
&\ds\leq\, C\,\Bigg(
\|\be^I_{\bsi_f}(t)\|^2_{\bbL^2(\Omega_f)}
+ \|\be^I_{\bu_p}(t)\|^2_{\bV_p}
+ \|\be^I_{\bsi_p}(t)\|^2_{\bbL^2(\Omega_p)}
+ \|\be^I_{\bvarphi}(t)\|^2_{\bLambda_{fh}}
+ \|\be^I_{\btheta}(t)\|^2_{\bLambda_{sh}}
+ \|\be^I_{\bgamma_f}(t)\|^2_{\bbQ_f}
\nonumber \\[1ex]
&\ds \quad +\,
\|\be^I_{\bgamma_p}(t)\|^2_{\bbQ_p}
+ \int^t_0 \Big(
\|\partial_t\,\be^I_{\bsi_f}\|^2_{\bbX_f}
+ \|\partial_t\,\be^I_{\bu_p}\|^2_{\bV_p}
+|\partial_t\,(\be^I_{\bvarphi} - \be^I_{\btheta})|^2_{\BJS}
+ \| \be^I_{\btheta}\|^2_{\bLambda_{sh}}
+ \|\partial_t\,\be^I_{\ubvarphi}\|^2_{\bY_h}
\nonumber \\[1ex]
& \ds \quad 
+ \|\partial_t\,\be^I_{\bgamma_f}\|^2_{\bbQ_f}
+ \|\partial_t\,\be^I_{\bgamma_p}\|^2_{\bbQ_p}
+ \|\partial_t\,A^{1/2}\,(\be^I_{\bsi_p} + \alpha_p\,\be^I_{p_p}\bI)\|^2_{\bbL^2(\Omega_p)}
+ \|\partial_t\,\be^I_{\bsi_p}\|^2_{\bbX_p} \Big)\,ds \nonumber \\[1ex]
& \ds\quad+\,
\|\be^I_{\bsi_f}(0)\|^2_{\bbL^2(\Omega_f)}
+ \|\be^I_{\bu_p}(0)\|^2_{\bV_p}
+ \|\be^I_{\bvarphi}(0)\|^2_{\bLambda_{fh}}
+ \|\be^I_{\btheta}(0)\|^2_{\bLambda_{sh}}
+ \|\be^I_{\bgamma_f}(0)\|^2_{\bbQ_f}\Bigg) \nonumber \\
&\ds +\,\delta_3\,\Bigg(
\|\be^h_{\bsi_f}(t)\|^2_{\bbX_f}
+ \|\be^h_{\bsi_p}(t)\|^2_{\bbL^2(\Omega_p)}
+ \|\be^h_{\lambda}(t)\|^2_{\Lambda_{ph}}
+ \int^t_0 \Big( 
\|\be^h_{\bsi_f}\|^2_{\bbX_f}
+ \|\be^h_{\bu_p}\|^2_{\bV_p}
+ |\be^h_{\bvarphi} - \be^h_{\btheta}|^2_{\BJS} \Big)\,ds \nonumber \\
&\ds \quad +\,
\int^t_0 \Big( \|\be^h_{\ubvarphi}\|^2_{\bY_h}
+ \|\be^h_{\ubu}\|^2_{\bZ} \Big)\,ds \Bigg)
+ \frac{1}{2}\,\int^t_0 
\|\partial_t\,A^{1/2}\,(\be^h_{\bsi_p} + \alpha_p\,\be^h_{p_p}\bI)\|^2_{\bbL^2(\Omega_p)}\, ds \nonumber \\[1ex]
&\ds\quad+\,
C\,\Big(\|\be^h_{\bsi_f}(0)\|^2_{\bbX_f}
+ \|\be^h_{\bu_p}(0)\|^2_{\bL^2(\Omega_p)}
+ \|\be^h_{\bsi_p}(0)\|^2_{\bbX_p}
+ |(\be^h_{\bvarphi} - \be^h_{\btheta})(0)|^2_{\BJS}
+ \|\be^h_{\lambda}(0)\|^2_{\Lambda_{ph}} \Big)\,. 
\end{align}
We note that $\|\be^h_{\bsi_p}(t)\|^2_{\bbL^2(\Omega_p)} + \|\be^h_{\lambda}(t)\|^2_{\Lambda_{ph}}$ can be bounded by using \eqref{eq:aux-taup-qp-inequality} and \eqref{eq:error-inf-sup3}, whereas all the other terms with $\delta_3$
can be bounded by the left hand side of \eqref{eq:error-analysis2}. 
Thus, combining \eqref{eq:error-analysis2} with \eqref{eq:error-inf-sup3} and \eqref{eq:error-analysis4}, using algebraic manipulations, and choosing $\delta_3$ small enough,  we get 
\begin{align}\label{eq:error-analysis6}
&\ds \|\be^h_{\bsi_f}(t)\|^2_{\bbX_f}
+ \|\be^h_{\bu_p}(t)\|^2_{\bL^2(\Omega_p)}
+ |(\be^h_\bvarphi - \be^h_\btheta)(t)|^2_{\BJS}
+ \|\be^h_{\lambda}(t)\|^2_{\Lambda_{ph}}
+ \|A^{1/2}(\be^h_{\bsi_p}+\alpha_p \, \be^h_{p_p} \, \bI)(t)\|^2_{\bbL^2(\Omega_p)} 
\nonumber \\[1ex]
&\ds\quad
+ \|\bdiv(\be^h_{\bsi_p}(t))\|^2_{\bL^2(\Omega_p)}
+ \|\be^h_{p_p}(t)\|^2_{\W_p}
+ \int^t_0 \Big( 
\|\be^h_{\bsi_f}\|^2_{\bbX_f}
+ \|\be^h_{\bu_p}\|^2_{\bV_p}
+ |\be^h_{\bvarphi} - \be^h_{\btheta}|^2_{\BJS} 
+ \|\be^h_{\ubvarphi}\|^2_{\bY_h} 
\nonumber \\[1ex]
&\ds\quad
+ \|\be^h_{\ubu}\|^2_{\bZ}
+ \|\bdiv(\be^h_{\bsi_p})\|^2_{\bL^2(\Omega_p)}  
+ \|\be^h_{p_p}\|^2_{\W_p} 
+ \|\partial_t\,A^{1/2}\,(\be^h_{\bsi_p}+ \alpha_p\,\be^h_{p_p}\,\bI)\|^2_{\bbL^2(\Omega_p)} 
+ s_0\|\partial_t\,\be^h_{p_p}\|^2_{\W_p} 
\Big)\, ds \nonumber \\[1ex]
&\ds\leq\, 
C\,\exp(T)\Bigg( 
\|\be^I_{\bsi_f}(t)\|^2_{\bbL^2(\Omega_f)}
+ \|\be^I_{\bu_p}(t)\|^2_{\bV_p}
+ \|\be^I_{\bsi_p}(t)\|^2_{\bbL^2(\Omega_p)}
+ \|\be^I_{\bvarphi}(t)\|^2_{\bLambda_{fh}}
+ \|\be^I_{\btheta}(t)\|^2_{\bLambda_{sh}} \nonumber \\[1ex]
&\ds\quad
+ \|\be^I_{\bgamma_f}(t)\|^2_{\bbQ_f}
+ \|\be^I_{\bgamma_p}(t)\|^2_{\bbQ_p}
+ \int^t_0 \Big( 
\|\be^I_{\ubsi}\|^2_{\bX} 
+ \|\be^I_{\ubvarphi}\|^2_{\bY_h}
+ \|\be^I_{\ubu}\|^2_{\bZ}
+ |\be^I_{\bvarphi} - \be^I_{\btheta}|^2_{\BJS}
+ \|\partial_t\,\be^I_{\ubsi}\|^2_{\bX} 
\Big)\,ds \nonumber \\[1ex]
& \ds\quad
+ \int^t_0 \Big( \|\partial_t\,\be^I_{\ubvarphi}\|^2_{\bY_h}
+ |\partial_t \,(\be^I_{\bvarphi} - \be^I_{\btheta})|^2_{\BJS}
+ \|\partial_t\,\be^I_{\bgamma_f}\|^2_{\bbQ_f}
+ \|\partial_t\,\be^I_{\bgamma_p}\|^2_{\bbQ_p} \Big)\,ds
+ \|\be^I_{\bsi_f}(0)\|^2_{\bbL^2(\Omega_f)} \nonumber \\[1ex]
&\ds\quad
+ \|\be^I_{\bu_p}(0)\|^2_{\bV_p}
+ \|\be^I_{\bvarphi}(0)\|^2_{\bLambda_{fh}}
+ \|\be^I_{\btheta}(0)\|^2_{\bLambda_{sh}}
+ \|\be^I_{\bgamma_f}(0)\|^2_{\bbQ_f}
+ \|\be^h_{\bsi_f}(0)\|^2_{\bbX_f}
+ \|\be^h_{\bu_p}(0)\|^2_{\bL^2(\Omega_p)} \nonumber \\[1ex]
&\ds\quad
+ \|\be^h_{\bsi_p}(0)\|^2_{\bbX_p}
+ (1 + s_0)\|\be^h_{p_p}(0)\|^2_{\W_p}
+ |(\be^h_{\bvarphi} - \be^h_{\btheta})(0)|^2_{\BJS}
+ \|\be^h_{\lambda}(0)\|^2_{\Lambda_{ph}} \Bigg).
\end{align}

Finally, we establish a bound on the initial data terms above. 
In fact, proceeding as in \eqref{init-data-bound}, recalling from Corollary~\ref{cor:initial-data} and Theorem~\ref{thm:well-posedness-semidiscrete} that
$(\ubsi(0), \ubvarphi(0))
= (\ubsi_{0}, \ubvarphi_0)$ and $(\ubsi_h(0),\ubvarphi_h(0)) = (\ubsi_{h,0},\ubvarphi_{h,0})$, using similar arguments to \eqref{eq:error-analysis2} in combination with the error system derived from \eqref{eq:auxiliary-initial-condition-semidiscrete-1}--\eqref{eq:auxiliary-initial-condition-semidiscrete-2}, we deduce
\begin{equation}\label{eq:error-analysis-0}
\begin{array}{l}
\ds \|\be^h_{\bsi_f}(0)\|^2_{\bbX_f} 
+ \|\be^h_{\bu_p}(0)\|^2_{\bV_p} 
+ \|A^{1/2}\,(\be^h_{\bsi_p}(0))\|^2_{\bbL^2(\Omega_p)}
+ \|\bdiv(\be^h_{\bsi_p}(0))\|^2_{\bL^2(\Omega_p)}
+ \|\be^h_{p_p}(0)\|^2_{\W_p} \\ [2ex]
\ds\quad 
+\,\,|(\be^h_{\bvarphi} - \be^h_{\btheta})(0)|^2_{\BJS}
+ \|\be^h_{\lambda}(0)\|^2_{\Lambda_{ph}}
\,\leq\,
C\,\Big(
\|\be^I_{\ubsi_0}\|^2_{\bX} 
+ \|\be^I_{\wt{\ubvarphi}_0}\|^2_{\bY_h}
+ \|\be^I_{\wt{\ubu}_0}\|^2_{\bZ} \Big)\,,
\end{array}
\end{equation}
where $\ubsi_0 = (\bsi_{f,0},\bu_{p,0},\bsi_{p,0}, p_{p,0})$, $\wt{\ubvarphi}_0 = (\bvarphi_0,\bomega_0,\lambda_0)$ and $\wt{\ubu}_0 = (\bu_{f,0},\bbeta_{p,0},\bgamma_{f,0},\brho_{p,0})$, and $\be^I_{\bsi_0}, \be^I_{\wt{\ubvarphi}_0}, \be^I_{\wt{\ubu}_0}$ denote their corresponding approximation errors.
Thus, using the error decomposition \eqref{eq:error-decomposition} 
in combination with \eqref{eq:error-analysis6}--\eqref{eq:error-analysis-0}, 
the triangle inequality, \eqref{eq:aux-taup-qp-inequality} and
the approximation 
properties \eqref{eq:approx-property1} and \eqref{eq:approx-property2}, 
we obtain \eqref{eq:error-analysis-result} with a positive constant
$C$ depending 
on parameters $\mu, \lambda_p, \mu_p, \alpha_p, k_{\min}, k_{\max}, \alpha_{\BJS}$, 
and the extra regularity assumptions for $\ubsi, \ubvarphi$, and $\ubu$
whose expressions are obtained from the right-hands side of
\eqref{eq:approx-property1} and \eqref{eq:approx-property2}.
This completes the proof in the conforming case \textbf{(S1)}.

The proof in the non-conforming case \textbf{(S2)} follows by using
similar arguments. We exploit the projection property
\eqref{eq:projection2} to conclude that some terms in
\eqref{eq:error-equation3} are zero, namely
$b_{\bn_f}(\be^h_{\bsi_f},\be^I_{\bvarphi})$,
$b_{\bn_p}(\be^h_{\bsi_p},\be^I_{\btheta})$, and
$b_{\Gamma}(\be^h_{\bu_p},\be^I_{\lambda})$, as well as terms
appearing in the operator $\cC$ (cf. \eqref{eq:bilinear-forms-2}):
$\pil \be^h_{\bvarphi}\cdot\bn_f,\be^I_\lambda\pir_{\Gamma_{fp}}$,
$\pil \be^I_\bvarphi\cdot\bn_f,\be^h_{\lambda} \pir_{\Gamma_{fp}}$,
$\pil \be^h_{\btheta}\cdot\bn_p,\be^I_\lambda\pir_{\Gamma_{fp}}$, and
$\pil \be^I_\btheta\cdot\bn_p,\be^h_{\lambda} \pir_{\Gamma_{fp}}$.  In
addition, in the non-conforming version of \eqref{eq:error-inf-sup1}
the terms $\|\be^I_{\lambda} \|_{\Lambda_{ph}}$,
$\|\be^I_{\bvarphi}\|_{\bLambda_{fh}}$, and
$\|\be^I_{\btheta}\|_{\bLambda_{sh}}$ do not appear, since the
bilinear forms $b_{\Gamma}(\bv_{ph}, \be^I_{\lambda})$,
$b_{\bn_f}(\btau_{fh},\be^I_{\bvarphi})$, and
$b_{\bn_p}(\btau_{ph},\be^I_{\btheta})$ are zero by a direct
application of the projection property \eqref{eq:projection2}.
\end{proof}

\section{A multipoint stress-flux mixed finite element method}\label{sec:multipoint-fem}

In this section, inspired by previous works on the multipoint flux
mixed finite element method for Darcy flow
\cite{iwy2010,wxy2012,wy2006,Brezzi.F;Fortin.M;Marini.L2006}
and the multipoint stress mixed finite element method for elasticity
\cite{msmfe-simpl,msmfe-quads,msfmfe-Biot}, we present a vertex
quadrature rule that allows for local elimination of the stresses,
rotations, and Darcy fluxes, leading to a positive-definite
cell-centered pressure-velocities-traces system. We emphasize that,
to the best of our knowledge, this is the first time such method
is developed for the Stokes equations. To that end, the
finite element spaces to be considered for both $(\bbX_{fh}, \bV_{fh},
\bbQ_{fh})$ and $(\bbX_{ph}, \bV_{sh}, \bbQ_{ph})$ are the triple
$\bbBDM_1 - \bP_0 - \bbP_1$, which have been shown to be stable
for mixed elasticity with weak stress symmetry in
\cite{brezzi2008mixed,BBF-reduced,FarFor}, whereas
$(\bV_{ph}, \W_{ph})$ is chosen
to be $\bBDM_1 - \rP_0$ \cite{bdm1985}, and the Lagrange multiplier spaces
$(\bLambda_{fh}, \bLambda_{sh}, \Lambda_{ph})$ are either
$\bP_1-\bP_1-\rP_1$ or $\bP^{\dc}_1-\bP^{\dc}_1-\rP^{\dc}_1$
satisfying \textbf{(S1)} or \textbf{(S2)}
(cf. \eqref{eq:conforming-spaces-Yh},
\eqref{eq:non-conforming-spaces-Yh}), respectively, where
$\rP^{\dc}_1$ denotes the piecewise linear discontinuous finite
element space and $\bP^{\dc}_1$ is its corresponding vector version.

\subsection{A quadrature rule setting}\label{seq:quadrature}
Let $S_\star$ denote the space of elementwise continuous functions on
$\cT_h^\star$. For any pair of tensor or vector valued functions
$\varphi$ and $\psi$ with elements in $S_\star$, we define the vertex
quadrature rule as in \cite{wy2006} (see also
\cite{msmfe-simpl,msfmfe-Biot}):
\begin{equation}\label{quad-rule}
(\varphi, \psi)_{Q,\Omega_\star}
\,:=\, \sum_{E\in \cT_h^{\star}}(\varphi, \psi)_{Q,E}
\,=\, \sum_{E\in \cT_h^{\star}} \frac{|E|}{s}
\sum^{s}_{i=1}\,\varphi(\br_i)\cdot\psi(\br_i),
\end{equation}
where $\star\in \{f, p\}$, $s=3$ on triangles and $s=4$ on tetrahedra,
$\br_i$, $i=1,\ldots,s$, are the vertices of the element $E$, and
$\cdot$ denotes the inner product for both vectors and tensors.

We will apply the quadrature rule for the bilinear forms $a_f$, $a_p$,
$a_e$ and $b_{\sk,\star}$, which will be denoted by $a^h_f$, $a^h_p$,
$a^h_e$ and $b^h_{\sk,\star}$, respectively. These bilinear forms
involve the stress spaces $\bbX_{fh}$ and $\bbX_{ph}$, the vorticity
space $\bbQ_{fh}$ and rotation space $\bbQ_{ph}$, and the Darcy
velocity space $\bV_{ph}$. The $\bBDM_1$ spaces have for degrees of
freedom $s-1$ normal components on each element edge (face), which can
be associated with the vertices of the edge (face).  At any element
vertex $\br_i$, the value of a tensor or vector function is uniquely
determined by its normal components at the associated two edges or
three faces. Also, the vorticity space $\bbQ_{fh}$ and the rotation
space $\bbQ_{ph}$ are vertex-based. Therefore the application of the
vertex quadrature rule \eqref{quad-rule} for the bilinear forms
involving the above spaces results in coupling only the degrees of
freedom associated with a mesh vertex, which allows for local
elimination of these variables. Next, we state a preliminary lemma to
be used later on, which has been proved in
\cite[Lemma~3.1]{msfmfe-Biot} and \cite[Lemma~2.2]{msmfe-simpl}.
\begin{lem}\label{lem:quadrature-error-analysis-1}
There exist positive constants $C_0$ and $C_1$ independent of $h$,
such that for any linear uniformly bounded and positive-definite
operator $L$, there hold
\begin{equation*}
(L(\varphi), \varphi)_{Q,\Omega_\star} \,\geq\, C_0\,\|\varphi\|_{\Omega_\star}^2,\quad 
  (L(\varphi), \psi)_{Q,\Omega_\star} \,\leq\,
  C_1\,\|\varphi\|_{\Omega_\star} \|\psi\|_{\Omega_\star},
  \quad \forall \, \varphi,\psi \in S_\star, \ \
  \star\in \{f,p\}.
\end{equation*}
Consequently, the bilinear form
$(L(\varphi),\varphi)_{Q,\Omega_\star}$ is an inner product in $\L^2(\Omega_\star)$
and $(L(\varphi),\varphi)^{1/2}_{Q,\Omega_\star}$ is a norm equivalent
to $\| \varphi\|_{\Omega_\star}$.
\end{lem}

The semidiscrete coupled multipoint stress-flux mixed finite element
method for \eqref{eq:evolution-problem-in-operator-form}
reads: Find $(\ubsi_h,\ubvarphi_h,\ubu_h):[0,T]\to \bX_h\times
\bY_h\times \bZ_h$ such that for all $(\ubtau_h,\ubpsi_h,\ubv_h)\in
\bX_h\times \bY_h\times \bZ_h$, and for a.e. $t\in (0,T)$,
\begin{equation}\label{eq:Stokes-Biot-formulation-multipoint}
\begin{array}{lll}
\ds \frac{\partial}{\partial t}\,\cE_h(\ubsi_h)(\ubtau_h) + \cA_h(\ubsi_h)(\ubtau_h) + \cB_1(\ubtau_h)(\ubvarphi_h) + \cB_h(\ubtau_h)(\ubu_h) & = & \bF(\ubtau_h) , \\ [1.5ex]
\ds -\,\cB_1(\ubsi_h)(\ubpsi_h) + \cC(\ubvarphi_h)(\ubpsi_h) & = & 0 , \\ [1.5ex]
\ds -\,\cB_h(\ubsi_h)(\ubv_h) & = & \bG(\ubv_h) ,
\end{array}
\end{equation}
where 
\begin{equation*}
\begin{array}{l}
\ds \cA_h(\ubsi_h)(\ubtau_h) \,:=\, a^h_f(\bsi_{fh},\btau_{fh}) + a^h_p(\bu_{ph},\bv_{ph}) + b_p(\bv_{ph},p_{ph}) - b_p(\bu_{ph},w_{ph}),\\[1.5ex]
\ds \cE_h(\ubsi_h)(\ubtau_h) \,:=\, a^h_e(\bsi_{ph},p_{ph};\btau_{ph},w_{ph}) + (s_0\,p_{ph},w_{ph})_{\Omega_p}, \\[1.5ex]
\ds \cB_h(\ubtau_h)(\ubv_h) 
\,:=\, b_f(\btau_{fh},\bv_{fh}) + b_s(\btau_{ph},\bv_{sh})
+ b^h_{\skf}(\btau_{fh},\bchi_{fh}) + b^h_{\skp}(\btau_{ph},\bchi_{ph}).
\end{array}
\end{equation*}

We next discuss the discrete inf-sup conditions. We recall the space
$\wt{\bX}_h$ defined in \eqref{eq:subspace-bbXh-tilde}. We also define
the discrete kernel of the operator $\cB_h$ as
\begin{equation}\label{eq:discrete-kernel-Bh}
  \wh\bV_h := \Big\{ \ubtau_h\in \bX_h :\quad \cB_h(\ubtau_h)(\ubv_h) = 0
  \quad \forall\,\bv_h\in \bZ_h \Big\} 
= \wh{\bbX}_{fh}\times \bV_{ph}\times \wh{\bbX}_{ph}\times \W_{ph},
\end{equation}
where
\begin{equation*}
\wh{\bbX}_{\star h} := \Big\{ \btau_{\star h}\in \bbX_{\star h} :\,
(\btau_{\star h}, \bxi_{\star h})_{Q,\Omega_\star} = 0 \ \
\forall \, \bxi_{\star h} \in \bbQ_{\star h}
\qan \bdiv(\btau_{\star h}) = \0 \qin \Omega_\star \Big\},\quad \star\in \{f,p\}, 
\end{equation*}
emphasizing the difference from the discrete kernel of $\cB$ defined
in \eqref{eq:discrete-kernel-B}.

\begin{lem}\label{lem:quadrature-error-analysis-2}
There exist positive constants $\wh{\beta}$ and $\wh{\beta}_1$, such that
\begin{equation}\label{eq:inf-sup-Bh} 
\sup_{\0\neq \ubtau_h\in \wt{\bX}_h} \, \frac{\cB_h(\ubtau_h)(\ubv_h)}{\|\ubtau_h\|_{\bX}} 
\,\geq\, \wh{\beta}\,\|\ubv_h\|_{\bZ} \quad \forall\,\ubv_h\in \bZ_h,
\end{equation}
\begin{equation}\label{eq:inf-sup-B1h}
\sup_{\0\neq \ubtau_h \in \wh\bV_h} \frac{\cB_1(\ubtau_h)(\ubpsi_h)}{\|\ubtau_h \|_{\bX}} 
\,\geq\, \wh{\beta}_{1}\,\|\ubpsi_h\|_{\bY_h} \quad \forall\,\ubpsi_h\in \bY_h.
\end{equation}
\end{lem}

\begin{proof}
The proof of \eqref{eq:inf-sup-Bh} follows from a slight adaptation of the argument
in \cite[Theorem~4.2]{msmfe-simpl}. The proof of \eqref{eq:inf-sup-B1h} is similar
to the proof of \eqref{eq:inf-sup-b1-semidiscrete}. The main difference is
replacing the interpolant satisfying \eqref{pi-tilde} by an interpolant
$\hat\Pi_h^f: \bbH^1(\Omega_f) \to \bbX_{fh}$ satisfying
\begin{equation*}
  \begin{array}{c}
  b_f(\hat\Pi_h^f\btau_f - \btau_f,\bv_{fh}) = 0 \quad
  \forall \, \bv_{fh} \in \bV_{fh}, \quad
  b_{\sk,f}^h(\hat\Pi_h^f\btau_f - \btau_f,\bchi_{fh}) = 0 \quad
  \forall \, \bchi_{fh} \in \bbQ_{fh}, \\ [1ex]
  \langle(\hat\Pi_h^f\btau_f - \btau_f)\bn_f,
  \btau_{fh}\bn_f\rangle_{\Gamma_{fp}\cup\Gamma_f^N} = 0
  \quad \forall \, \btau_{fh} \in \bbX_{fh},
  \end{array}
\end{equation*}
whose existence follows from the inf-sup condition for $\cB_h$
\eqref{eq:inf-sup-Bh}.
\end{proof}

We can establish the following well-posedness result.

\begin{thm}\label{thm:well-posdness-multipoint}
For each compatible initial data $(\ubsi_{h,0}, \ubvarphi_{h,0},
\ubu_{h,0})$ satisfying \eqref{eq:initial-condition-semidiscrete} and
\begin{equation*}
\f_f\in \W^{1,1}(0,T;\bV_f'),\quad \f_p\in \W^{1,1}(0,T;\bV_s'),\quad q_f\in \W^{1,1}(0,T;\bbX'_f),\quad q_p\in \W^{1,1}(0,T;\W'_p),
\end{equation*}
there exists a unique solution of
\eqref{eq:Stokes-Biot-formulation-multipoint},
$(\ubsi_h,\ubvarphi_h,\ubu_h):[0,T]\to \bX_h\times \bY_h\times \bZ_h$ such that
$(\bsi_{ph},p_{ph}) \in \W^{1,\infty}(0,T;\bbX_{ph}) \times \W^{1,\infty}(0,T;\W_{ph})$, and $(\ubsi_h(0), \ubvarphi_h(0), \bu_{fh}(0), \bgamma_{fh}(0)) = (\ubsi_{h,0}, \ubvarphi_{h,0}, \bu_{fh,0}, \bgamma_{fh,0})$. Moreover, assuming sufficient regularity of the data, a stability bound as in \eqref{eq:discrete-stability-bound} also holds.
\end{thm}
\begin{proof}
The theorem follows from similar arguments to the proof of 
Theorem~\ref{thm:well-posedness-semidiscrete}, in conjunction with
Lemmas~\ref{lem:quadrature-error-analysis-1} and
\ref{lem:quadrature-error-analysis-2}.
\end{proof}

\subsection{Error analysis}

Now, we obtain the error estimates and theoretical rates of
convergence for the multipoint stress-flux mixed scheme
\eqref{eq:Stokes-Biot-formulation-multipoint}.  To that end, for each
$\bsi_{fh}$, $\btau_{fh} \in \bbX_{fh}$, $\bu_{ph}$, $\bv_{ph} \in
\bV_{ph}$, $\bsi_{ph}$, $\btau_{ph} \in \bbX_{ph}$, $p_{ph}$, $w_{ph}
\in \W_{ph}$, $\bchi_{fh} \in \bbQ_{fh}$, and $\bchi_{ph} \in
\bbQ_{ph}$, we denote the quadrature errors by
\begin{equation}\label{eq:quadrature-local-errors}
\begin{array}{rl}
\delta_f(\bsi_{fh}, \btau_{fh}) 
\,= & a_f(\bsi_{fh}, \btau_{fh}) - a_f^{h}(\bsi_{fh}, \btau_{fh}), \\ [1.5ex]
\delta_p(\bu_{ph}, \bv_{ph}) 
\,= & a_p(\bu_{ph}, \bv_{ph}) - a_p^{h}(\bu_{ph}, \bv_{ph}), \\ [1.5ex]
\delta_e(\bsi_{ph}, p_{ph}; \btau_{ph}, w_{ph}) 
\,= & a_e(\bsi_{ph}, p_{ph}; \btau_{ph}, w_{ph}) - a_e^{h}(\bsi_{ph}, p_{ph}; \btau_{ph}, w_{ph}), \\ [1.5ex]
\delta_{\sk,\star}(\bchi_{\star h}, \btau_{\star h}) 
\,= & b_{\sk,\star}(\bchi_{\star h}, \btau_{\star h})
- b_{\sk,\star}^{h}(\bchi_{\star h}, \btau_{\star h}),\quad \star\in \{f,p\}.
\end{array}
\end{equation}

Next, for the operator $A$ (cf. \eqref{eq:operator-A-definition}) we
will say that $A\in \bbW^{1,\infty}_{\cT_h^p}$ if $A\in
\bbW^{1,\infty}(E)$ for all $E \in \cT_h^p$ and
$\|A\|_{\bbW^{1,\infty}(E)}$ is uniformly bounded independently of
$h$.  Similar notation holds for $\bK^{-1}$. In the next lemma we
establish bounds on the quadrature errors. The proof follows from a
slight adaptation of \cite[Lemma~5.2]{msmfe-simpl} to our context (see
also \cite{wy2006,msfmfe-Biot}).
\begin{lem}\label{lem:quadrature-error-analysis-3}
  If $\bK^{-1}\in \bbW^{1,\infty}_{\cT_h^p}$ and $A \in \bbW^{1,\infty}_{\cT_h^p}$,
  then there is a constant $C>0$ independent of $h$ such that
\begin{align*}
|\delta_f(\bsi_{fh}, \btau_{fh})| 
& \,\leq\,  
\ds C \sum_{E\in \cT_h^f} h\,\|\bsi_{fh}\|_{\bbH^1(E)}\, \|\btau_{fh}\|_{\bbL^2(E)}, \\
|\delta_p(\bu_{ph}, \bv_{ph})| 
& \,\leq\,
\ds C \sum_{E \in \cT_h^p} h\,\|\bK^{-1}\|_{\bbW^{1,\infty}(E)}\, \|\bu_{ph}\|_{\bH^1(E)}\, \|\bv_{ph}\|_{\bL^2(E)}, \\
|\delta_e(\bsi_{ph}, p_{ph}; \btau_{ph}, w_{ph})| 
& \,\leq\,
\ds C \sum_{E \in \cT_h^p} h\,\|A\|_{\bbW^{1,\infty}(E)} \|(\bsi_{ph},p_{ph})\|_{\bbH^1(E)\times \L^2(E)} \|(\btau_{ph},w_{ph})\|_{\bbL^2(E)\times \L^2(E)}, \\
|\delta_{\sk,\star}(\btau_{\star h},\bchi_{\star h})| 
& \,\leq\, 
\ds C \sum_{E \in \cT_h^\star} h\,\|\btau_{\star h}\|_{\bbL^2(E)}\,
\|\bchi_{\star h}\|_{\bbH^1(E)}, \quad \star\in \{f,p\}, \\
|\delta_{\sk,\star}(\btau_{\star h},\bchi_{\star h})| 
& \,\leq\,  
\ds C \sum_{E \in \cT_h^\star} h\,\|\btau_{\star h}\|_{\bbH^1(E)}\,
\|\bchi_{\star h}\|_{\bbL^2(E)}, \quad \star\in \{f,p\}, 
\end{align*}
for all $\bsi_{fh}, \btau_{fh}\in \bbX_{fh}$, $\bu_{ph}, \bv_{ph} \in \bV_{ph}$,
$\bsi_{ph}, \btau_{ph}\in \bbX_{ph}$, $p_{ph}, w_{ph}\in \W_{ph}$,
$\bchi_{fh}\in \bbQ_{fh}$, $\bchi_{ph}\in \bbQ_{ph}$.
\end{lem}

We are ready to establish the convergence of the multipoint stress-flux
mixed finite element method.

\begin{thm}\label{thm:error-estimate-multipoint}
For the solutions of the continuous and semidiscrete problems
\eqref{eq:evolution-problem-in-operator-form} and
\eqref{eq:Stokes-Biot-formulation-multipoint}, respectively,
assuming sufficient regularity of the true solution
according to \eqref{eq:approx-property1} and \eqref{eq:approx-property2},
there exists a positive constant $C$ independent of $h$ and $s_0$, such that
\begin{align}\label{eq:error-analysis-result-quadrature}
& \ds \|\be_{\bsi_f}\|_{\L^\infty(0,T;\bbX_{f})}
+ \|\be_{\bsi_f}\|_{\L^2(0,T;\bbX_{f})}
+ \|\be_{\bu_p}\|_{\L^\infty(0,T;\bL^2(\Omega_p))}
+ \|\be_{\bu_p}\|_{\L^2(0,T;\bV_{p})} 
+ |\be_{\bvarphi} - \be_{\btheta}|_{\L^\infty(0,T;\BJS)} \nonumber \\[0.5ex]
& \ds\quad
+ |\be_{\bvarphi} - \be_{\btheta}|_{\L^2(0,T;\BJS)}
+ \|\be_{\lambda}\|_{\L^\infty(0,T;\Lambda_{ph})}
+ \|\be_{\ubvarphi}\|_{\L^2(0,T;\bY_h)}
+ \|\be_{\ubu}\|_{\L^2(0,T;\bZ)}
+ \|A^{1/2}(\be_{\bsi_p})\|_{\L^\infty(0,T;\bbL^2(\Omega_{p}))} \nonumber \\[0.5ex]
& \ds\quad
+ \|\bdiv(\be_{\bsi_p})\|_{\L^\infty(0,T;\bL^2(\Omega_p))}
+ \|\be_{p_p}\|_{\L^\infty(0,T;\W_{p})} 
+ \|\bdiv(\be_{\bsi_p})\|_{\L^2(0,T;\bL^2(\Omega_p))}
+ \|\be_{p_p}\|_{\L^2(0,T;\W_{p})} \nonumber \\[0.5ex]
& \ds\quad 
+ \|\partial_t\,A^{1/2}(\be_{\bsi_p} + \alpha_p \be_{p_p}\bI)\|_{\L^2(0,T;\bbL^2(\Omega_p))}
+ \sqrt{s_0}\|\partial_t\,\be_{p_p}\|_{\L^2(0,T;\W_{p})} \nonumber \\[0.5ex]
& \ds \leq\,\, 
C\,\Big( h + h^{1+r} \Big)\,,
\end{align}
where $r$ is defined in \eqref{eq:approx-property1}.
\end{thm}
\begin{proof}
To obtain the error equations, we subtract the
multipoint stress-flux mixed finite element formulation
\eqref{eq:Stokes-Biot-formulation-multipoint} from
the continuous one \eqref{eq:evolution-problem-in-operator-form}.
Using the error decomposition \eqref{eq:error-decomposition} and applying some
algebraic manipulations, we obtain the error system:
\begin{equation}\label{eq:error-equation-multipoint}
\begin{array}{l}
\ds \big(\partial_t\,\cE_h + \cA_h\big)(\be^h_{\ubsi})(\ubtau_h) + \cB_1(\ubtau_h)(\be^h_{\ubvarphi}) + \cB_h(\ubtau_h)(\be^h_{\ubu}) \\ [2ex] 
\ds\quad\,=\, - \big(\partial_t\,\cE + \cA\big)(\be^I_{\ubsi})(\ubtau_h) - \cB_1(\ubtau_h)(\be^I_{\ubvarphi}) - \cB(\ubtau_h)(\be^I_{\ubu}) - \bdelta_{fep}(I_h(\ubsi),P_h(\ubu))(\ubtau_h), \\ [2ex]
\ds -\,\cB_1(\be^h_{\ubsi})(\ubpsi_h) + \cC(\be^h_{\ubvarphi})(\ubpsi_h) \,=\, \cB_1(\be^I_{\ubsi})(\ubpsi_h) - \cC(\be^I_{\ubvarphi})(\ubpsi_h) \\ [2ex]
\ds -\,\cB_h(\be^h_{\ubsi})(\ubv_h) \,=\, \cB(\be^I_{\ubsi})(\ubv_h) + \bdelta_{fp}(I_h(\ubsi))(\ubv_h)\,,
\end{array}
\end{equation}
for all $(\ubtau_h,\ubpsi_h,\ubv_h)\in \bX_h\times \bY_h\times \bZ_h$, where
\begin{equation*}
\begin{array}{l}
\ds \bdelta_{fep}(I_h(\ubsi),P_h(\ubu))(\ubtau_h) 
\,:=\, -\,\delta_f(I^{\bbX_f}_h(\bsi_{f}),\btau_{fh}) - \delta_e(I^{\bbX_p}_h(\bsi_{p}),p_{p};\btau_{ph},w_{ph})\\ [2ex]
\ds\quad -\,\delta_p(I^{\bV_p}_h(\bu_{p}),\bv_{ph})  - \delta_{\sk,f}(\btau_{fh},P^{\bbQ_f}_h(\bgamma_{f})) - \delta_{\sk,p}(\btau_{ph},P^{\bbQ_p}_h(\bgamma_{p}))
\end{array}
\end{equation*}
and
\begin{equation*}
\bdelta_{fp}(I_h(\ubsi))(\ubv_h) \,:=\, 
\delta_{\sk,f}(I^{\bbX_f}_h(\bsi_{f}),\bchi_{fh}) 
+ \delta_{\sk,p}(I^{\bbX_p}_h(\bsi_{p}),\bchi_{ph})\,.
\end{equation*}
Notice that the error system \eqref{eq:error-equation-multipoint} is
similar to \eqref{eq:error-equation}, except for the additional
quadrature error terms. The rest of the proof follows from the
arguments in the proof of \eqref{eq:error-analysis-result}, using
Lemmas~\ref{lem:quadrature-error-analysis-1},
\ref{lem:quadrature-error-analysis-2} and
\ref{lem:quadrature-error-analysis-3}, and utilizing
the continuity bounds of the interpolation operators
$I^{\bX_{\star}}_h, I^{\bv_p}_h, P^{\bbQ_{\star}}_h$
\cite[Lemma~5.1]{msmfe-simpl}:
\begin{align*}
\ds \|I^{\bbX_\star}_h(\btau_{\star h})\|_{\bbH^1(E)} 
& \,\leq\, C\,\|\btau_{\star h}\|_{\bbH^1(E)} \quad \forall\,\btau_{\star h}\in \bbH^1(E)\,,\quad \star\in \{f,p\}\,, \\[0.5ex] 
\ds \|P^{\bbQ_\star}_h(\bchi_{\star h})\|_{\bbH^1(E)} 
& \,\leq\, C\,\|\bchi_{\star h}\|_{\bbH^1(E)} \quad \forall\,\bchi_{\star h}\in \bbH^1(E)\,, \\[0.5ex]
\ds \|I^{\bV_p}_h(\bv_{ph})\|_{\bH^1(E)} 
& \,\leq\, C\,\|\bv_{ph}\|_{\bH^1(E)} \quad \forall\,\bv_{ph} \in \bH^1(E)\,.
\end{align*}
We omit further details, and refer to
\cite{msmfe-simpl,wy2006,msfmfe-Biot} for more details on the error
analysis of the multipoint flux and multipoint stress mixed finite
element methods on simplicial grids.
\end{proof}

\subsection{Reduction to a cell-centered pressure-velocities-traces system}

In this section we focus on the fully discrete problem associated to
\eqref{eq:Stokes-Biot-formulation-multipoint}
(cf. \eqref{eq:evolution-problem-in-operator-form},
\eqref{eq:Stokes-Biot-formulation-semidiscrete}), and describe how to
obtain a reduced cell-centered system for the algebraic problem at
each time step. For the time discretization we employ the backward
Euler method.  Let $\Delta t$ be the time step, $T = M\,\Delta t$,
$t_m = m\,\Delta t$, $m=0,\dots,M$.  Let $d_t\,u^m :=
(\Delta t)^{-1}(u^m - u^{m-1})$ be the first order (backward)
discrete time derivative, where $u^m := u(t_m)$.  Then the fully
discrete model reads: given $(\ubsi^0_h, \ubvarphi^0_h, \ubu^0_h) =
(\ubsi_{h,0}, \ubvarphi_{h,0}, \ubu_{h,0})$ satisfying
\eqref{eq:initial-condition-semidiscrete}, find $(\ubsi^m_h,
\ubvarphi^m_h, \ubu^m_h)\in \bX_h\times \bY_h\times \bZ_h$,
$m=1,\dots,M$, such that for all $(\ubtau_h, \ubpsi_h, \ubv_h)\in
\bX_h\times \bY_h\times \bZ_h$,
\begin{equation}\label{eq:Stokes-Biot-fully-discrete-formulation}
\begin{array}{lll}
\ds d_t\,\cE_h(\ubsi^m_h)(\ubtau_h) + \cA_h(\ubsi^m_h)(\ubtau_h) + \cB_1(\ubtau_h)(\ubvarphi^m_h) + \cB_h(\ubtau_h)(\ubu^m_h) & = & \ds \bF(\ubtau_h)\,, \\ [2ex]
\ds -\,\cB_1(\ubsi^m_h)(\ubpsi_h) + \cC(\ubvarphi^m_h)(\ubpsi_h) & = & 0\,,  \\ [2ex]
\ds -\,\cB_h(\ubsi^m_h)(\ubv_h) & = & \bG(\ubv_h)\,.
\end{array}
\end{equation}

\begin{rem}
The well-posedness and error estimate associated to the fully discrete problem 
\eqref{eq:Stokes-Biot-fully-discrete-formulation} can be derived employing similar 
arguments to Theorems~\ref{thm:well-posdness-multipoint} and 
\ref{thm:error-estimate-multipoint} in combination with the theory developed in 
\cite[Sections~6 and 9]{akyz2018}.
In particular, we note that at each time step the well-posedness of the fully discrete 
problem \eqref{eq:Stokes-Biot-fully-discrete-formulation}, with $m=1,\dots,M$, follows from 
similar arguments to the proof of Lemma~\ref{lem:resolvent-system}.
\end{rem}

Notice that the first row in \eqref{eq:Stokes-Biot-fully-discrete-formulation} can be rewritten equivalently as
\begin{equation}\label{eq:SB-fdf-first-row}
\left( (\Delta t)^{-1}\cE_h + \cA_h \right)(\ubsi^m_h)(\ubtau_h) + \cB_1(\ubtau_h)(\ubvarphi^m_h) + \cB_h(\ubtau_h)(\ubu^m_h) \,=\, \bF(\ubtau_h) + (\Delta t)^{-1}\cE_h(\ubsi^{m-1}_h)(\ubtau_h)\,.
\end{equation}
Let us associate with the operators in 
\eqref{eq:Stokes-Biot-fully-discrete-formulation}--\eqref{eq:SB-fdf-first-row}
matrices denoted in the same way. We then have
\begin{equation*}
\left( (\Delta t)^{-1}\,\cE_h + \cA_h \right) = 
\left(\begin{array}{cccc}
A_{\bsi_f \bsi_f} & 0 & 0 & 0 \\ [1ex]
0 & A_{\bu_p \bu_p} & 0 & A^\rt_{\bu_p p_p} \\ [1ex]
0 & 0 & A_{\bsi_p \bsi_p} & A^\rt_{\bsi_p p_p} \\ [1ex]
0 & -A_{\bu_p p_p} & A_{\bsi_p p_p} & A_{p_p p_p}
\end{array}\right),\quad
\cB_h =
\left(\begin{array}{cccc}
A_{\bsi_f \bu_f} & 0 & 0 & 0 \\ [1ex]
0 & 0 & A_{\bsi_p \bu_s} & 0 \\ [1ex]
A_{\bsi_f \bgamma_f} & 0 & 0 & 0 \\ [1ex]
0 & 0 & A_{\bsi_p \bgamma_p} & 0 
\end{array}\right),
\end{equation*}
\begin{equation*}
\ds \cB_1 =
\left(\begin{array}{cccc}
A_{\bsi_f \bvarphi} & 0 & 0 & 0 \\ [1ex]
0 & 0 & A_{\bsi_p \btheta} & 0 \\ [1ex]
0 & A_{\bu_p \lambda} & 0 & 0 
\end{array}\right),\quad
\cC =
\left(\begin{array}{ccc}
A_{\bvarphi \bvarphi} & A^\rt_{\bvarphi \btheta} & A^\rt_{\bvarphi \lambda} \\ [1ex]
A_{\bvarphi \btheta} & A_{\btheta \btheta} & A^\rt_{\btheta \lambda} \\ [1ex]
-A_{\bvarphi \lambda} & -A_{\btheta \lambda} & 0 
\end{array}\right),
\end{equation*}
with
\begin{align*}
&\ds A_{\bsi_f \bsi_f} \sim a^h_f(\cdot, \cdot),\,\,
A_{\bu_p \bu_p} \sim a^h_p(\cdot, \cdot),\,\,
A_{\bsi_p \bsi_p} \sim (\Delta t)^{-1}\,a^h_e(\cdot,0;\cdot,0), \,\,
\ds A_{\bsi_p p_p} \sim (\Delta t)^{-1} a^h_e(\cdot,0;\0,\cdot), \\[1ex]
& A_{p_p p_p} \sim (\Delta t)^{-1} a^h_e(\0,\cdot;\0,\cdot)
+ (\Delta t)^{-1}(s_0\,\cdot,\cdot)_{\Omega_p}, 
\ds A_{\bu_p p_p} \sim b_p(\cdot,\cdot),\,\,
A_{\bsi_f \bvarphi} \sim b_{\bn_f}(\cdot,\cdot),\,\,
A_{\bu_p \lambda} \sim b_{\Gamma}(\cdot,\cdot),  \\[1ex]
& A_{\bsi_p \btheta} \sim b_{\bn_p}(\cdot,\cdot),\,\,
\ds A_{\bvarphi \bvarphi} \sim c_{\BJS}(\cdot,\0;\cdot,\0),\,\,
A_{\bvarphi \btheta} \sim c_{\BJS}(\cdot,\0;\0,\cdot),\,\,
A_{\btheta \btheta} \sim c_{\BJS}(\0,\cdot;\0,\cdot),\,\,
A_{\bvarphi \lambda} \sim c_{\Gamma}(\cdot,\0;\cdot),\\[1ex]
& \ds
A_{\btheta \lambda} \sim c_{\Gamma}(\0,\cdot;\cdot),\,\,
A_{\bsi_f \bu_f} \sim b_f(\cdot,\cdot), \,\,
A_{\bsi_f \bgamma_f} \sim b^h_{\skf}(\cdot,\cdot),\,\,
A_{\bsi_p \bu_s} \sim b_s(\cdot,\cdot),\,\,
A_{\bsi_p \bgamma_p} \sim b^h_{\skp}(\cdot,\cdot), 
\end{align*}
where the notation $A \sim a$ means that the matrix $A$ is associated with
the bilinear form $a$. Denoting the algebraic vectors corresponding to the
variables $\ubsi^m_h$, $\ubvarphi^m_h$, and $\ubu^m_h$ in the same way,
we can then write the system
\eqref{eq:Stokes-Biot-fully-discrete-formulation} in a matrix-vector form as
\begin{equation}\label{matrix-form}
 \left( \begin{array}{ccc}
   (\Delta t)^{-1}\,\cE_h + \cA_h & \cB_1^\rt & \cB_h^\rt \\
   -\cB_1 & \cC & 0 \\
   -\cB_h & 0 & 0
 \end{array} \right)
 \left(\begin{array}{c}
   \ubsi^m_h \\ \ubvarphi^m_h \\ \ubu^m_h \end{array} \right) =
 \left( \begin{array}{c}
   \bF + (\Delta t)^{-1}\cE_h(\ubsi^{m-1}_h) \\ 0 \\ \bG \end{array} \right).
  \end{equation}

As we noted in Section~\ref{seq:quadrature}, due to the the use of the
vertex quadrature rule, the degrees of freedom (DOFs) of the Stokes
stress $\bsi^m_{fh}$, Darcy velocity $\bu^m_{ph}$ and poroelastic
stress tensor $\bsi^m_{ph}$ associated with a mesh vertex become
decoupled from the rest of the DOFs.  As a result, the assembled mass
matrices have a block-diagonal structure with one block per mesh
vertex.  The dimension of each block equals the number of DOFs
associated with the vertex.  These matrices can then be easily
inverted with local computations. Inverting each local
block in $A_{\bu_p \bu_p}$ allows for expressing the Darcy
velocity DOFs associated with a vertex in terms of the Darcy pressure
$p^m_{ph}$ at the centers of the elements that share the vertex, as
well as the trace unknown $\lambda^m_h$ on neighboring edges (faces)
for vertices on $\Gamma_{fp}$.
Similarly, inverting each local block in $A_{\bsi_f \bsi_f}$
allows for expressing the Stokes stress DOFs
associated with a vertex in terms of neighboring Stokes velocity
$\bu^m_{fh}$, vorticity $\bgamma^m_{fh}$, and trace $\bvarphi^m_h$.
Finally, inverting each local block in $A_{\bsi_p \bsi_p}$
allows for expressing the poroelastic stress DOFs
associated with a vertex in terms of neighboring
Darcy pressure $p^m_{ph}$, structure velocity $\bu^m_{sh}$,
structure rotation $\bgamma^m_{ph}$, and trace $\btheta^m_h$.
Then we have
\begin{equation}\label{eq:sigma-f-sigmap-formulae}
\begin{array}{l}
  \bu^m_{ph} \,=\, -\,A^{-1}_{\bu_p \bu_p} A^\rt_{\bu_p p_p}\,p^m_{ph} - A^{-1}_{\bu_p \bu_p} A^\rt_{\bu_p \lambda}\,\lambda^m_h, \\ [2ex]
  \bsi^m_{fh} \,=\, -\,A^{-1}_{\bsi_f \bsi_f} A^\rt_{\bsi_f \bvarphi}\,\bvarphi^m_h - A^{-1}_{\bsi_f \bsi_f} A^\rt_{\bsi_f \bu_f}\,\bu^m_{fh} - A^{-1}_{\bsi_f \bsi_f} A^\rt_{\bsi_f \bgamma_f}\,\bgamma^m_{fh}, \\ [2ex]
\bsi^m_{ph} \,=\, -\,A^{-1}_{\bsi_p \bsi_p}A^\rt_{\bsi_p p_p}\,p^m_{ph} - A^{-1}_{\bsi_p \bsi_p} A^\rt_{\bsi_p \btheta}\,\btheta^m_h - A^{-1}_{\bsi_p \bsi_p} A^\rt_{\bsi_p \bu_s}\,\bu^m_{sh} - A^{-1}_{\bsi_p \bsi_p} A^\rt_{\bsi_p \bgamma_p}\,\bgamma^m_{ph}.
\end{array}
\end{equation}
The reduced matrix associated to \eqref{matrix-form} in terms of  
$(p^m_{ph}, \bvarphi^m_h, \btheta^m_h, \lambda^m_h, \bu^m_{fh}, \bu^m_{sh}, \bgamma^m_{fh}, \bgamma^m_{ph})$ is given by
\begin{equation}\label{matrix1}
    \arraycolsep 3pt
\left(\begin{array}{cccccccc}
A_{p_p \bsi_p p_p} + A_{p_p \bu_p p_p} & 0 & -A_{p_p \bsi_p \btheta} & A_{p_p \bu_p \lambda} & 0 & -A_{p_p \bsi_p \bu_s} & 0 & -A_{p_p \bsi_p \bgamma_p} \\ [1ex]
0 & A_{\bvarphi \bvarphi}{+}A_{\bvarphi \bsi_f \bvarphi} & A^\rt_{\bvarphi \btheta} & A^\rt_{\bvarphi \lambda} & A_{\bu_f \bsi_f \bvarphi} & 0 & A_{\bgamma_f \bsi_f \bvarphi} & 0 \\ [1ex]
A^\rt_{p_p \bsi_p \theta} & A_{\bvarphi \btheta} &A_{\btheta \btheta}{+}A_{\btheta \bsi_p \btheta} & A^\rt_{\btheta \lambda} & 0 & A_{\bu_s \bsi_p \btheta} & 0 & A_{\bgamma_p \bsi_p \btheta} \\ [1ex]
A^\rt_{p_p \bu_p \lambda} & -A_{\bvarphi \lambda} & -A_{\btheta \lambda} & A_{\lambda \bu_p \lambda} & 0 & 0 & 0 & 0 \\ [1ex] 
0 & A^\rt_{\bu_f \bsi_f \bvarphi} & 0 & 0 & A_{\bu_f \bsi_f \bu_f} & 0 & A_{\bu_f \bsi_f \bgamma_f} & 0 \\ [1ex]
A^\rt_{p_p \bsi_p \bu_s} & 0 & A^\rt_{\bu_s \bsi_p \btheta} & 0 & 0 & A_{\bu_s \bsi_p \bu_s} & 0 & A_{\bu_s \bsi_p \bgamma_p} \\ [1ex]
0 & A^\rt_{\bgamma_f \bsi_f \bvarphi} & 0 & 0 & A^\rt_{\bu_f \bsi_f \bgamma_f} & 0 & A_{\bgamma_f\bsi_f \bgamma_f} & 0 \\ [1ex]
A^\rt_{p_p \bsi_p \bgamma_p} & 0 & A^\rt_{\bgamma_p \bsi_p \btheta} & 0 & 0 & A^\rt_{\bu_s \bsi_p \bgamma_p} & 0 & A_{\gamma_p \bsi_p \bgamma_p}
\end{array}\right)
\end{equation}
where
\begin{align}\label{eq:matrices-1}
  & \ds A_{p_p \bsi_p p_p} = A_{p_p p_p} - A_{\bsi_p p_p} A^{-1}_{\bsi_p \bsi_p} A^\rt_{\bsi_p p_p}, \,\,
A_{p_p \bu_p p_p} =   A_{\bu_p p_p} A^{-1}_{\bu_p \bu_p} A^\rt_{\bu_p p_p},\,\,
A_{p_p \bsi_p \btheta} = A_{\bsi_p p_p} A^{-1}_{\bsi_p \bsi_p} A^\rt_{\bsi_p \btheta}, \nonumber \\[0.5ex]
& \ds A_{p_p \bu_p \lambda} = A_{\bu_p p_p} A^{-1}_{\bu_p \bu_p} A^\rt_{\bu_p \lambda},\,\,
A_{p_p \bsi_p \bu_s} = A_{\bsi_p p_p} A^{-1}_{\bsi_p \bsi_p} A^\rt_{\bsi_p \bu_s},\,\,
A_{p_p \bsi_p \bgamma_p} = A_{\bsi_p p_p} A^{-1}_{\bsi_p \bsi_p} A^\rt_{\bsi_p \bgamma_p}, \nonumber \\[0.5ex]  
& \ds A_{\bvarphi \bsi_f \bvarphi} = A_{\bsi_f \bvarphi} A^{-1}_{\bsi_f \bsi_f}A^\rt_{\bsi_f \bvarphi},\,\, 
A_{\btheta \bsi_p \btheta} = A_{\bsi_p \btheta} A^{-1}_{\bsi_p \bsi_p} A^\rt_{\bsi_p \btheta}, \nonumber \\[0.5ex] 
& \ds A_{\lambda \bu_p \lambda} = A_{\bu_p \lambda} A^{-1}_{\bu_p \bu_p} A^\rt_{\bu_p \lambda},\,\,
A_{\bu_f \bsi_f \bvarphi} = A_{\bsi_f \bvarphi} A^{-1}_{\bsi_f \bsi_f} A^\rt_{\bsi_f \bu_f},\,\, 
A_{\bu_f \bsi_f \bu_f} = A_{\bsi_f \bu_f} A^{-1}_{\bsi_f \bsi_f} A^\rt_{\bsi_f \bu_f}, \\[0.5ex]
& \ds A_{\bu_f \bsi_f \bgamma_f} = A_{\bsi_f \bu_f} A^{-1}_{\bsi_f \bsi_f} A^\rt_{\bsi_f \bgamma_f},\,\,
A_{\bu_s \bsi_p \btheta} = A_{\bsi_p \btheta} A^{-1}_{\bsi_p \bsi_p} A^\rt_{\bsi_p \bu_s},\,\,
A_{\bu_s \bsi_p \bu_s} = A_{\bsi_p \bu_s} A^{-1}_{\bsi_p \bsi_p} A^\rt_{\bsi_p \bu_s}, \nonumber \\[0.5ex]
& \ds A_{\bu_s \bsi_p \bgamma_p} = A_{\bsi_p \bu_s} A^{-1}_{\bsi_p \bsi_p} A^\rt_{\bsi_p \bgamma_p},\,\,
A_{\bgamma_p \bsi_p \bgamma_p} = A_{\bsi_p \bgamma_p} A^{-1}_{\bsi_p \bsi_p} A^\rt_{\bsi_p \bgamma_p},\,\,
A_{\bgamma_p \bsi_p \btheta} = A_{\bsi_p \btheta} A^{-1}_{\bsi_p \bsi_p} A^\rt_{\bsi_p \bgamma_p}, \nonumber \\[0.5ex]
& \ds A_{\bgamma_f \bsi_f \bgamma_f} = A_{\bsi_f \bgamma_f} A^{-1}_{\bsi_f \bsi_f} A^\rt_{\bsi_f \bgamma_f},\,\,
A_{\bgamma_f \bsi_f \bvarphi} = A_{\bsi_f \bvarphi} A^{-1}_{\bsi_f \bsi_f} A^\rt_{\bsi_f \bgamma_f}. \nonumber
\end{align}
Furthermore, due to the vertex quadrature rule, the vorticity and
structure rotation DOFs corresponding to each vertex of the
grid become decoupled from the rest of the DOFs, leading
to block-diagonal matrices $A_{\bgamma_f \bsi_f \bgamma_f}$ and
$A_{\bgamma_p \bsi_p \bgamma_p}$. Recalling the matrix definitions in 
\eqref{eq:matrices-1}, each block is symmetric and positive definite
and thus locally invertible, due
the positive definiteness of $A^{-1}_{\bsi_f \bsi_f}$ and $A^{-1}_{\bsi_p \bsi_p}$
and the inf-sup condition \eqref{eq:inf-sup-B-semidiscrete}. We then have
\begin{equation}\label{eq:gammaf-gammap-formulae}
\begin{array}{l}
\bgamma^m_{fh} \,=\, -\,A^{-1}_{\bgamma_f \bsi_f \bgamma_f} A_{\bgamma_f \bsi_f \bvarphi}\,\bvarphi^m_h - A^{-1}_{\bgamma_f \bsi_f \bgamma_f} A^\rt_{\bu_f \bsi_f \bgamma_f}\,\bu^m_{fh}, \\ [2ex]
\bgamma^m_{ph} \,=\, -\,A^{-1}_{\bgamma_p \bsi_p \bgamma_p} A^\rt_{p_p \bsi_p \bgamma_p}\,p^m_{ph} - A^{-1}_{\bgamma_p \bsi_p \bgamma_p} A_{\bgamma_p \bsi_p \btheta}\,\btheta^m_h - A^{-1}_{\bgamma_p \bsi_p \bgamma_p} A^\rt_{\bu_s \bsi_p \bgamma_p}\,\bu^m_{sh},
\end{array}
\end{equation}
and using some algebraic manipulation, we obtain the reduced problem
$\bA \vec{\bp}^m_h = \vec{\bF}$, with vector solution $\vec{\bp}^m_h
:= (p^m_{ph}, \bvarphi^m_h, \btheta^m_h, \lambda^m_h, \bu^m_{fh},
\bu^m_{sh})$ and matrix
\begin{equation}\label{eq:cell-centered-system-2}
\bA = 
\left(\begin{array}{cccccc}
  \wt{A}_{p_p \bsi_p p_p} + A_{p_p \bu_p p_p}
  & 0
  & -\wt{A}_{p_p \bsi_p \theta}
  & A_{p_p \bu_p \lambda}
  & 0
  & -\wt{A}_{p_p \bsi_p \bu_s}
  \\ [1ex]

  0
  & \wt{A}_{\bvarphi \bsi_f \bvarphi}{+}A_{\bvarphi \bvarphi}
  & A^\rt_{\bvarphi \btheta}
  & A^\rt_{\bvarphi \lambda}
  & \wt{A}_{\bu_f \bsi_f \bvarphi}
  & 0
  \\ [1ex]

  \wt{A}^\rt_{p_p \bsi_p \btheta}
  & A_{\bvarphi \btheta}
  & \wt{A}_{\btheta \bsi_p \btheta}{+}A_{\btheta \btheta}
  & A^\rt_{\btheta \lambda}
  & 0
& \wt{A}_{\bu_s \bsi_p \btheta}
  \\ [1ex]

  A^\rt_{p_p \bu_p \lambda}
  & -A_{\bvarphi \lambda}
  & -A_{\btheta \lambda}
  & A_{\lambda \bu_p \lambda}
  & 0
  & 0
  \\ [1ex]

  0
  & \wt{A}^\rt_{\bu_f \bsi_f \bvarphi}
  & 0
  & 0
  & \wt{A}_{\bu_f \bsi_f \bu_f}
  & 0
  \\ [1ex]

  \wt{A}^\rt_{p_p \bsi_p \bu_s}
  & 0
  & \wt{A}^\rt_{\bu_s \bsi_p \btheta}
  & 0
  & 0
  & \wt{A}_{\bu_s \bsi_p \bu_s}
  
\end{array}\right)
\end{equation}
where 
\begin{align}\label{eq:matrices-2}
  & \wt{A}_{p_p \bsi_p p_p} = A_{p_p \bsi_p p_p}
  + A_{p_p \bsi_p \bgamma_p} A^{-1}_{\bgamma_p \bsi_p \bgamma_p} A^\rt_{p_p \bsi_p \bgamma_p},\,\,
\wt{A}_{p_p \bsi_p \btheta} = A_{p_p \bsi_p \btheta} - A_{p \bsi_p \btheta} A^{-1}_{\bgamma_p \bsi_p \bgamma_p} A^\rt_{\bgamma_p \bsi_p \btheta}, \nonumber \\[0.5ex]
& \wt{A}_{p_p \bsi_p \bu_s} = A_{p_p \bsi_p \bu_s} - A_{p_p \bsi_p \bgamma_p} A^{-1}_{\bgamma_p \bsi_p \bgamma_p} A^\rt_{\bu_s \bsi_p \bgamma_p},\,\,
\wt{A}_{\bvarphi \bsi_f \bvarphi} = A_{\bvarphi \bsi_f \bvarphi} - A_{\bgamma_f \bsi_f \bvarphi} A^{-1}_{\bgamma_f \bsi_f \bgamma_f} A^\rt_{\bgamma_f \bsi_f \bvarphi}, \nonumber \\[0.5ex]
& \wt{A}_{\bu_f \bsi_f \bvarphi} = A_{\bu_f \bsi_f \bvarphi} - A_{\bgamma_f \bsi_f \bvarphi} A^{-1}_{\bgamma_f \bsi_f \bgamma_f} A^\rt_{\bu_f \bsi_f \bgamma_f},\,\,
\wt{A}_{\btheta \bsi_p \btheta} = A_{\btheta \bsi_p \btheta} - A_{\bgamma_p \bsi_p \btheta} A^{-1}_{\bgamma_p \bsi_p \bgamma_p} A^\rt_{\bgamma_p \bsi_p \btheta},\qquad \\[0.5ex]
&\wt{A}_{\bu_s \bsi_p \btheta} = A_{\bu_s \bsi_p \btheta} - A_{\bgamma_p \bsi_p \btheta} A^{-1}_{\bgamma_p \bsi_p \bgamma_p} A^\rt_{\bu_s \bsi_p \bgamma_p},\,\,
\wt{A}_{\bu_f \bsi_f \bu_f} = A_{\bu_f \bsi_f \bu_f} - A_{\bu_f \bsi_f \bgamma_f} A^{-1}_{\bgamma_f \bsi_f \bgamma_f} A^\rt_{\bu_f \bsi_f \bgamma_f}, \nonumber \\[0.5ex]
& \wt{A}_{\bu_s \bsi_p \bu_s} = A_{\bu_s \bsi_p \bu_s} - A_{\bu_s \bsi_p \bgamma_p} A^{-1}_{\bgamma_p \bsi_p \bgamma_p} A^\rt_{\bu_s \bsi_p \bgamma_p}, \nonumber
\end{align}
and the right hand side vector $\vec{\bF}$ has been obtained by
transforming the right-hand side in
\eqref{eq:Stokes-Biot-fully-discrete-formulation} accordingly to the
procedure above. Note that, after solving the problem with matrix
\eqref{eq:cell-centered-system-2}, we can recover
$\bu^m_{ph}, \bsi^m_{fh}, \bsi^m_{ph}$ and $\bgamma^m_{fh},
\bgamma^m_{ph}$ through the formulae 
\eqref{eq:sigma-f-sigmap-formulae} and
\eqref{eq:gammaf-gammap-formulae}, respectively, thus
obtaining the full solution to
\eqref{eq:Stokes-Biot-fully-discrete-formulation}.

\begin{lem}
  The cell-centered finite difference system for the pressure-velocities-traces
  problem \eqref{eq:cell-centered-system-2} is positive definite.	
\end{lem}
\begin{proof}
Consider a vector $\vec{\bq}^\rt = (w^\rt_{ph}\,\, \bpsi^\rt_h\,\, \bphi^\rt_h\,\,
\xi^\rt_h\,\, \bv^\rt_{fh}\,\, \bv^\rt_{sh}) \neq \vec{\0}$. Employing the matrices
in \eqref{eq:matrices-1} and \eqref{eq:matrices-2} and some algebraic
manipulations, we obtain
\begin{equation}\label{eq:pd-reduced-matrix}
  \begin{array}{l}
\ds \vec{\bq}^\rt\,\bA\,\vec{\bq} \,=\,
w^\rt_{ph}\big( A_{p_p p_p} - A_{\bsi_p p_p} A^{-1}_{\bsi_p \bsi_p} A^\rt_{\bsi_p p_p}\big)w_{ph}
+ w^\rt_{ph} A_{p_p \bsi_p \bgamma_p} A^{-1}_{\bgamma_p \bsi_p \bgamma_p}
A^\rt_{p_p \bsi_p \bgamma_p} w_{ph}
\\ [2ex]
\ds\quad +\,\,\big( A^\rt_{\bu_p p_p}\,w_{ph}
+ A^\rt_{\bu_p \lambda}\,\xi_h \big)^\rt A^{-1}_{\bu_p \bu_p}
\big( A^\rt_{\bu_p p_p}\,w_{ph} + A^\rt_{\bu_p \lambda}\,\xi_h \big) 
+ \,\, (\bpsi^\rt_h \,\, \bphi^\rt_h)
\left(\begin{array}{cc} A_{\bvarphi \bvarphi} & A^\rt_{\bvarphi \btheta} \\
  A_{\bvarphi \btheta} & A_{\btheta \btheta} \end{array} \right)
\left( \begin{array}{c} \bpsi_h \\ \bphi_h \end{array} \right)\\ [2ex]
\ds\quad +\,\, (\bpsi^\rt_h \,\, \bv^\rt_{fh})
\left(\begin{array}{cc} \wt{A}_{\bvarphi \bsi_f \bvarphi}
  & \wt{A}_{\bu_f \bsi_f \bvarphi} \\
  \wt{A}^\rt_{\bu_f \bsi_f \bvarphi} & \wt{A}_{\bu_f \bsi_f \bu_f} \end{array} \right)
\left( \begin{array}{c} \bpsi_h \\ \bv_{fh} \end{array} \right)
+\,\, (\bphi^\rt_h \,\, \bv^\rt_{sh})
\left(\begin{array}{cc} \wt{A}_{\btheta \bsi_p \btheta}
  & \wt{A}_{\bu_s \bsi_p \btheta} \\
  \wt{A}^\rt_{\bu_s \bsi_p \btheta} & \wt{A}_{\bu_s \bsi_p \bu_s} \end{array} \right)
\left( \begin{array}{c} \bphi_h \\ \bv_{sh} \end{array} \right).
\end{array} 
\end{equation}
Now, we focus on analyzing the six terms in the right-hand side of
\eqref{eq:pd-reduced-matrix}.  The first term is non-negative due to
\cite[Theorem~7.7.6]{Horn-Johnson} and the fact that the matrix
$A_{p_p p_p} - A_{\bsi_p p_p} A^{-1}_{\bsi_p \bsi_p} A^\rt_{\bsi_p p_p}$ is a Schur
complement of the matrix
\begin{equation*}
\left(\begin{array}{cc}
A_{\bsi_p \bsi_p} & A^\rt_{\bsi_p p_p} \\[1ex] 
A_{\bsi_p p_p} & A_{p_p p_p}
\end{array}\right),
\end{equation*}
which is positive semi-definite as a consequence of the ellipticity
property of the operator $a_e$ (cf. \eqref{eq:bilinear-forms-1} and
\eqref{eq:monotonicity-E-A}). The second term is nonnegative, since
the matrix $A_{\bgamma_p \bsi_p \bgamma_p}$ is positive definite, as noted
in \eqref{eq:gammaf-gammap-formulae}.
The third term is positive for $(w^\rt_{ph}\,\, \xi^\rt_h) \neq \vec{\0}$,
due to the positive-definiteness of
$A^{-1}_{\bu_p \bu_p}$ and the inf-sup condition
\eqref{eq:p-lambda-bound semidiscrete}.
The fourth term is non-negative since the operator $\cC$
(cf. \eqref{eq:pos-sem-def-c}) is positive semi-definite. The matrices in the
last two terms are Schur complements of the matrices
$$
A_f := \left( \begin{array}{ccc}
  A_{\bvarphi \bsi_f \bvarphi} & A_{\bu_f \bsi_f \bvarphi} & A_{\bgamma_f \bsi_f \bvarphi}\\
  A^\rt_{\bu_f \bsi_f \bvarphi} & A_{\bu_f \bsi_f \bu_f} & A_{\bu_f \bsi_f \bgamma_f}\\
  A^\rt_{\bgamma_f \bsi_f \bvarphi} & A^\rt_{\bu_f \bsi_f \bgamma_f} & A_{\bgamma_f \bsi_f \bgamma_f}
\end{array}\right)
\quad \mbox{and} \quad
A_p := \left( \begin{array}{ccc}
  A_{\btheta \bsi_p \btheta} & A_{\bu_s \bsi_p \btheta} & A_{\bgamma_p \bsi_p \btheta}\\
  A^\rt_{\bu_s \bsi_p \btheta} & A_{\bu_s \bsi_p \bu_s} & A_{\bu_s \bsi_p \bgamma_p}\\
  A^\rt_{\bgamma_p \bsi_p \btheta} & A^\rt_{\bu_s \bsi_p \bgamma_p} & A_{\bgamma_p \bsi_p \bgamma_p}
\end{array}\right),
$$
respectively, which are positive definite. In particular, for
$\vec{\bv}_f^\rt = (\bpsi^\rt_h\,\, \bv^\rt_{fh}\,\, \bchi^\rt_{fh}) \neq \vec{\0}$
and $\vec{\bv}_p^\rt = (\bphi^\rt_h\,\, \bv^\rt_{sh}\,\, \bchi^\rt_{ph})
\neq \vec{\0}$, we have
\begin{align*}
& \vec{\bv}_f^\rt A_f \vec{\bv}_f = \big(A^\rt_{\bsi_f \bvarphi}\,\bpsi_h
+ A^\rt_{\bsi_f \bu_f}\,\bv_{fh}
+ A^\rt_{\bsi_f \bgamma_f}\,\bchi_{fh} \big)^\rt A^{-1}_{\bsi_f \bsi_f}
\big(A^\rt_{\bsi_f \bvarphi}\,\bpsi_h
+ A^\rt_{\bsi_f \bu_f}\,\bv_{fh}
+ A^\rt_{\bsi_f \bgamma_f}\,\bchi_{fh} \big) > 0, \\
& \vec{\bv}_p^\rt A_p \vec{\bv}_p = \big( A^\rt_{\bsi_p \btheta}\,\bphi_h +
A^\rt_{\bsi_p \bu_s}\,\bv_{sh} + A^\rt_{\bsi_p \bgamma_p}\,\bchi_{ph} \big)^\rt
A^{-1}_{\bsi_p \bsi_p} \big( A^\rt_{\bsi_p \btheta}\,\bphi_h +
A^\rt_{\bsi_p \bu_s}\,\bv_{sh} + A^\rt_{\bsi_p \bgamma_p}\,\bchi_{ph} \big) > 0,
\end{align*}
due to the positive-definiteness of $A^{-1}_{\bsi_f \bsi_f}$ and $A^{-1}_{\bsi_p \bsi_p}$,
along with the combined inf-sup condition for
$\cB_h(\ubtau_h)(\ubv_h) + \cB_1(\ubtau_h)(\ubpsi_h)$. The latter follows from
the inf-sup conditions \eqref{eq:inf-sup-Bh} and \eqref{eq:inf-sup-B1h},
using that \eqref{eq:inf-sup-B1h} holds in the kernel of $\cB_h$.
Then, applying again
\cite[Theorem~7.7.6]{Horn-Johnson}, we conclude that the last two terms
in \eqref{eq:pd-reduced-matrix} are positive for
$(\bpsi^\rt_h \,\, \bv^\rt_{fh}) \neq \vec{\0}$ and
$(\bphi^\rt_h \,\, \bv^\rt_{sh}) \neq \vec{\0}$. Therefore
$\vec{\bq}^\rt\,\bA\,\vec{\bq} > 0$ for all $\vec{\bq}\neq \vec{\0}$,
implying that the matrix $\bA$ from \eqref{eq:cell-centered-system-2}
is positive definite.
\end{proof}

\begin{rem}
The solution of the reduced system with the matrix $\bA$ from
\eqref{eq:cell-centered-system-2} results in significant computational
savings compared to the original system \eqref{matrix-form}. In
particular, five of the eleven variables have been eliminated. Three of
the remaining variables are Lagrange multipliers that appear only on
the interface $\Gamma_{fp}$. The other three are the cell-centered
velocities and Darcy pressure, with only $n$ DOFs per element in the
Stokes region and $n + 1$ DOFs per element in the Biot region, which are
the smallest possible number of DOFs for the sub-problems. Furthermore,
since the reduced system is positive definite, efficient iterative solvers
such as GMRES can be utilized for its solution.
\end{rem}

\section{Numerical results}\label{sec:numerical-results}

In this section we present numerical results that illustrate the
behavior of the fully discrete multipoint stress-flux mixed finite element method
\eqref{eq:Stokes-Biot-fully-discrete-formulation}.  Our implementation
is in two dimensions and it is based on {\tt FreeFem++}
\cite{freefem}, in conjunction with the direct linear solver {\tt
  UMFPACK} \cite{umfpack}.  For spatial discretization, we use the
$(\bbBDM_1-\bP_0-\bbP_1)$ spaces for Stokes, the
$(\bbBDM_1-\bP_0-\bbP_1)-(\bBDM_1-\rP_0)$ spaces for Biot, and either
$(\bP_1 - \bP_1 - \rP_1)$ or $\bP^\dc_1-\bP^\dc_1-\rP^\dc_1$ for the
Lagrange multipliers.
We present three examples. Example~1 is used to corroborate the rates
of convergence.  Example 2 is a simulation of the coupling of surface
and subsurface hydrological systems, focusing on the qualitative
behavior of the solution.  Example 3 illustrates an application to
flow in a poroelastic medium with an irregularly shaped cavity, using
physically realistic parameters.

\subsection{Example 1: convergence test}

In this test we study the convergence rates for the space discretization
using an analytical solution.  The domain is $\overline\Omega =
\overline\Omega_f\cup\overline\Omega_p$, where $\Omega_f = (0,1)\times
(0,1)$ and $\Omega_p = (0,1)\times (-1,0)$.  In particular, the upper half is associated with the Stokes flow, while the
lower half represents the flow in the poroelastic structure governed
by the Biot system, see Figure~\ref{fig:ex1} (left). 
The interface conditions are enforced
along the interface $\Gamma_{fp}$.  The parameters and analytical solution are
given in Figure~\ref{fig:ex1} (right). The solution is designed to 
satisfy the interface conditions \eqref{eq:interface-conditions}.
The right hand side functions $\f_f, q_f, \f_p$ and $q_p$ are computed
from \eqref{eq:Stokes-2}--\eqref{eq:Biot-model} using the true
solution.  The model problem is then complemented with the appropriate
boundary conditions, which are described in Figure~\ref{fig:ex1} (left),
and initial data.  Notice that the boundary
conditions for $\bsi_f, \bu_f, \bu_p, \bsi_p$, and $\bbeta_p$
(cf. \eqref{eq:Stokes-2}--\eqref{eq:Biot-model}) are not homogeneous
and therefore the right-hand side of the resulting system must be
modified accordingly. The total simulation time for this example is
$T=0.01$ and the time step is $\Delta t = 10^{-3}$.  The time
step is sufficiently small, so that the time discretization error does
not affect the convergence rates.

\begin{figure}
  \begin{minipage}{0.25\textwidth}
\begin{center}
  \includegraphics[width=0.9\textwidth]{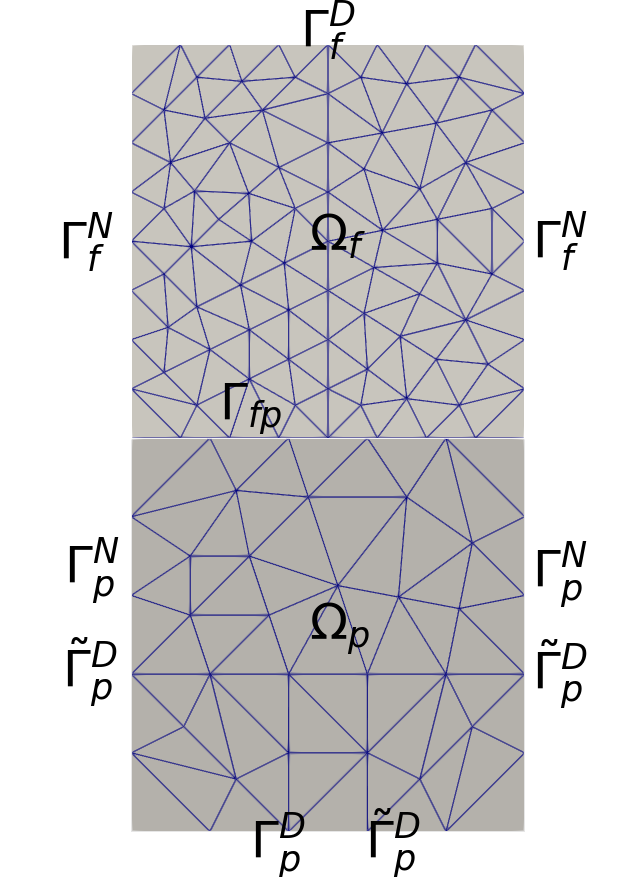}
\end{center}
  \end{minipage}
  \begin{minipage}{0.75\textwidth}
    \begin{align*}
& \mu = 1,\quad
\alpha_p = 1,\quad
\lambda_p = 1,\quad
\mu_p = 1,\quad
s_0 = 1,\quad
\bK = \bI,\quad
\alpha_{\BJS} = 1, \\
& \bu_f = \pi\,\cos(\pi\,t)
\left(\begin{array}{c}
-3x + \cos(y) \\ y + 1
\end{array}\right), \\
& p_f = \exp(t)\,\sin(\pi\,x)\,\cos\left(\frac{\pi\,y}{2}\right) + 2\,\pi\,\cos(\pi\,t),\\
&p_p = \exp(t)\,\sin(\pi\,x)\,\cos\left(\frac{\pi\,y}{2}\right), \\
&\bu_p = -\frac{1}{\mu}\,\bK\,\nabla p_p, \quad
\bbeta_p = \sin(\pi\,t)\,
\left(\begin{array}{c}
-3x + \cos(y) \\ y + 1
\end{array}\right).
\end{align*}
    \end{minipage}
\caption{Example 1, domain and coarsest mesh level (left), parameters and analytical solution (right).}
\label{fig:ex1}
\end{figure}

Tables~\ref{table2-example1} and \ref{table1-example1} show the convergence history for a sequence of quasi-uniform mesh refinements with non-matching grids along
the interface
employing conforming and non-conforming spaces for the Lagrange multipliers (cf. \eqref{eq:conforming-spaces-Yh}--\eqref{eq:non-conforming-spaces-Yh}), respectively. In the tables, $h_{f}$ and $h_p$ denote the mesh sizes in $\Omega_f$ and $\Omega_p$, respectively, while the mesh sizes for their traces on $\Gamma_{fp}$ are
$h_{tf}$ and $h_{tp}$, satisfying $h_{tf} = \frac{5}{8}\,h_{tp}$.
We note that the Stokes pressure and the displacement at time $t_m$ are recovered by the post-processed formulae $p^m_f = -\frac{1}{n} (\tr(\bsi^m_f) - 2\,\mu\,q^m_f)$ (cf. \eqref{eq:Stokes-2}) and $\bbeta^m_p = \bbeta^{m-1}_p + \Delta t\,\bu^m_s$ (cf.
  Remark~\ref{rem:post-processed-formulae}), respectively.
The results illustrate that spatial rates of convergence $\cO(h)$, as provided by Theorem~\ref{thm:error-estimate-multipoint}, are attained for all subdomain variables in their natural norms.
The Lagrange multiplier variables, which are approximated in $\bP_1-\bP_1-\rP_1$ and $\bP^\dc_1-\bP^\dc_1-\rP^\dc_1$, exhibit rates of convergence $\cO(h^{3/2})$ and $\cO(h^2)$ in the $\H^{1/2}$ and $\L^2$-norms on $\Gamma_{fp}$, respectively, which is consistent with the order of approximation.

\begin{table}[h]
\begin{center}
\begin{tabular}{c||cc|cc|cc|cc}
\hline
& \multicolumn{2}{|c|}{$\|\be_{\bsi_f}\|_{\ell^2(0,T;\bbX_f)}$}  & \multicolumn{2}{|c|}{$\|\be_{\bu_f}\|_{\ell^2(0,T;\bV_f)}$} & \multicolumn{2}{|c|}{$\|\be_{\bgamma_f}\|_{\ell^2(0,T;\bbQ_f)}$} & \multicolumn{2}{|c}{$\|\be_{p_f}\|_{\ell^2(0,T;\L^2(\Omega_f))}$} \\ 
$h_f$ & error & rate & error & rate & error & rate & error & rate \\  \hline
0.1964 & 2.2E-02 &   --   & 2.7E-02 &   --   & 2.4E-03 &   --   & 6.3E-03 &   --   \\ 
0.0997 & 1.2E-02 & 0.95 & 1.4E-02 & 1.00 & 9.3E-04 & 1.41 & 3.1E-03 & 1.05 \\ 
0.0487 & 5.7E-03 & 0.99 & 6.8E-03 & 0.99 & 4.2E-04 & 1.11 & 1.6E-03 & 0.93 \\
0.0250 & 2.9E-03 & 1.04 & 3.4E-03 & 1.04 & 2.0E-04 & 1.13 & 7.8E-04 & 1.07 \\
0.0136 & 1.4E-03 & 1.14 & 1.7E-03 & 1.15 & 9.4E-05 & 1.23 & 3.9E-04 & 1.15 \\
0.0072 & 7.1E-04 & 1.08 & 8.4E-04 & 1.10 & 4.7E-05 & 1.09 & 2.0E-04 & 1.02 \\ 
\hline 
\end{tabular}
			
\medskip
	
\begin{tabular}{c||cc|cc|cc|cc|cc}
\hline
& \multicolumn{2}{|c|}{$\|\be_{\bsi_p}\|_{\ell^\infty(0,T;\bbX_p)}$}  & \multicolumn{2}{|c|}{$\|\be_{\bu_s}\|_{\ell^2(0,T;\bV_s)}$} & \multicolumn{2}{|c|}{$\|\be_{\bgamma_p}\|_{\ell^2(0,T;\bbQ_p)}$} & \multicolumn{2}{|c}{$\|\be_{\bu_p}\|_{\ell^2(0,T;\bV_p)}$} & \multicolumn{2}{|c}{$\|\be_{p_p}\|_{\ell^\infty(0,T;\W_p)}$} \\ 
$h_p$ & error & rate & error & rate & error & rate & error & rate & error & rate \\  \hline
0.2828 & 2.7E-01 &   --   & 4.3E-02 &   --   & 3.4E-02 &   --   & 1.0E-01 &   --   & 7.5E-02 &   --   \\ 
0.1646 & 1.4E-01 & 1.27 & 2.2E-02 & 1.23 & 9.4E-03 & 2.38 & 5.2E-02 & 1.27 & 3.8E-02 & 1.25 \\ 
0.0779 & 6.7E-02 & 0.97 & 1.1E-02 & 0.96 & 2.2E-03 & 1.96 & 2.5E-02 & 1.00 & 1.9E-02 & 0.93 \\
0.0434 & 3.4E-02 & 1.17 & 5.4E-03 & 1.19 & 5.8E-04 & 2.25 & 1.2E-02 & 1.24 & 9.4E-03 & 1.22 \\
0.0227 & 1.7E-02 & 1.06 & 2.7E-03 & 1.07 & 2.0E-04 & 1.68 & 5.9E-03 & 1.08 & 4.7E-03 & 1.07 \\
0.0124 & 8.4E-03 & 1.15 & 1.4E-03 & 1.15 & 8.1E-05 & 1.48 & 2.9E-03 & 1.15 & 2.4E-03 & 1.14 \\ 
\hline 
\end{tabular}
			
\medskip
			
\begin{tabular}{cc||c||cc||c||cc|cc}
\hline
\multicolumn{2}{c||}{$\|\be_{\bbeta_p}\|_{\ell^2(0,T;\bL^2(\Omega_p))}$} &
& \multicolumn{2}{|c||}{$\|\be_{\bvarphi}\|_{\ell^2(0,T;\bLambda_f)}$} & & \multicolumn{2}{|c|}{$\|\be_{\btheta}\|_{\ell^2(0,T;\bLambda_s))}$} & \multicolumn{2}{|c}{$\|\be_{\lambda}\|_{\ell^2(0,T;\Lambda_p)}$} \\ 
error & rate & $h_{tf}$ & error & rate & $h_{tp}$ & error & rate & error & rate \\  \hline
2.7E-04 &   --   & 1/8   & 1.6E-03 &   --   & 1/5   & 1.6E-02 &   --   & 6.9E-03 &   --   \\ 
1.4E-04 & 1.23 & 1/16  & 3.7E-04 & 2.11 & 1/10  & 5.7E-03 & 1.49 & 2.5E-03 & 1.49 \\
6.7E-05 & 0.96 & 1/32  & 1.3E-04 & 1.45 & 1/20  & 1.2E-03 & 2.31 & 8.5E-04 & 1.52 \\
3.4E-05 & 1.19 & 1/64  & 4.6E-05 & 1.54 & 1/40  & 3.4E-04 & 1.76 & 3.0E-04 & 1.50 \\
1.7E-05 & 1.07 & 1/128 & 1.2E-05 & 1.96 & 1/80  & 1.1E-04 & 1.62 & 1.1E-04 & 1.50 \\
8.4E-06 & 1.15 & 1/256 & 3.6E-06 & 1.70 & 1/160 & 2.2E-05 & 2.34 & 3.7E-05 & 1.54 \\
\hline 
\end{tabular}
\caption{Example 1, errors and convergence rates with $\bP_1 - \bP_1 - \rP_1$ Lagrange multipliers.}\label{table2-example1}
\end{center}
\end{table}

\begin{table}[h]
\begin{center}
\begin{tabular}{c||cc|cc|cc|cc}
\hline
& \multicolumn{2}{|c|}{$\|\be_{\bsi_f}\|_{\ell^2(0,T;\bbX_f)}$}  & \multicolumn{2}{|c|}{$\|\be_{\bu_f}\|_{\ell^2(0,T;\bV_f)}$} & \multicolumn{2}{|c|}{$\|\be_{\bgamma_f}\|_{\ell^2(0,T;\bbQ_f)}$} & \multicolumn{2}{|c}{$\|\be_{p_f}\|_{\ell^2(0,T;\L^2(\Omega_f))}$} \\ 
$h_f$ & error & rate & error & rate & error & rate & error & rate \\  \hline
0.1964 & 2.2E-02 &   --   & 2.7E-02 &   --   & 2.4E-03 &   --   & 6.1E-03 &   --   \\ 
0.0997 & 1.2E-02 & 0.94 & 1.4E-02 & 1.00 & 9.7E-04 & 1.31 & 3.1E-03 & 1.02 \\ 
0.0487 & 5.7E-03 & 0.99 & 6.8E-03 & 0.99 & 4.2E-04 & 1.16 & 1.6E-03 & 0.92 \\
0.0250 & 2.8E-03 & 1.04 & 3.4E-03 & 1.04 & 2.0E-04 & 1.13 & 7.8E-04 & 1.07 \\
0.0136 & 1.4E-03 & 1.14 & 1.7E-03 & 1.15 & 9.4E-05 & 1.23 & 3.9E-04 & 1.15 \\
0.0072 & 7.1E-04 & 1.08 & 8.4E-04 & 1.09 & 4.7E-05 & 1.09 & 2.0E-04 & 1.02 \\ 
\hline 
\end{tabular}
		
\medskip
\begin{tabular}{c||cc|cc|cc|cc|cc}
\hline
& \multicolumn{2}{|c|}{$\|\be_{\bsi_p}\|_{\ell^\infty(0,T;\bbX_p)}$}  & \multicolumn{2}{|c|}{$\|\be_{\bu_s}\|_{\ell^2(0,T;\bV_s)}$} & \multicolumn{2}{|c|}{$\|\be_{\bgamma_p}\|_{\ell^2(0,T;\bbQ_p)}$} & \multicolumn{2}{|c}{$\|\be_{\bu_p}\|_{\ell^2(0,T;\bV_p)}$} & \multicolumn{2}{|c}{$\|\be_{p_p}\|_{\ell^\infty(0,T;\W_p)}$} \\ 
$h_p$ & error & rate & error & rate & error & rate & error & rate & error & rate \\  \hline
0.2828 & 2.7E-01 &   --   & 4.3E-02 &   --   & 3.4E-02 &   --   & 1.0E-01 &   --   & 7.5E-02 &   --   \\ 
0.1646 & 1.4E-01 & 1.27 & 2.2E-02 & 1.23 & 9.4E-03 & 2.39 & 5.2E-02 & 1.26 & 3.8E-02 & 1.25 \\ 
0.0779 & 6.7E-02 & 0.97 & 1.1E-02 & 0.96 & 2.2E-03 & 1.96 & 2.5E-02 & 1.00 & 1.9E-02 & 0.93 \\
0.0434 & 3.4E-02 & 1.17 & 5.4E-03 & 1.19 & 5.8E-04 & 2.25 & 1.2E-02 & 1.24 & 9.4E-03 & 1.22 \\
0.0227 & 1.7E-02 & 1.06 & 2.7E-03 & 1.07 & 2.0E-04 & 1.67 & 5.9E-03 & 1.08 & 4.7E-03 & 1.07 \\
0.0124 & 8.4E-03 & 1.15 & 1.4E-03 & 1.15 & 8.1E-05 & 1.48 & 2.9E-03 & 1.15 & 2.4E-03 & 1.14 \\ 
\hline 
\end{tabular}
	
\medskip

\begin{tabular}{cc||c||cc||c||cc|cc}
\hline
\multicolumn{2}{c||}{$\|\be_{\bbeta_p}\|_{\ell^2(0,T;\bL^2(\Omega_p))}$} &
& \multicolumn{2}{|c||}{$\|\be_{\bvarphi}\|_{\ell^2(0,T;\bL^2(\Gamma_{fp}))}$} & & \multicolumn{2}{|c|}{$\|\be_{\btheta}\|_{\ell^2(0,T;\bL^2(\Gamma_{fp}))}$} & \multicolumn{2}{|c}{$\|\be_{\lambda}\|_{\ell^2(0,T;\L^2(\Gamma_{fp}))}$} \\ 
error & rate & $h_{tf}$ & error & rate & $h_{tp}$ & error & rate & error & rate \\  \hline
2.7E-04 &   --   & 1/8   & 4.1E-04 &   --   & 1/5   & 7.9E-03 &   --   & 1.1E-03 &   --   \\ 
1.4E-04 & 1.23 & 1/16  & 2.0E-04 & 1.04 & 1/10  & 2.9E-03 & 1.46 & 3.1E-04 & 1.87 \\ 
6.7E-05 & 0.96 & 1/32  & 2.4E-05 & 3.07 & 1/20  & 5.7E-04 & 2.34 & 7.7E-05 & 2.01 \\
3.4E-05 & 1.19 & 1/64  & 6.4E-06 & 1.89 & 1/40  & 1.5E-04 & 1.89 & 1.9E-05 & 2.00 \\
1.7E-05 & 1.07 & 1/128 & 1.6E-06 & 1.97 & 1/80  & 3.8E-05 & 2.01 & 4.9E-06 & 1.98 \\
8.4E-06 & 1.15 & 1/256 & 4.0E-07 & 2.02 & 1/160 & 9.0E-06 & 2.09 & 1.2E-06 & 2.09 \\ 
\hline 
\end{tabular}
\caption{Example 1, errors and convergence rates with 
  $\bP^\dc_1 - \bP^\dc_1 - \rP^\dc_1$ Lagrange multipliers.} \label{table1-example1}
\end{center}
\end{table}

\subsection{Example 2: coupled surface and subsurface flows}

In this example, we simulate coupling of surface and subsurface flows, which could be used to describe the interaction between a river and an aquifer. We consider the domain $\Omega = (0,2)\times(-1,1)$.  We associate the upper half with the river flow modeled by Stokes equations, while the lower half represents the flow in the aquifer governed by the Biot system. The appropriate interface conditions are enforced along the interface $y = 0$. In this example we focus on the qualitative behavior of the solution and use unit physical parameters:
\begin{equation*}
\mu = 1,\quad
\alpha_p = 1,\quad
\lambda_p = 1,\quad
\mu_p = 1,\quad
s_0 = 1,\quad
\bK = \bI,\quad
\alpha_{\BJS} = 1.
\end{equation*}
The body forces terms and external source are set to zero, as well as the initial conditions. 
The flow is driven through a parabolic fluid velocity on the left boundary of the fluid region with boundary conditions specified as follows:
\begin{equation*}
\begin{array}{lll}
\ds \bu_f = (-40y(y-1) \ \  0)^{t} & \mbox{on} & \Gamma_{f,left}, \\ [1ex]
\ds \bu_f = \0 & \mbox{on} & \Gamma_{f,top}, \\ [1ex]
\ds \bsi_f\bn_f = \0 & \mbox{on} & \Gamma_{f,right}, \\ [1ex]
\ds p_p = 0 \qan \bsi_p\bn_p = \0  & \mbox{on} & \Gamma_{p,bottom}, \\ [1ex]
\ds \bu_p\cdot\bn_p = 0 \qan \bu_{s} = \0 & \mbox{on} & \Gamma_{p,left}\cup \Gamma_{p,right}.
\end{array}
\end{equation*}
The simulation is run for a total time $T=3$ with a time step $\Delta t = 0.06$. 
The computed solution is presented in Figure~\ref{fig:example2}. 
\begin{figure}[h]
\begin{center}
\includegraphics[width=0.32\textwidth]{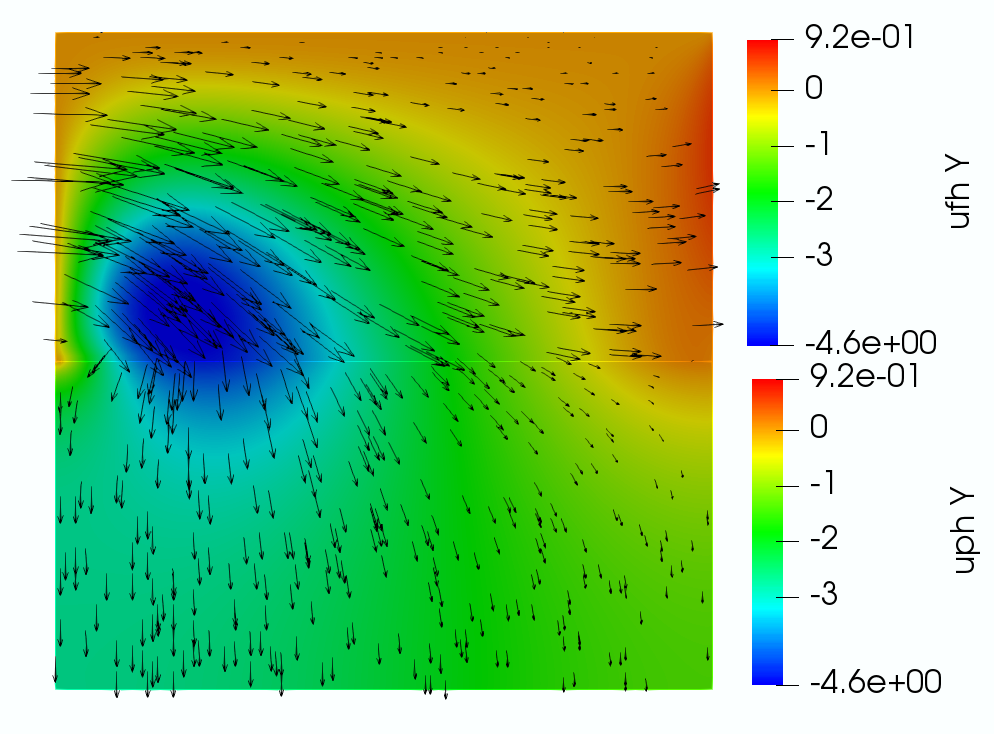}
\includegraphics[width=0.32\textwidth]{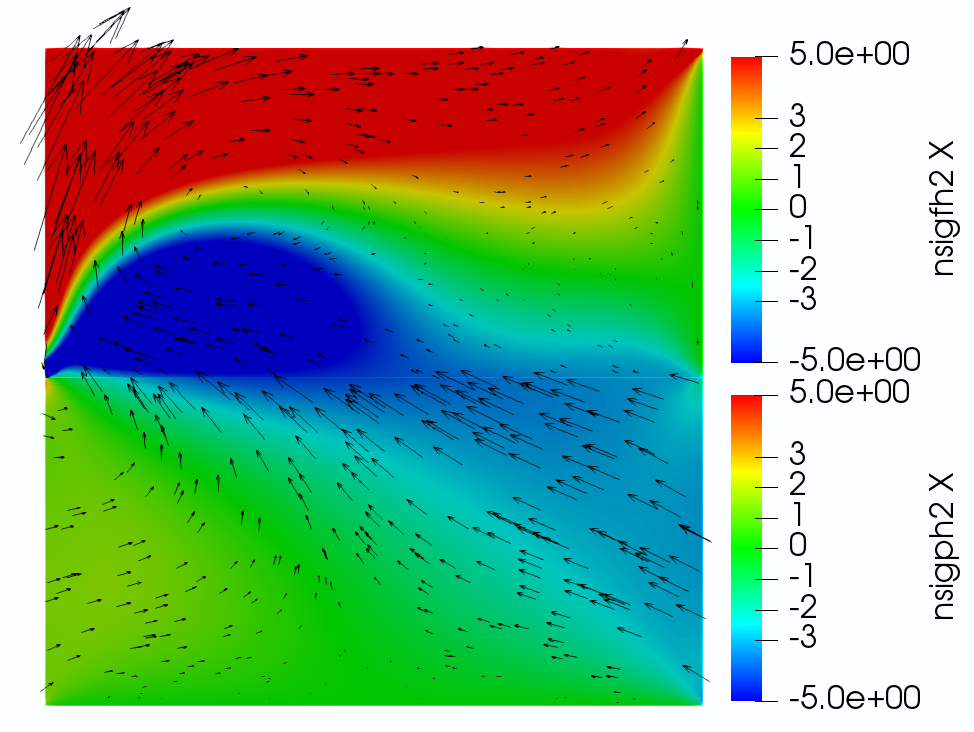}
\includegraphics[width=0.32\textwidth]{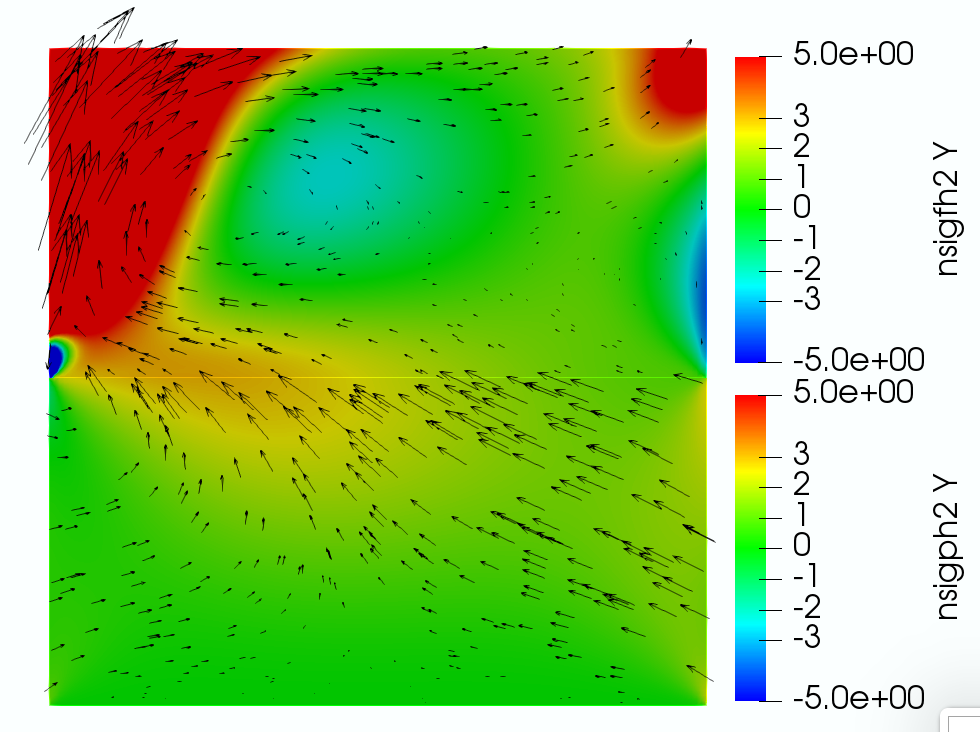}

\includegraphics[width=0.32\textwidth]{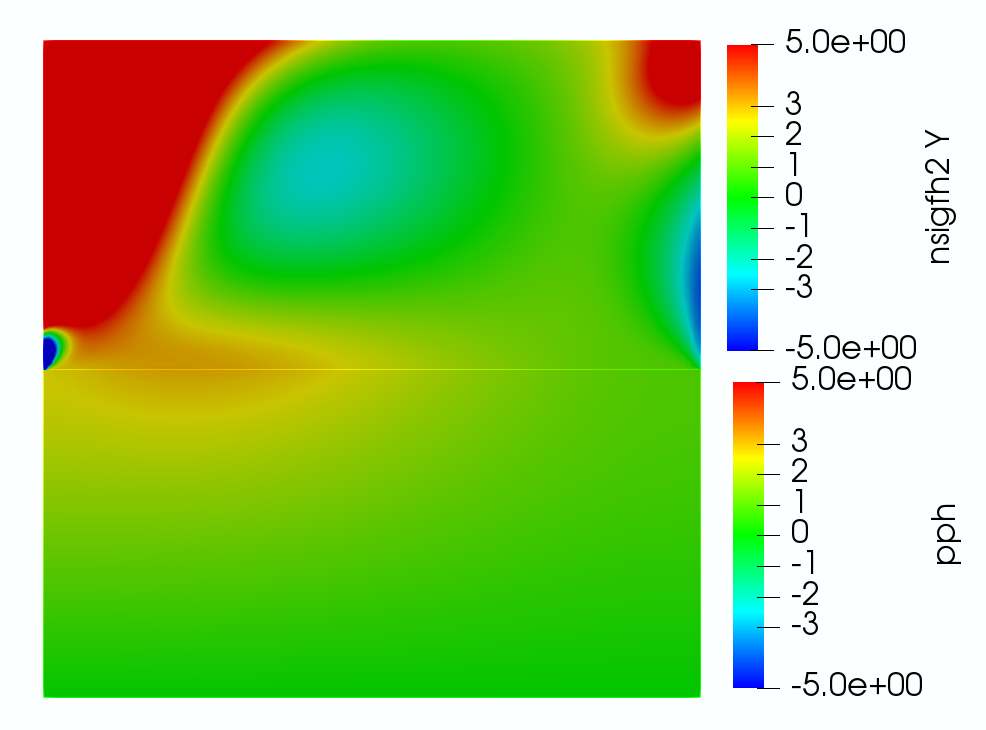}
\includegraphics[width=0.32\textwidth]{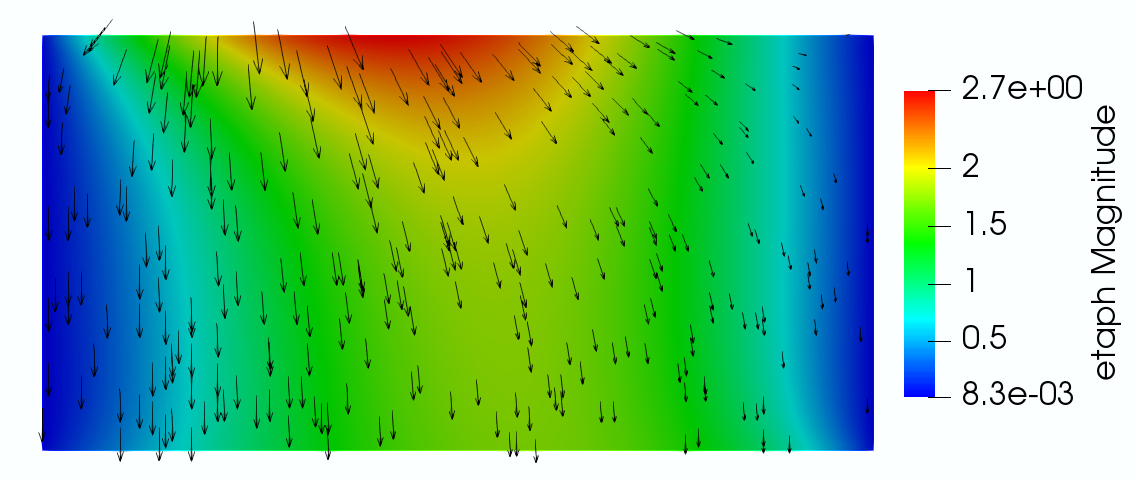}
\caption{Example 2, computed solution at $T=3$. Top left: 
  velocities $\bu_{fh}$ and $\bu_{ph}$ (arrows), $\bu_{fh,2}$ and $\bu_{ph,2}$ (color). Top middle and right: negative stresses $-(\bsi_{fh,12},\bsi_{fh,22})^\rt$ and
  $-(\bsi_{ph,12},\bsi_{ph,22})^\rt$ (arrows);
  middle: $-\bsi_{fh,12}$ and $-\bsi_{ph,12}$ (color); right:
  $-\bsi_{fh,22}$ and $-\bsi_{ph,22}$ (color).
Bottom left: negative Stokes stress $-\bsi_{fh,22}$ and Darcy pressure $p_{ph}$. Bottom right: displacement $\bbeta_{ph}$ (arrows) and its magnitude (color).} \label{fig:example2}
\end{center}
\end{figure}
From the velocity plot (top left), we see that the flow in the Stokes region is moving primarily from left to right, driven by the parabolic inflow condition,
with some of the fluid percolating downward into the poroelastic medium due to the zero pressure at the bottom, which simulates gravity. The mass conservation
$\bu_{f}\cdot \bn_f + \left(\partial_t\bbeta_p+\bu_{p}\right)\cdot \bn_p = 0$ on the interface with $\bn_p = (0,1)^\rt$ indicates the continuity of
the second components of the fluid velocity and Darcy velocity when
the displacement becomes steady, which is observed from the color plot of
the vertical velocity. The stress plots (top middle and right) illustrate the
ability of our fully mixed formulation to compute accurate $\bbH(\bdiv)$ stresses in both the fluid and poroelastic regions, without the need for numerical differentiation. In addition, the conservation of momentum $\bsi_f \bn_f
+\bsi_p \bn_p = \0$ and balance of normal stress $(\bsi_f \bn_f)\cdot
\bn_f = -p_p$ imply that $\bsi_{f,12}=\bsi_{p,12}$,
$\bsi_{f,22}=\bsi_{p,22}$ and $-\bsi_{f,22}=p_p$ on the
interface. These conditions are verified from the top middle and right color plots, as well as the bottom left plot. Furthermore, the arrows in the stress plots are formed by the second columns of the stresses, whose traces on the interface are $\bsi_f\bn_f$ and $-\bsi_p \bn_p$, respectively. For visualization purpose, the
Stokes stress is scaled by a factor of $1/5$ compared to the poroelastic stress,
due to large difference in their magnitudes away from the interface. Nevertheless, the continuity of the vector field across the interface is evident, consistent with the conservation of momentum condition $\bsi_f \bn_f + \bsi_p \bn_p = \0$. 
The overall qualitative behavior of the computed stresses is consistent with the
specified boundary and interface conditions. In particular, we observe
large fluid stress along the top boundary due to the no slip condition, as well as along the interface due to the slip with friction condition. The singularity near
the lower left corner of the Stokes region is due to the mismatch in boundary
conditions between the fluid and poroelastic regions. Finally, the last plot shows
that the inflow from the Stokes region causes deformation of the poroelastic medium.

\subsection{Example 3: irregularly shaped fluid-filled cavity}

This example features highly irregularly shaped cavity motivated by
modeling flow through vuggy or naturally fractured reservoirs or
aquifers. It uses physical units and realistic parameter values taken
from the reservoir engineering literature \cite{gr2015}:
\begin{equation*}
\begin{array}{c}
\mu = 10^{-6} \text{ kPa s},\quad
\alpha_p = 1,\quad
\lambda_p = 5/18 \times 10^7 \text{ kPa},\quad
\mu_p = 5/12 \times 10^7 \text{ kPa},\\[2ex]
s_0 = 6.89 \times 10^{-2} \text{ kPa}^{-1},\quad
\bK = 10^{-8}\times \bI \text{ m}^2,\quad
\alpha_{\BJS} = 1.
\end{array}
\end{equation*}
We emphasize that the problem features very small permeability and storativity, as well as large Lam\'e parameters. These are parameter regimes that are known to lead
locking in modeling of the Biot system of poroelasticity \cite{Lee-Biot-five-field,Yi-Biot-locking}.
The domain is $\Omega = (0,1)\times(0,1)$, with a large fluid-filled cavity in the interior. The body forces and external sources are set to zero. The flow is driven from left to right via a pressure drop of 1 kPa, with boundary conditions specified as follows:
\begin{equation*}
\begin{array}{l}
  \bsi_{f} \bn_f\cdot\bn_f = 1000, \quad \bu_f\cdot\bt_f = 0 \qon \Gamma_{f, right}, \\[2ex]
p_p = 1001 \qon \Gamma_{p,left}, \quad p_p= 1000 \qon \Gamma_{p,right} 
\qan \bu_p\cdot \bn_p = 0  \qon \Gamma_{p,top}\cup \Gamma_{p,bottom},\\[2ex]
\bsi_p \, \bn_p=-\alpha_p \, p_p \, \bn_p \qon \Gamma_{p,left} \cup \Gamma_{p,right} \qan
\bu_s = \0  \qon \Gamma_{p,top}\cup \Gamma_{p,bottom}.
\end{array}
\end{equation*}
The total simulation time is $T=10 \,$s with a time step of size $\Delta t = 0.05 \,$s. To avoid inconsistency between the initial and boundary conditions for $p_p$, we start with $p_p=1000$ on $\Gamma_{p,left}$ and gradually increase it to reach $p_p=1001$ at $t = 0.5 \,$s. Similar adjustment is done for $\bsi_p \bn_p$.

The simulation results at the final time $T=10 \,$s are shown in
Figure~\ref{fig:example3}. In the top plots, we present the Darcy
pressure and Darcy velocity vector, the displacement vector with its
magnitude, and the first row of the poroelastic stress with its
magnitude. Since the pressure variation is small relative to its
value, for visualization purpose we plot its difference from the
reference pressure, $p_p - 1000$. The Darcy velocity and the pressure
drop are largest in the region between the left inflow boundary and
the cavity. The displacement is largest around the cavity, due to the
large fluid velocity within the cavity and the slip with friction
interface condition. The poroelastic stress exhibits singularities
near some of the sharp tips of the cavity. The bottom plots show the
fluid pressure and velocity vector, the velocity vector with its
magnitude, and the first row of the fluid stress with its
magnitude. Similarly to the Darcy pressure, we plot $p_f - 1000$.  A channel-like flow profile is clearly visible
within the cavity, with the largest velocity along a central path away
from the cavity walls. The fluid pressure is decreasing
from left to right along the central path of the cavity.  Consistent
with the poroelastic stress, the fluid stress near the tips 
of the cavity is relatively larger. We emphasize that, despite
the locking regime of the parameters, the computed solution is free of
locking and spurious oscillations. This example illustrates the ability of our method to handle computationally challenging problems with physically realistic parameters in poroelastic locking regimes.

\begin{figure}[ht]
\begin{center}
\includegraphics[width=0.32\textwidth]{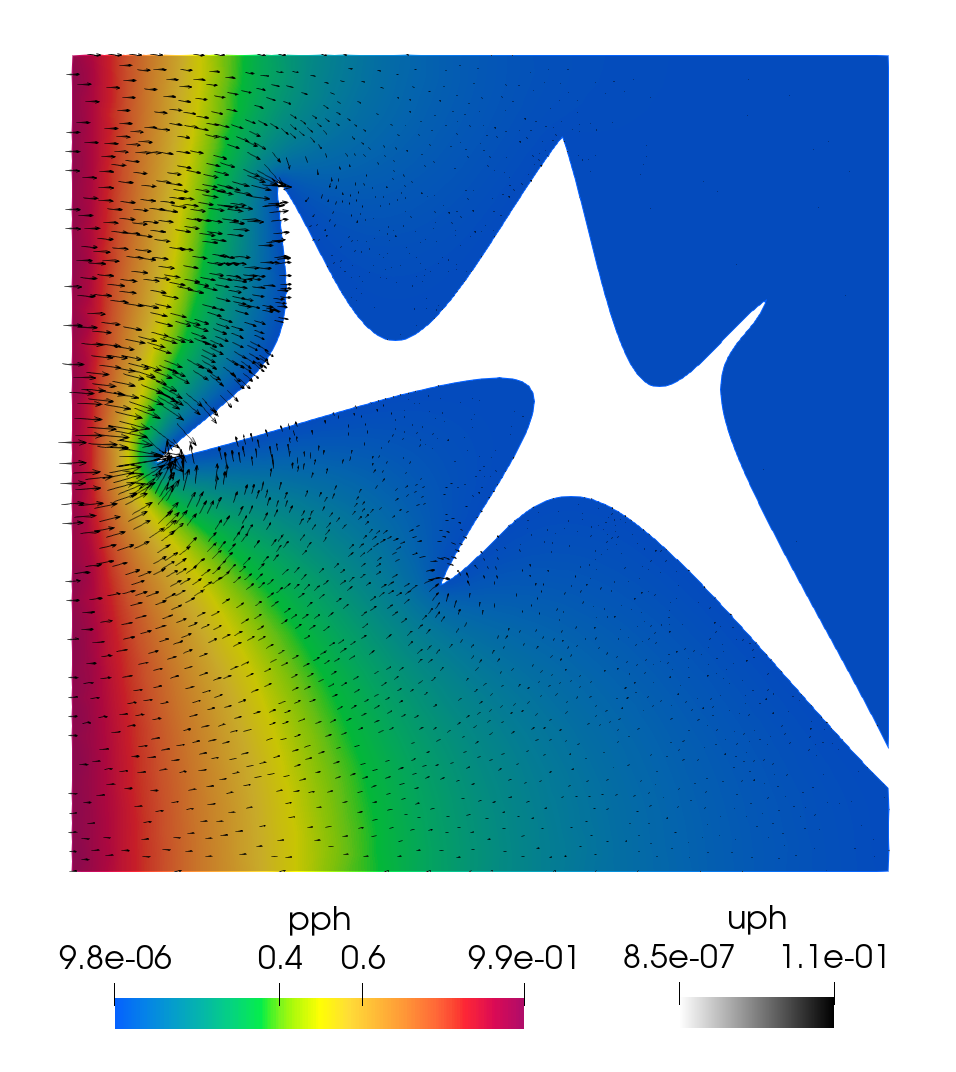}
\includegraphics[width=0.32\textwidth]{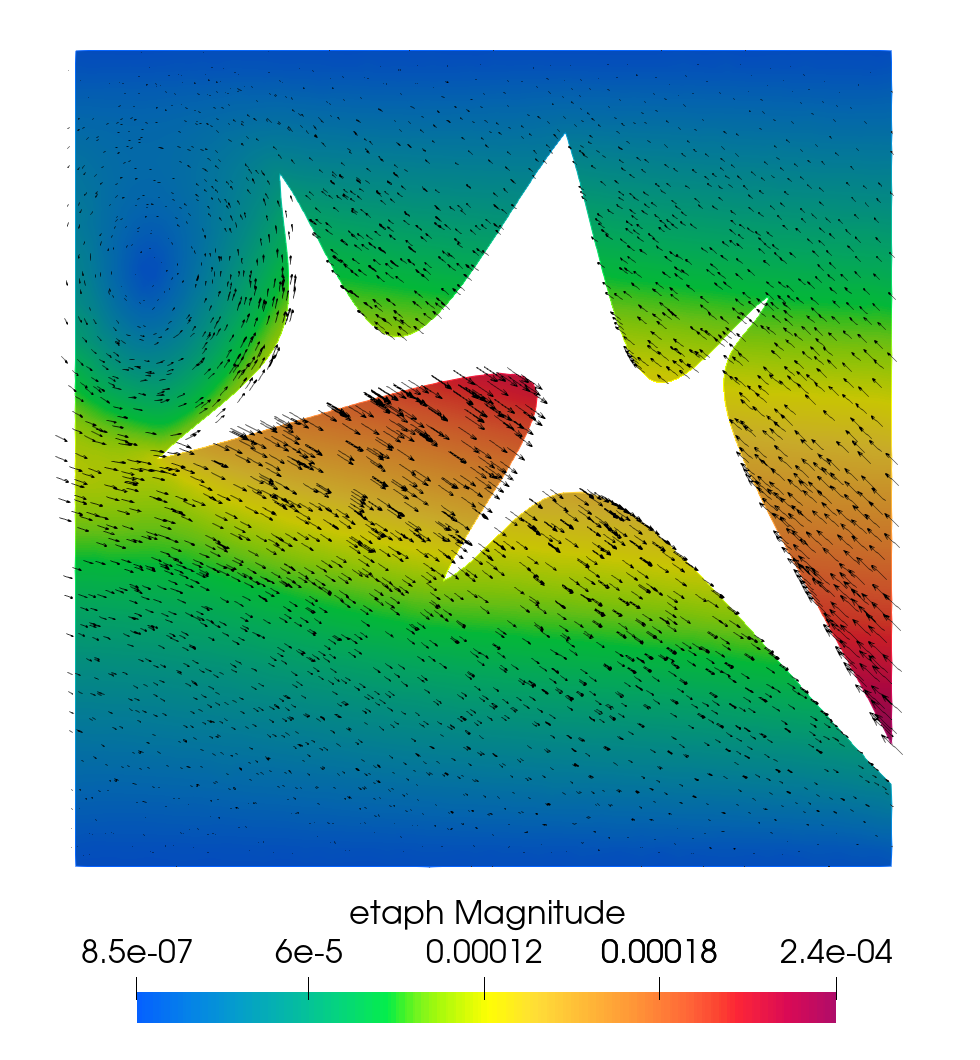}
\includegraphics[width=0.32\textwidth]{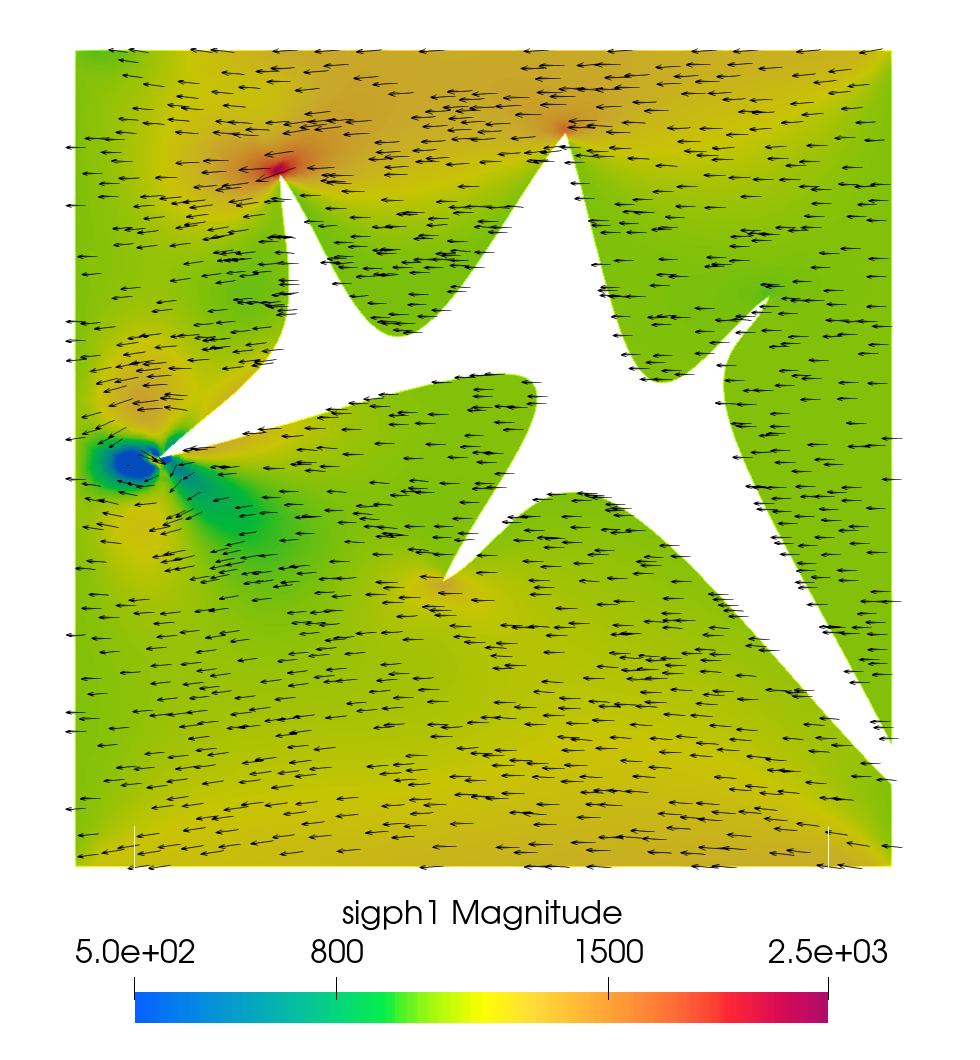}

\includegraphics[width=0.32\textwidth]{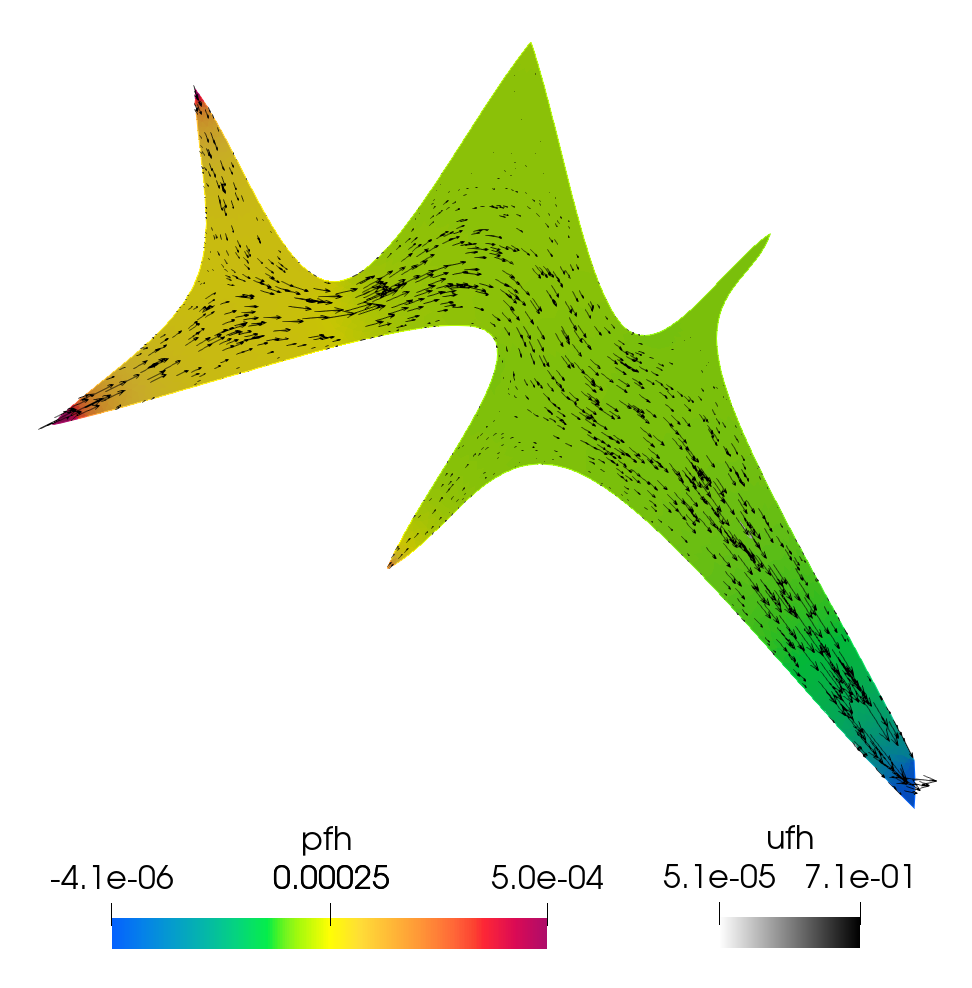}
\includegraphics[width=0.32\textwidth]{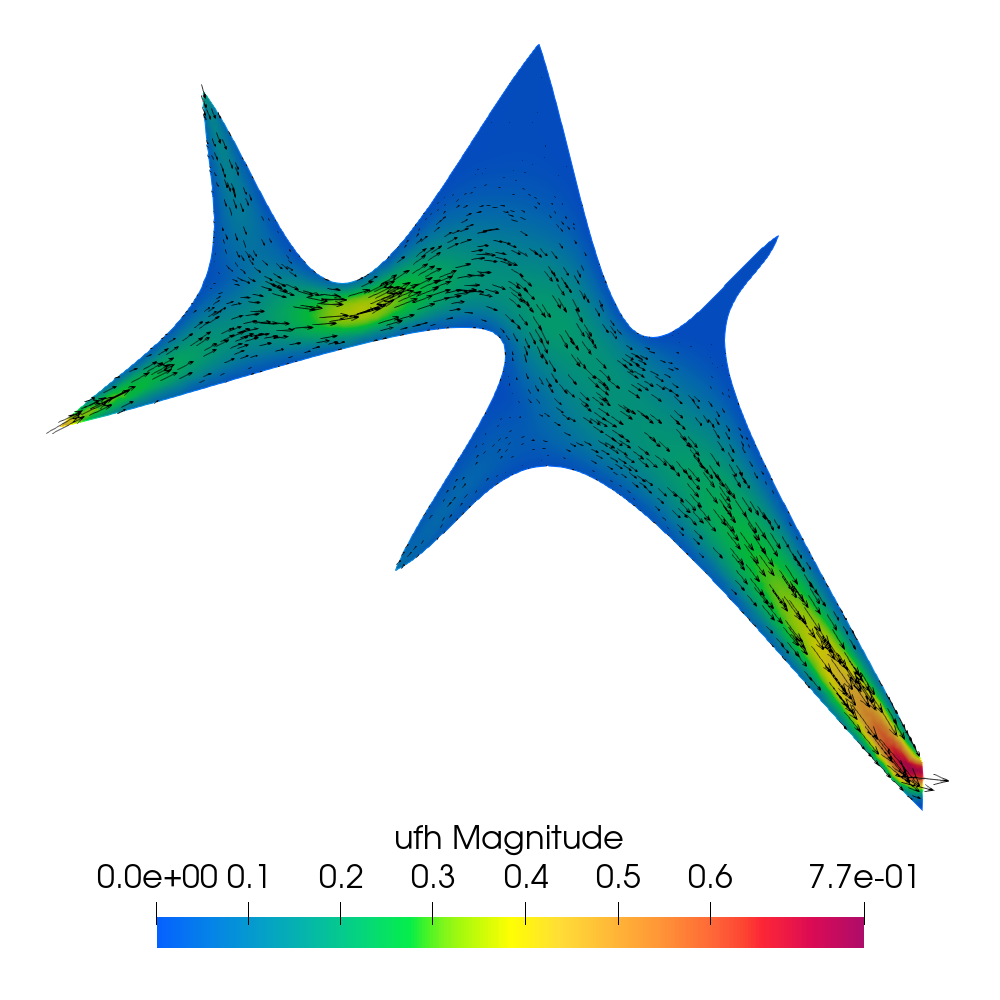}
\includegraphics[width=0.32\textwidth]{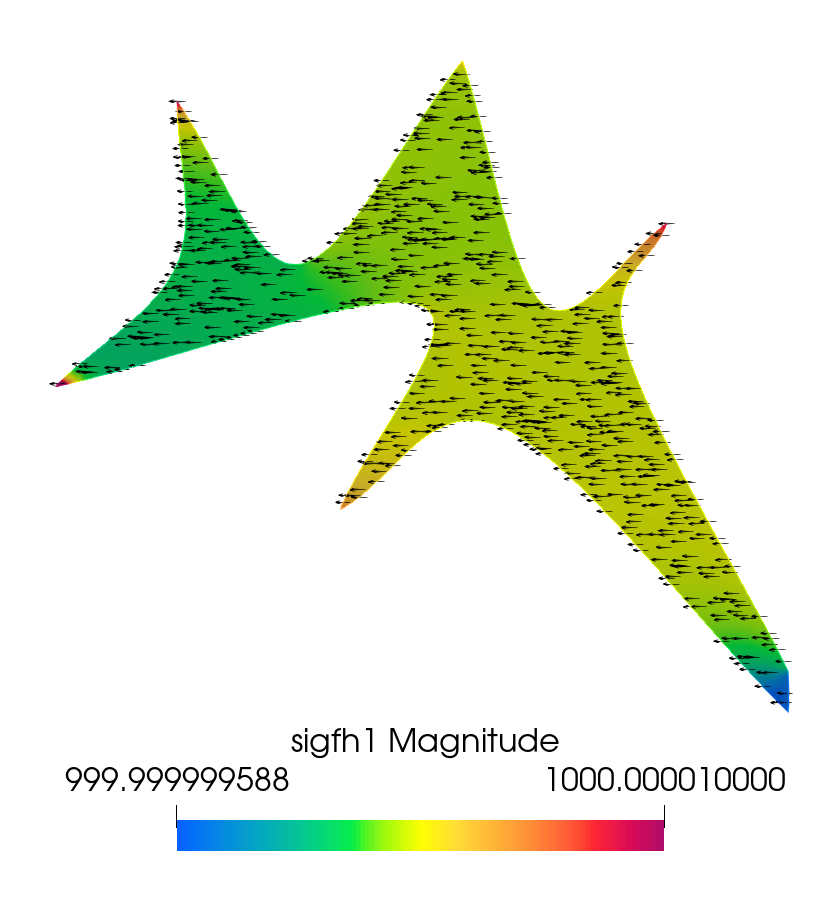}
\caption{Example 3, computed solution at $T = 10\,$s. Top left: Darcy velocity (arrows) and pressure (color). Top middle: displacement (arrows) and its magnitude (color). Top right: first row of the poroelastic stress tensor (arrows) and its magnitude (color). Bottom left: Stokes velocity (arrows) and pressure (color). Bottom middle: Stokes velocity (arrows) and its magnitude (color). Bottom right: first row of the Stokes stress (arrows) and its magnitude (color).} \label{fig:example3}
\end{center}
\end{figure}

\section{Conclusions}
In this paper we present and analyze the first, to the best of our
knowledge, fully dual mixed formulation of the quasi-static
Stokes-Biot model, and its mixed finite element approximation, using a
velocity-pressure Darcy formulation, a weakly symmetric
stress-displacement-rotation elasticity formulation, and a weakly
symmetric stress-velocity-vorticity Stokes formulation. Essential-type
interface conditions are imposed via suitable Lagrange
multipliers. The numerical method features accurate stresses and Darcy
velocity with local mass and momentum conservation. Furthermore, a new
multipoint stress-flux mixed finite element method is developed that
allows for local elimination of the Darcy velocity, the fluid and
poroelastic stresses, the vorticity, and the rotation, resulting in a
reduced positive definite cell-centered pressure-velocities-traces
system. The theoretical results are complemented by a series of numerical experiments that illustrate the convergence rates for all variables in their natural norms, as well as the ability of the method to simulate physically realistic problems motivated by applications to coupled surface-subsurface flows and flows in fractured poroelastic media with parameter values in locking regimes.

\bibliographystyle{abbrv}
\bibliography{caucao-li-yotov-1}

\begin{thebibliography}{10}

\bibitem{adgm2019}
J.~A. Almonacid, H.~S. D{\'i}az, G.~N. Gatica, and A.~M{\'a}rquez.
\newblock A fully-mixed finite element method for the
  {D}arcy--{F}orchheimer/{S}tokes coupled problem.
\newblock {\em IMA J. Numer. Anal.}, 40(2):1454--1502, 2020.

\bibitem{at1979}
M.~Amara and J.~M. Thomas.
\newblock Equilibrium finite elements for the linear elastic problem.
\newblock {\em Numer. Math.}, 33(4):367--383, 1979.

\bibitem{aeny2019}
I.~Ambartsumyan, V.~J. Ervin, T.~Nguyen, and I.~Yotov.
\newblock A nonlinear {S}tokes-{B}iot model for the interaction of a
  non-{N}ewtonian fluid with poroelastic media.
\newblock {\em ESAIM Math. Model. Numer. Anal.}, 53(6):1915--1955, 2019.

\bibitem{fpsi-transport}
I.~Ambartsumyan, E.~Khattatov, T.~Nguyen, and I.~Yotov.
\newblock Flow and transport in fractured poroelastic media.
\newblock {\em GEM Int. J. Geomath.}, 10(1):1--34, 2019.

\bibitem{msmfe-simpl}
I.~Ambartsumyan, E.~Khattatov, J.~M. Nordbotten, and I.~Yotov.
\newblock A multipoint stress mixed finite element method for elasticity on
  simplicial grids.
\newblock {\em SIAM J. Numer. Anal.}, 58(1):630--656, 2020.

\bibitem{msmfe-quads}
I.~Ambartsumyan, E.~Khattatov, J.~M. Nordbotten, and I.~Yotov.
\newblock A multipoint stress mixed finite element method for elasticity on
  quadrilateral grids.
\newblock {\em Numer. Methods Partial Differential Equations},
  37(3):1886--1915, 2021.

\bibitem{msfmfe-Biot}
I.~Ambartsumyan, E.~Khattatov, and I.~Yotov.
\newblock A coupled multipoint stress--multipoint flux mixed finite element
  method for the {B}iot system of poroelasticity.
\newblock {\em Comput. Methods Appl. Mech. Engrg.}, 372:113407, 2020.

\bibitem{akyz2018}
I.~Ambartsumyan, E.~Khattatov, I.~Yotov, and P.~Zunino.
\newblock A {L}agrange multiplier method for a {S}tokes-{B}iot
  fluid-poroelastic structure interaction model.
\newblock {\em Numer. Math.}, 140(2):513--553, 2018.

\bibitem{abd1984}
D.~N. Arnold, F.~Brezzi, and J.~Douglas.
\newblock {PEERS}: a new mixed finite element for plane elasticity.
\newblock {\em Japan J. Appl. Math.}, 1(2):347--367, 1984.

\bibitem{afw2007}
D.~N. Arnold, R.~S. Falk, and R.~Winter.
\newblock Mixed finite element methods for linear elasticity with weakly
  imposed symmetry.
\newblock {\em Math. Comp.}, 76(260):1699--1723, 2007.

\bibitem{awanou2013}
G.~Awanou.
\newblock Rectangular mixed elements for elasticity with weakly imposed
  symmetry condition.
\newblock {\em Adv. Comput. Math.}, 38(2):351--367, 2013.

\bibitem{bqq2009}
S.~Badia, A.~Quaini, and A.~Quarteroni.
\newblock Coupling {B}iot and {N}avier-{S}tokes equations for modelling
  fluid-poroelastic media interaction.
\newblock {\em J. Comput. Phys.}, 228(21):7986--8014, 2009.

\bibitem{Bergkamp-etal}
E.~A. Bergkamp, C.~V. Verhoosel, J.~J.~C. Remmers, and D.~M.~J. Smeulders.
\newblock A staggered finite element procedure for the coupled {S}tokes-{B}iot
  system with fluid entry resistance.
\newblock {\em Comput. Geosci.}, 24(4):1497--1522, 2020.

\bibitem{b1941}
M.~Biot.
\newblock General theory of three-dimensional consolidation.
\newblock {\em J. Appl. Phys.}, 12:155--164, 1941.

\bibitem{brezzi2008mixed}
D.~Boffi, F.~Brezzi, L.~F. Demkowicz, R.~G. Dur\'an, R.~S. Falk, and M.~Fortin.
\newblock {\em Mixed finite elements, compatibility conditions, and
  applications}, volume 1939 of {\em Lecture Notes in Mathematics}.
\newblock Springer-Verlag, Berlin; Fondazione C.I.M.E., Florence, 2008.

\bibitem{BBF-reduced}
D.~Boffi, F.~Brezzi, and M.~Fortin.
\newblock Reduced symmetry elements in linear elasticity.
\newblock {\em Commun. Pure Appl. Anal.}, 8(1):95--121, 2009.

\bibitem{bdm1985}
F.~Brezzi, J.~Douglas, and L.~D. Marini.
\newblock Two families of mixed finite elements for second order elliptic
  problems.
\newblock {\em Numer. Math.}, 47(2):217--235, 1985.

\bibitem{Brezzi-Fortin}
F.~Brezzi and M.~Fortin.
\newblock {\em Mixed and Hybrid Finite Element Methods}.
\newblock Springer Series in Computational Mathematics, 15. Springer-Verlag,
  New York, 1991.

\bibitem{Brezzi.F;Fortin.M;Marini.L2006}
F.~Brezzi, M.~Fortin, and L.~D. Marini.
\newblock Error analysis of piecewise constant pressure approximations of
  {D}arcy's law.
\newblock {\em Comput. Methods Appl. Mech. Eng.}, 195:1547--1559, 2006.

\bibitem{byzz2015}
M.~Bukac, I.~Yotov, R.~Zakerzadeh, and P.~Zunino.
\newblock Partitioning strategies for the interaction of a fluid with a
  poroelastic material based on a {N}itsche's coupling approach.
\newblock {\em Comput. Methods Appl. Mech. Engrg.}, 292:138--170, 2015.

\bibitem{byz2015}
M.~Bukac, I.~Yotov, and P.~Zunino.
\newblock An operator splitting approach for the interaction between a fluid
  and a multilayered poroelastic structure.
\newblock {\em Numer. Methods Partial Differential Equations},
  31(4):1054--1100, 2015.

\bibitem{Buk-Yot-Zun-fracture}
M.~Bukac, I.~Yotov, and P.~Zunino.
\newblock Dimensional model reduction for flow through fractures in poroelastic
  media.
\newblock {\em ESAIM Math. Model. Numer. Anal.}, 51(4):1429--1471, 2017.

\bibitem{Cesm-Chid}
A.~Cesmelioglu and P.~Chidyagwai.
\newblock Numerical analysis of the coupling of free fluid with a poroelastic
  material.
\newblock {\em Numer. Methods Partial Differential Equations}, 36(3):463--494,
  2020.

\bibitem{Cesm-etal-optim}
A.~Cesmelioglu, H.~Lee, A.~Quaini, K.~Wang, and S.-Y. Yi.
\newblock Optimization-based decoupling algorithms for a fluid-poroelastic
  system.
\newblock In {\em Topics in numerical partial differential equations and
  scientific computing}, volume 160 of {\em IMA Vol. Math. Appl.}, pages
  79--98. Springer, New York, 2016.

\bibitem{cesm2017}
S.~Cesmelioglu.
\newblock Analysis of the coupled {N}avier-{S}tokes/{B}iot problem.
\newblock {\em J. Math. Anal. Appl.}, 456(2):970--991, 2017.

\bibitem{ciar1978}
P.~Ciarlet.
\newblock {\em The Finite Element Method for Elliptic Problems}.
\newblock Studies in Mathematics and its Applications, Vol. 4. North-Holland
  Publishing Co., Amsterdam-New York-Oxford, 1978.

\bibitem{cgg2010}
B.~Cockburn, J.~Gopalakrishnan, and J.~Guzm\'an.
\newblock A new elasticity element made for enforcing weak stress symmetry.
\newblock {\em Math. Comp.}, 79(271):1331--1349, 2010.

\bibitem{umfpack}
T.~Davis.
\newblock Algorithm 832: {UMFPACK} {V}4.3 - an unsymmetric-pattern multifrontal
  method.
\newblock {\em ACM Trans. Math. Software}, 30(2):196--199, 2004.

\bibitem{ejs2009}
V.~J. Ervin, E.~W. Jenkins, and S.~Sun.
\newblock Coupled generalized nonlinear {S}tokes flow with flow through a
  porous medium.
\newblock {\em SIAM J. Numer. Anal.}, 47(2):929--952, 2009.

\bibitem{FarFor}
M.~Farhloul and M.~Fortin.
\newblock Dual hybrid methods for the elasticity and the {S}tokes problems: a
  unified approach.
\newblock {\em Numer. Math.}, 76(4):419--440, 1997.

\bibitem{gs2007}
J.~Galvis and M.~Sarkis.
\newblock Non-matching mortar discretization analysis for the coupling
  {S}tokes-{D}arcy equations.
\newblock {\em Electron. Trans. Numer. Anal.}, 26:350--384, 2007.

\bibitem{Gatica}
G.~N. Gatica.
\newblock {\em A {S}imple {I}ntroduction to the {M}ixed {F}inite {E}lement
  {M}ethod. {T}heory and {A}pplications}.
\newblock Springer Briefs in Mathematics. Springer, Cham, 2014.

\bibitem{ghm2003}
G.~N. Gatica, N.~Heuer, and S.~Meddahi.
\newblock On the numerical analysis of nonlinear twofold saddle point problems.
\newblock {\em IMA J. Numer. Anal.}, 23(2):301--330, 2003.

\bibitem{gmor2014}
G.~N. Gatica, A.~M\'arquez, R.~Oyarz\'ua, and R.~Rebolledo.
\newblock Analysis of an augmented fully-mixed approach for the coupling of
  quasi-{N}ewtonian fluids and porous media.
\newblock {\em Comput. Methods Appl. Mech. Engrg.}, 270:76--112, 2014.

\bibitem{gos2011}
G.~N. Gatica, R.~Oyarz\'ua, and F.~J. Sayas.
\newblock Analysis of fully-mixed finite element methods for the
  {S}tokes--{D}arcy coupled problem.
\newblock {\em Math. Comp.}, 80(276):1911--1948, 2011.

\bibitem{gos2012}
G.~N. Gatica, R.~Oyarz\'ua, and F.~J. Sayas.
\newblock A twofold saddle point approach for the coupling of fluid flow with
  nonlinear porous media flow.
\newblock {\em IMA J. Numer. Anal.}, 32(3):845--887, 2012.

\bibitem{gr2015}
V.~Girault, M.~F. Wheeler, B.~Ganis, and M.~E. Mear.
\newblock A lubrication fracture model in a poroelastic medium.
\newblock {\em Math. Models Methods Appl. Sci.}, 25(4):587--645, 2015.

\bibitem{freefem}
F.~Hecht.
\newblock New development in {F}ree{F}em++.
\newblock {\em J. Numer. Math.}, 20(3-4):251--265, 2012.

\bibitem{Horn-Johnson}
R.~Horn and C.~R. Johnson.
\newblock {\em Matrix analysis}.
\newblock Corrected reprint of the 1985 original. Cambridge University Press,
  Cambridge, 1990.

\bibitem{iwy2010}
R.~Ingram, M.~F. Wheeler, and I.~Yotov.
\newblock A multipoint flux mixed finite element method on hexahedra.
\newblock {\em SIAM J. Math. Anal.}, 48(4):1281--1312, 2010.

\bibitem{keilegavlen2017finite}
E.~Keilegavlen and J.~M. Nordbotten.
\newblock Finite volume methods for elasticity with weak symmetry.
\newblock {\em Int. J. Numer. Meth. Engng.}, 112(8):939--962, 2017.

\bibitem{Khat-Yot}
E.~Khattatov and I.~Yotov.
\newblock Domain decomposition and multiscale mortar mixed finite element
  methods for linear elasticity with weak stress symmetry.
\newblock {\em ESAIM Math. Model. Numer. Anal.}, 53(6):2081--2108, 2019.

\bibitem{Klausen-Winther-2006a}
R.~A. Klausen and R.~Winther.
\newblock Robust convergence of multi point flux approximation on rough grids.
\newblock {\em Numer. Math.}, 104(3):317--337, 2006.

\bibitem{Kunwar-etal}
H.~Kunwar, H.~Lee, and K.~Seelman.
\newblock Second-order time discretization for a coupled quasi-{N}ewtonian
  fluid-poroelastic system.
\newblock {\em Internat. J. Numer. Methods Fluids}, 92(7):687--702, 2020.

\bibitem{Lee-Biot-five-field}
J.~J. Lee.
\newblock Robust error analysis of coupled mixed methods for {B}iot's
  consolidation model.
\newblock {\em J. Sci. Comput.}, 69(2):610--632, 2016.

\bibitem{lee2016towards}
J.~J. Lee.
\newblock Towards a unified analysis of mixed methods for elasticity with
  weakly symmetric stress.
\newblock {\em Adv. Comput. Math.}, 42(2):361--376, 2016.

\bibitem{fpsi-mixed-elast}
T.~Li and I.~Yotov.
\newblock A mixed elasticity formulation for fluid--poroelastic structure
  interaction.
\newblock arXiv:2011.00132v2 [math.NA].

\bibitem{Nedelec86}
J.-C. N{\'e}d{\'e}lec.
\newblock A new family of mixed finite elements in {${\bf R}^3$}.
\newblock {\em Numer. Math.}, 50(1):57--81, 1986.

\bibitem{Jan-IJNME}
J.~M. Nordbotten.
\newblock Cell-centered finite volume discretizations for deformable porous
  media.
\newblock {\em Internat. J. Numer. Methods Engrg.}, 100(6):399--418, 2014.

\bibitem{nordbotten2015convergence}
J.~M. Nordbotten.
\newblock Convergence of a cell-centered finite volume discretization for
  linear elasticity.
\newblock {\em SIAM J. Numer. Anal.}, 53(6):2605--2625, 2015.

\bibitem{Jan-SINUM-Biot}
J.~M. Nordbotten.
\newblock Stable cell-centered finite volume discretization for {B}iot
  equations.
\newblock {\em SIAM J. Numer. Anal.}, 54(2):942--968, 2016.

\bibitem{Showalter}
R.~E. Showalter.
\newblock {\em Monotone {O}perators in {B}anach {S}pace and {N}onlinear
  {P}artial {D}ifferential {E}quations}.
\newblock Mathematical Surveys and Monographs, 49. American Mathematical
  Society, Providence, RI, 1997.

\bibitem{s2005}
R.~E. Showalter.
\newblock Poroelastic filtration coupled to {S}tokes flow.
\newblock {\em Control theory of partial differential equations. Lect. Notes
  Pure Appl. Math., 242, Chapman \& Hall/CRC, Boca Raton, FL}, pages 229--241,
  2005.

\bibitem{s2010}
R.~E. Showalter.
\newblock Nonlinear degenerate evolution equations in mixed formulations.
\newblock {\em SIAM J. Math. Anal.}, 42(5):2114--2131, 2010.

\bibitem{stenberg1988}
R.~Stenberg.
\newblock A family of mixed finite elements for the elasticity problem.
\newblock {\em Numer. Math.}, 53(5):513--538, 1988.

\bibitem{Wen-He}
J.~Wen and Y.~He.
\newblock A strongly conservative finite element method for the coupled
  {S}tokes-{B}iot model.
\newblock {\em Comput. Math. Appl.}, 80(5):1421--1442, 2020.

\bibitem{wxy2012}
M.~F. Wheeler, G.~Xue, and I.~Yotov.
\newblock A multipoint flux mixed finite element method on distorted
  quadrilaterals and hexahedra.
\newblock {\em Numer. Math.}, 121(1):165--204, 2012.

\bibitem{wy2006}
M.~F. Wheeler and I.~Yotov.
\newblock A multipoint flux mixed finite element method.
\newblock {\em SIAM J. Numer. Anal.}, 44(5):2082--2106, 2006.

\bibitem{Yi-Biot-mixed}
S.-Y. Yi.
\newblock Convergence analysis of a new mixed finite element method for
  {B}iot's consolidation model.
\newblock {\em Numer. Methods Partial Differential Equations},
  30(4):1189--1210, 2014.

\bibitem{Yi-Biot-locking}
S.-Y. Yi.
\newblock A study of two modes of locking in poroelasticity.
\newblock {\em SIAM J. Numer. Anal.}, 55(4):1915--1936, 2017.

\end{thebibliography}

\end{document}